\documentclass[a4paper,10pt]{amsart}
\usepackage{enumerate, amsmath, amsfonts, amssymb, amsthm, mathtools, thmtools, wasysym, graphics, graphicx, frcursive, xparse, comment, multicol, ytableau,enumitem, xcolor, shuffle, hyperref}
\usepackage[all,cmtip]{xy}
\definecolor{darkblue}{rgb}{0.0,0,0.7}

\hypersetup{colorlinks=true, citecolor=darkblue, linkcolor=darkblue}
\usepackage{tikz}\usetikzlibrary{calc,through,backgrounds,shapes,matrix}
\usetikzlibrary{decorations.pathreplacing}

\newtheorem{theorem}{Theorem}[section]
\newtheorem{lemma}[theorem]{Lemma}
\newtheorem{proposition}[theorem]{Proposition}
\newtheorem{corollary}[theorem]{Corollary}

\newtheorem{definition}[theorem]{Definition}
\theoremstyle{remark}
\newtheorem{remark}[theorem]{Remark}
\newtheorem{problem}[theorem]{Problem}
\newtheorem{example}[theorem]{Example}
\numberwithin{equation}{section}

\newcommand{\wo}{w_{\circ}}

\newcommand{\Po}{\mathcal{P}}
\newcommand{\blank}{\phantom{2}}
\newcommand{\gap}{\hspace{1in} \\ \vspace{-.2in}}

\newcounter{x}
\newcounter{y}
\newcounter{z}

\newcommand\xaxis{210}
\newcommand\yaxis{-30}
\newcommand\zaxis{90}

\newcommand\topside[3]{
  \fill[fill=white, draw=black,shift={(\xaxis:#1)},shift={(\yaxis:#2)},
  shift={(\zaxis:#3)}] (0,0) -- (30:1) -- (0,1) --(150:1)--(0,0);
}

\newcommand\leftside[3]{
  \fill[fill=gray, draw=black,shift={(\xaxis:#1)},shift={(\yaxis:#2)},
  shift={(\zaxis:#3)}] (0,0) -- (0,-1) -- (210:1) --(150:1)--(0,0);
}

\newcommand\rightside[3]{
  \fill[fill=black, draw=black,shift={(\xaxis:#1)},shift={(\yaxis:#2)},
  shift={(\zaxis:#3)}] (0,0) -- (30:1) -- (-30:1) --(0,-1)--(0,0);
}

\newcommand\cube[3]{
  \topside{#1}{#2}{#3} \leftside{#1}{#2}{#3} \rightside{#1}{#2}{#3}
}

\newcommand\planepartition[1]{
 \setcounter{x}{-1}
  \foreach \a in {#1} {
    \addtocounter{x}{1}
    \setcounter{y}{-1}
    \foreach \b in \a {
      \addtocounter{y}{1}
      \setcounter{z}{-1}
      \foreach \c in {0,...,\b} {
        \addtocounter{z}{1}
      \ifthenelse{\c=0}{\setcounter{z}{-1},\addtocounter{y}{0}}{
        \cube{\value{x}}{\value{y}}{\value{z}}}
      }
    }
  }
}

\usepackage{todonotes}
\newcommand{\nathan}[1]{\todo[size=\tiny,color=green!30]{#1 \\ \hfill --- N.}}
\newcommand{\Nathan}[1]{\todo[size=\tiny,inline,color=green!30]{#1
      \\ \hfill --- N.}}

\newcommand{\Oliver}[1]{\todo[size=\tiny,inline,color=red!30]{#1
      \\ \hfill --- O.}}
\newcommand{\defn}[1]{\emph{\color{blue} #1}} 

\usepackage[nameinlink]{cleveref}
\usepackage[all]{xy}
\usepackage[T1]{fontenc}
\crefname{conjecture}{Conjecture}{Conjectures}
\crefname{problem}{Problem}{Problems}

\DeclareMathOperator{\Red}{{\sf Red}}

\newcommand{\well}{{j}}
\newcommand{\IT}{{\sf IT}}
\DeclareMathOperator{\SYT}{{\sf ST}}

\newcommand{\K}{{$K$}}
\newcommand{\X}{{X}}
\newcommand{\Y}{{Y}}
\newcommand{\Z}{{Z}}
\newlength{\cellsize} 
\newsavebox{\cell}
\sbox{\cell}{\begin{picture}(18,18)
\put(0,0){\line(1,0){18}}
\put(0,0){\line(0,1){18}}
\put(18,0){\line(0,1){18}}
\put(0,18){\line(1,0){18}}
\end{picture}}
\newcommand\cellify[1]{\def\thearg{#1}\def\nothing{}
\ifx\thearg\nothing
\vrule width0pt height\cellsize depth0pt\else
\hbox to 0pt{\usebox{\cell} \hss}\fi%
\vbox to \cellsize{
\vss
\hbox to \cellsize{\hss$#1$\hss}
\vss}}
\newcommand\tableau[1]{\vtop{\let\\\cr
\baselineskip -16000pt \lineskiplimit 16000pt \lineskip 0pt
\ialign{&\cellify{##}\cr#1\crcr}}}
\newcommand{\trap}[2]{\Phi^+_{B_{{#1},{#2}}}}

\newcommand{\pic}{\begin{tikzpicture}}
\newcommand{\epic}{\end{tikzpicture}}

\newcommand{\PP}{{\sf PP}}

\newcommand{\len}{{\sf len}}
\newcommand{\R}{\mathcal{C}}

\newcommand{\rect}{\mathsf{rect}}

\newcommand{\Kinf}{{\sf infusion}}
\newcommand{\fold}{{\sf fold}}

\newcommand{\woJ}{w_\circ^J}

\newcommand{\Gr}{{\sf Gr}}
\newcommand{\OG}{{\sf OG}}
\newcommand{\LG}{{\sf LG}}

\newcommand{\CC}{\mathbb{C}}
\newcommand{\QQ}{\mathbb{Q}}
\newcommand{\ZZ}{\mathbb{Z}}

\newcommand{\hgt}{{\sf ht}}

\newcommand{\GP}{{\sf G}/{\sf P}}

\newcommand{\T}{\mathsf{T}}
\newcommand{\U}{\mathsf{U}}
\newcommand{\x}{{\sf x}}
\newcommand{\w}{\mathsf{w}}
\newcommand{\vv}{{\sf v}}
\newcommand{\uu}{{\sf u}}
\newcommand{\p}{\ell}
\newcommand{\kk}{{m}}
\newcommand{\J}{{\mathcal{J}}}
\newcommand{\swap}{{\sf swap}}
\newcommand{\jdt}{{\sf jdt}}

\newcommand{\rhow}{{\sf read}}

\begin{document}

\author[Z.~Hamaker]{Zachary Hamaker}
\address[Z.~Hamaker]{Department of Mathematics, University of Michigan, 530 Church St. \\ Ann Arbor, MI U.S.A. 48103}
\email{zachary.hamaker@gmail.com}
\thanks{Z.H. was supported by the Institute for Mathematics and its Applications with funds provided by the National Science Foundation.}

\author[R.~Patrias]{Rebecca Patrias}
\address[R. Patrias]{LaCIM, Universit\'e du Qu\'ebec \`a Montr\'eal \\
Montr\'eal (Qu\'ebec), Canada}
\email{patriasr@lacim.ca}
\thanks{R.P. was partially supported by NSF Grant DMS-1148634.}
\author[O.Pechenik]{Oliver Pechenik}
\address[O.~Pechenik]{Department of Mathematics, Rutgers University, Piscataway, NJ 08854}
\email{oliver.pechenik@rutgers.edu}
\thanks{O.P. was supported by an Illinois Distinguished Fellowship and an NSF Graduate Research Fellowship.}
\author[N.~Williams]{Nathan Williams}
\address[N.~Williams]{Department of Mathematical Sciences, University of Texas at Dallas}
\email{nathan.f.williams@gmail.com}

\title{Doppelg\"angers: Bijections of Plane Partitions}

\begin{abstract}

We say two posets are {\it doppelg\"angers} if they have the same number of $P$-partitions of each height $k$.  We give a uniform framework for bijective proofs that posets are doppelg\"angers by synthesizing \K-theoretic Schubert calculus techniques of H.~Thomas and A.~Yong with M.~Haiman's rectification bijection and an observation of R.~Proctor. Geometrically, these bijections reflect the rational equivalence of certain subvarieties of minuscule flag manifolds.
As a special case, we provide the first bijective proof of a 1983 theorem of R.~Proctor---that plane partitions of height $k$ in a rectangle are equinumerous with plane partitions of height $k$ in a shifted trapezoid.


\end{abstract}

\maketitle

\section{Introduction}


\subsection{Doppelg\"angers}
\label{sec:coi_and_min}

Fix $\Po$ a finite partially-ordered set (\defn{poset}) with order relation $\succeq$.  For $\p \in \mathbb{Z}_{\geq 0}$, a \defn{$\Po$-partition of height $\p$} is a weakly order-preserving map $\T: \Po \to \{0,1,\ldots,\p\}$.  That is, if $p \succ q \in \Po$, then $\T(p) \geq \T(q)$.  We write $\PP^{[\p]}(\Po)$ for the set of all $\Po$-partitions of height $\p$ and---for ease of notation---we refer to $\Po$-partitions as \defn{plane partitions}.  

For $p\succ q \in \Po$, we say $p$ \defn{covers} $q$---written $p \gtrdot q$---if there is no $r \in \Po$ such that $p\succ r \succ q$.  We draw the \defn{Hasse diagram} of $\Po$ as a graph directed upwards in the page, with vertices indexed by the elements of $\Po$ and an edge from the vertex for $q$ to the vertex for $p$ when $p \gtrdot q$.



\begin{example}\rm
The posets with Hasse diagrams $\raisebox{-0.4\height}{\begin{tikzpicture}[scale=.2]
    \draw[thick] (0 cm,0) -- (1 cm,1);
        \draw[thick] (0 cm,0) -- (-1 cm,1);
        \draw[thick] (1 cm,1) -- (0 cm,2);
        \draw[thick] (-1 cm,1) -- (0 cm,2);
    \draw[thick,solid,fill=white] (0cm,0) circle (.5cm);
    \draw[thick,solid,fill=white] (1cm,1) circle (.5cm);
        \draw[thick,solid,fill=white] (-1cm,1) circle (.5cm);
        \draw[thick,solid,fill=white] (0cm,2) circle (.5cm);
  \end{tikzpicture}}$ and $\raisebox{-0.4\height}{\begin{tikzpicture}[scale=.2]
        \draw[thick] (-1 cm,3) -- (0 cm,2);
        \draw[thick] (-1 cm,1) -- (0 cm,2);
                \draw[thick] (0 cm,2) -- (1 cm,3);
    \draw[thick,solid,fill=white] (1cm,3) circle (.5cm);
    \draw[thick,solid,fill=white] (-1cm,3) circle (.5cm);
        \draw[thick,solid,fill=white] (-1cm,1) circle (.5cm);
        \draw[thick,solid,fill=white] (0cm,2) circle (.5cm);
  \end{tikzpicture}}$ each have six plane partitions of height $1$, as illustrated below.  Vertices in the Hasse diagrams are colored white if their image is $0$ and gray if their image is $1$.

\vspace{0.5em}
\begin{center}\tiny
\begin{tikzpicture}[scale=.3]
    \draw[thick] (0 cm,0) -- (1 cm,1);
        \draw[thick] (0 cm,0) -- (-1 cm,1);
        \draw[thick] (1 cm,1) -- (0 cm,2);
        \draw[thick] (-1 cm,1) -- (0 cm,2);
    \draw[thick,solid,fill=white] (0cm,0) circle (.5cm);
    \draw[thick,solid,fill=white] (1cm,1) circle (.5cm);
        \draw[thick,solid,fill=white] (-1cm,1) circle (.5cm);
        \draw[thick,solid,fill=white] (0cm,2) circle (.5cm);
  \end{tikzpicture}
\begin{tikzpicture}[scale=.3]
    \draw[thick] (0 cm,0) -- (1 cm,1);
        \draw[thick] (0 cm,0) -- (-1 cm,1);
        \draw[thick] (1 cm,1) -- (0 cm,2);
        \draw[thick] (-1 cm,1) -- (0 cm,2);
    \draw[thick,solid,fill=white] (0cm,0) circle (.5cm);
    \draw[thick,solid,fill=white] (1cm,1) circle (.5cm);
        \draw[thick,solid,fill=white] (-1cm,1) circle (.5cm);
        \draw[thick,solid,fill=gray] (0cm,2) circle (.5cm);
  \end{tikzpicture}
  \begin{tikzpicture}[scale=.3]
    \draw[thick] (0 cm,0) -- (1 cm,1);
        \draw[thick] (0 cm,0) -- (-1 cm,1);
        \draw[thick] (1 cm,1) -- (0 cm,2);
        \draw[thick] (-1 cm,1) -- (0 cm,2);
    \draw[thick,solid,fill=white] (0cm,0) circle (.5cm);
    \draw[thick,solid,fill=gray] (1cm,1) circle (.5cm);
        \draw[thick,solid,fill=white] (-1cm,1) circle (.5cm);
        \draw[thick,solid,fill=gray] (0cm,2) circle (.5cm);
  \end{tikzpicture}
\begin{tikzpicture}[scale=.3]
    \draw[thick] (0 cm,0) -- (1 cm,1);
        \draw[thick] (0 cm,0) -- (-1 cm,1);
        \draw[thick] (1 cm,1) -- (0 cm,2);
        \draw[thick] (-1 cm,1) -- (0 cm,2);
    \draw[thick,solid,fill=white] (0cm,0) circle (.5cm);
    \draw[thick,solid,fill=white] (1cm,1) circle (.5cm);
        \draw[thick,solid,fill=gray] (-1cm,1) circle (.5cm);
        \draw[thick,solid,fill=gray] (0cm,2) circle (.5cm);
  \end{tikzpicture}
\begin{tikzpicture}[scale=.3]
    \draw[thick] (0 cm,0) -- (1 cm,1);
        \draw[thick] (0 cm,0) -- (-1 cm,1);
        \draw[thick] (1 cm,1) -- (0 cm,2);
        \draw[thick] (-1 cm,1) -- (0 cm,2);
    \draw[thick,solid,fill=white] (0cm,0) circle (.5cm);
    \draw[thick,solid,fill=gray] (1cm,1) circle (.5cm);
        \draw[thick,solid,fill=gray] (-1cm,1) circle (.5cm);
        \draw[thick,solid,fill=gray] (0cm,2) circle (.5cm);
  \end{tikzpicture}
\begin{tikzpicture}[scale=.3]
    \draw[thick] (0 cm,0) -- (1 cm,1);
        \draw[thick] (0 cm,0) -- (-1 cm,1);
        \draw[thick] (1 cm,1) -- (0 cm,2);
        \draw[thick] (-1 cm,1) -- (0 cm,2);
    \draw[thick,solid,fill=gray] (0cm,0) circle (.5cm);
    \draw[thick,solid,fill=gray] (1cm,1) circle (.5cm);
        \draw[thick,solid,fill=gray] (-1cm,1) circle (.5cm);
        \draw[thick,solid,fill=gray] (0cm,2) circle (.5cm);
  \end{tikzpicture}
\vspace{1em}

    \begin{tikzpicture}[scale=.3]
  \draw[thick] (-1 cm,3) -- (0 cm,2);
        \draw[thick] (-1 cm,1) -- (0 cm,2);
                \draw[thick] (0 cm,2) -- (1 cm,3);
    \draw[thick,solid,fill=white] (1cm,3) circle (.5cm);
    \draw[thick,solid,fill=white] (-1cm,3) circle (.5cm);
        \draw[thick,solid,fill=white] (-1cm,1) circle (.5cm);
        \draw[thick,solid,fill=white] (0cm,2) circle (.5cm);
  \end{tikzpicture}
    \begin{tikzpicture}[scale=.3]
  \draw[thick] (-1 cm,3) -- (0 cm,2);
        \draw[thick] (-1 cm,1) -- (0 cm,2);
                \draw[thick] (0 cm,2) -- (1 cm,3);
    \draw[thick,solid,fill=gray] (1cm,3) circle (.5cm);
    \draw[thick,solid,fill=gray] (-1cm,3) circle (.5cm);
        \draw[thick,solid,fill=white] (-1cm,1) circle (.5cm);
        \draw[thick,solid,fill=white] (0cm,2) circle (.5cm);
  \end{tikzpicture}
  \begin{tikzpicture}[scale=.3]
    \draw[thick] (-1 cm,3) -- (0 cm,2);
        \draw[thick] (-1 cm,1) -- (0 cm,2);
                \draw[thick] (0 cm,2) -- (1 cm,3);
    \draw[thick,solid,fill=white] (1cm,3) circle (.5cm);
    \draw[thick,solid,fill=gray] (-1cm,3) circle (.5cm);
        \draw[thick,solid,fill=white] (-1cm,1) circle (.5cm);
        \draw[thick,solid,fill=white] (0cm,2) circle (.5cm);
  \end{tikzpicture}
  \begin{tikzpicture}[scale=.3]
  \draw[thick] (-1 cm,3) -- (0 cm,2);
        \draw[thick] (-1 cm,1) -- (0 cm,2);
                \draw[thick] (0 cm,2) -- (1 cm,3);
    \draw[thick,solid,fill=gray] (1cm,3) circle (.5cm);
    \draw[thick,solid,fill=white] (-1cm,3) circle (.5cm);
        \draw[thick,solid,fill=white] (-1cm,1) circle (.5cm);
        \draw[thick,solid,fill=white] (0cm,2) circle (.5cm);
  \end{tikzpicture}
\begin{tikzpicture}[scale=.3]
   \draw[thick] (-1 cm,3) -- (0 cm,2);
        \draw[thick] (-1 cm,1) -- (0 cm,2);
                \draw[thick] (0 cm,2) -- (1 cm,3);
    \draw[thick,solid,fill=gray] (1cm,3) circle (.5cm);
    \draw[thick,solid,fill=gray] (-1cm,3) circle (.5cm);
        \draw[thick,solid,fill=white] (-1cm,1) circle (.5cm);
        \draw[thick,solid,fill=gray] (0cm,2) circle (.5cm);
  \end{tikzpicture}
\begin{tikzpicture}[scale=.3]
  \draw[thick] (-1 cm,3) -- (0 cm,2);
        \draw[thick] (-1 cm,1) -- (0 cm,2);
                \draw[thick] (0 cm,2) -- (1 cm,3);
    \draw[thick,solid,fill=gray] (1cm,3) circle (.5cm);
    \draw[thick,solid,fill=gray] (-1cm,3) circle (.5cm);
        \draw[thick,solid,fill=gray] (-1cm,1) circle (.5cm);
        \draw[thick,solid,fill=gray] (0cm,2) circle (.5cm);
  \end{tikzpicture}
\end{center}
\normalsize
\label{ex:ex1}
\end{example}.

The number of plane partitions of height $\p$ in a poset $\Po$ with $n$ elements is counted by its \defn{order polynomial} $\left|\PP^{[\p]}(\Po)\right|$.  As its name suggests, the order polynomial is a polynomial in $\p$ of degree $n$~\cite[Section 13]{stanley1972ordered}.  One can check that both posets from~\Cref{ex:ex1} have the same order polynomial: \[\left|\PP^{[\p]}\left(\raisebox{-0.4\height}{\begin{tikzpicture}[scale=.2]
    \draw[thick] (0 cm,0) -- (1 cm,1);
        \draw[thick] (0 cm,0) -- (-1 cm,1);
        \draw[thick] (1 cm,1) -- (0 cm,2);
        \draw[thick] (-1 cm,1) -- (0 cm,2);
    \draw[thick,solid,fill=white] (0cm,0) circle (.5cm);
    \draw[thick,solid,fill=white] (1cm,1) circle (.5cm);
        \draw[thick,solid,fill=white] (-1cm,1) circle (.5cm);
        \draw[thick,solid,fill=white] (0cm,2) circle (.5cm);
 \end{tikzpicture}}\right)\right|=\left|\PP^{[\p]}\left(\raisebox{-0.4\height}{\begin{tikzpicture}[scale=.2]
        \draw[thick] (-1 cm,3) -- (0 cm,2);
        \draw[thick] (-1 cm,1) -- (0 cm,2);
                \draw[thick] (0 cm,2) -- (1 cm,3);
    \draw[thick,solid,fill=white] (1cm,3) circle (.5cm);
    \draw[thick,solid,fill=white] (-1cm,3) circle (.5cm);
        \draw[thick,solid,fill=white] (-1cm,1) circle (.5cm);
        \draw[thick,solid,fill=white] (0cm,2) circle (.5cm);
  \end{tikzpicture}}\right)\right|=\frac{1}{12}(\p+1)(\p+2)^2(\p + 3).\]

This example motivates the following definition.

\begin{definition}
Let $\Po, \mathcal{Q}$ be two finite posets.  We say that $\Po$ and $\mathcal{Q}$ are \defn{doppelg\"angers} if they have the same order polynomial.
\label{def:doppel}
\end{definition}

Here are some trivial examples: any poset is its own doppelg\"anger, as are any poset and its \defn{dual} (obtained by reversing the order relation, which flips the Hasse diagram upside down).  We can explain the doppelg\"angers in~\Cref{ex:ex1} with the simple observation that the poset obtained by adjoining a new \emph{minimal} element to a poset $\Po$ is a doppelg\"anger with the poset obtained by adjoining a new \emph{maximal} element to $\Po$.

\subsection{Minuscule Doppelg\"angers}

We use the combinatorics of the \K-theoretic Schubert calculus of minuscule flag varieties to bijectively establish certain pairs of posets $(\Lambda_X,\Phi^+_Y)$ as doppelg\"angers in \Cref{thm:main_thm1}.   As summarized in~\Cref{sec:philo}, our philosophy is that---given a ring with basis indexed by combinatorial objects and a combinatorial rule for computing structure coefficients---multiplicity-free products in that ring may be equivalently stated as bijections.   In addition to providing two infinite families of doppelg\"angers, we also exhibit one non-trivial exceptional pair.   The first poset $\Lambda_X$ in each pair is a poset that describes the Schubert cell decomposition of a minuscule flag variety (a \defn{minuscule poset}, defined in~\Cref{sec:minuscule}); the second poset $\Phi^+_Y$ is a poset on the positive roots of another, {\it a priori} unrelated, root system (defined in~\Cref{sec:coincidental}).  To streamline our exposition, we defer the precise definitions and root-theoretic meanings of these posets until~\Cref{sec:minuscule,sec:coincidental}, simply defining the posets for now by their Hasse diagrams in~\Cref{fig:names}.


%

\begin{figure}[htbp]
\begin{center}
\begin{tabular}{|c|cc|cc|} \hline 
	Label &
	Poset Name & Hasse Diagram & Hasse Diagram  & Poset Name \\  \hline \hline  
(B) &
 $\Lambda_{\Gr(k,n)}$ & \raisebox{-0.5\height}{\begin{tikzpicture}[scale=.35]
	\draw[thick] (1 cm,3) -- (2 cm,4) -- (1 cm,5) -- (0cm,4)--(1cm,3);
    \draw[thick] (3,1)--(4,2);
    \draw[thick] (3,7)--(4,6);
    \draw[thick,dotted] (1,3)--(3,1);
    \draw[thick,dotted] (2,4)--(4,2);
    \draw[thick,dotted] (1,5)--(4,2);
    \draw[thick,dotted] (4,6)--(6,4);
    \draw[thick,dotted] (4,2)--(6,4);
    \draw[thick,dotted] (5,3)--(2,6);
    \draw[thick,dotted] (2 cm,2) -- (5 cm,5);
    \draw[thick,dotted] (2 cm,4) -- (4 cm,6);
    \draw[thick,dotted] (1 cm,5) -- (3 cm,7);

    \draw[thick,solid,fill=white] (0cm,4) circle (.5cm) ;
    \draw[thick,solid,fill=white] (1cm,3) circle (.5cm) ;
    \draw[thick,solid,fill=white] (3cm,7) circle (.5cm) ;
    \draw[thick,solid,fill=white] (3cm,1) circle (.5cm) ;
    \draw[thick,solid,fill=white] (2cm,4) circle (.5cm) ;
    \draw[thick,solid,fill=white] (1cm,5) circle (.5cm) ;
    \draw[thick,solid,fill=white] (4cm,2) circle (.5cm) ;
    \draw[thick,solid,fill=white] (4cm,6) circle (.5cm) ;
    \draw[thick,solid,fill=white] (6cm,4) circle (.5cm) ;
    \draw [decorate,decoration={brace,amplitude=10pt}] (7,3.5) -- (3.5,0) node [midway, below,yshift=-0.8em,xshift=1.2em] {$n-k$};
    \draw [decorate,decoration={brace,amplitude=10pt}] (2.5,0) -- (-1,3.5) node [midway, below,yshift=-0.8em,xshift=-1.2em] {$k$};
\end{tikzpicture}} & 
\raisebox{-0.5\height}{\begin{tikzpicture}[scale=.35]
	\draw[thick] (4 cm,2) -- (5 cm,3)--(6,2);
    \draw[thick,dotted] (0 cm,4) -- (2 cm,6);
	\draw[thick] (0 cm,6) -- (1 cm,7);
    \draw[thick] (2 cm,6) -- (0 cm,8);
    \draw[thick,dotted] (5,3) -- (2,6);
    \draw[thick,dotted] (4,2) -- (0,6);
    \draw[thick,dotted] (3,5) -- (1,3);
    \draw[thick,dotted] (2,2) -- (0,4);
    \draw[thick,dotted] (1,5) -- (0,4);
    \draw[thick,dotted] (2,2) -- (4,4);
 
    \draw[thick,solid,fill=white] (2cm,2) circle (.5cm) ;
    \draw[thick,solid,fill=white] (0cm,6) circle (.5cm) ;
    \draw[thick,solid,fill=white] (0cm,8) circle (.5cm) ;
    \draw[thick,solid,fill=white] (0cm,4) circle (.5cm) ;
    \draw[thick,solid,fill=white] (1cm,7) circle (.5cm) ;
    \draw[thick,solid,fill=white] (4cm,2) circle (.5cm) ;
    \draw[thick,solid,fill=white] (2cm,6) circle (.5cm) ;
    \draw[thick,solid,fill=white] (5cm,3) circle (.5cm) ;
    \draw[thick,solid,fill=white] (6cm,2) circle (.5cm) ;
    \draw [decorate,decoration={brace,amplitude=10pt}] (1.5,1) -- (-1,3.5) node [midway, below,yshift=-0.8em,xshift=-1.2em] {\scalebox{0.8}{$n-2k+1$}};
    \draw [decorate,decoration={brace,amplitude=10pt}] (0.5,9) -- (7,2.5) node [midway, above,yshift=0.8em,xshift=1.2em] {$n-1$};
\end{tikzpicture}} & $\trap{k}{n}$  \\ \hline
(H) &
 $\Lambda_{\OG(6,12)}$ &  \raisebox{-0.5\height}{\begin{tikzpicture}[scale=.25]
    \draw[thick] (3 cm,5) -- (5 cm,7);
	\draw[thick] (3 cm,7) -- (5 cm,9);
    \draw[thick] (2 cm,8) -- (5 cm,11);
    \draw[thick] (1 cm,9) -- (4 cm,12);
    \draw[thick] (3 cm,13) -- (5 cm,11);
    \draw[thick] (3cm,11) -- (5 cm,9);
    \draw[thick] (2cm,10) -- (5cm,7);
    \draw[thick] (1cm,9) -- (4cm,6);
    \draw[thick,solid,fill=white] (3cm,5) circle (.5cm) ;
    \draw[thick,solid,fill=white] (4cm,6) circle (.5cm) ;
    \draw[thick,solid,fill=white] (5cm,7) circle (.5cm) ;
    \draw[thick,solid,fill=white] (3cm,7) circle (.5cm) ;
    \draw[thick,solid,fill=white] (4cm,8) circle (.5cm) ;
    \draw[thick,solid,fill=white] (5cm,9) circle (.5cm) ;
    \draw[thick,solid,fill=white] (2cm,8) circle (.5cm) ;
    \draw[thick,solid,fill=white] (3cm,9) circle (.5cm) ;
    \draw[thick,solid,fill=white] (4cm,10) circle (.5cm) ;
    \draw[thick,solid,fill=white] (5cm,11) circle (.5cm) ;
    \draw[thick,solid,fill=white] (1cm,9) circle (.5cm) ;
    \draw[thick,solid,fill=white] (2cm,10) circle (.5cm) ;
    \draw[thick,solid,fill=white] (3cm,11) circle (.5cm) ;
    \draw[thick,solid,fill=white] (4cm,12) circle (.5cm) ;
    \draw[thick,solid,fill=white] (3cm,13) circle (.5cm) ;
\end{tikzpicture}} & \raisebox{-0.5\height}{\begin{tikzpicture}[scale=.25]
    \draw[thick] (3 cm,9) -- (6 cm,12);
    \draw[thick] (1 cm,9) -- (5 cm,13);
    \draw[thick] (3 cm,13) -- (4 cm,14);
    \draw[thick] (1 cm,17) -- (6 cm,12);
    \draw[thick] (3 cm,13) -- (5 cm,11);
    \draw[thick] (3cm,11) -- (5 cm,9);
    \draw[thick] (2cm,10) -- (3cm,9);
    \draw[thick,solid,fill=white] (5cm,9) circle (.5cm) ;
    \draw[thick,solid,fill=white] (3cm,9) circle (.5cm) ;
    \draw[thick,solid,fill=white] (4cm,10) circle (.5cm) ;
    \draw[thick,solid,fill=white] (5cm,11) circle (.5cm) ;
    \draw[thick,solid,fill=white] (6cm,12) circle (.5cm) ;
    \draw[thick,solid,fill=white] (1cm,9) circle (.5cm) ;
    \draw[thick,solid,fill=white] (2cm,10) circle (.5cm) ;
    \draw[thick,solid,fill=white] (3cm,11) circle (.5cm) ;
    \draw[thick,solid,fill=white] (4cm,12) circle (.5cm) ;
    \draw[thick,solid,fill=white] (5cm,13) circle (.5cm) ;
    \draw[thick,solid,fill=white] (3cm,13) circle (.5cm) ;
    \draw[thick,solid,fill=white] (4cm,14) circle (.5cm) ;
    \draw[thick,solid,fill=white] (3cm,15) circle (.5cm) ;
    \draw[thick,solid,fill=white] (2cm,16) circle (.5cm) ;
    \draw[thick,solid,fill=white] (1cm,17) circle (.5cm) ;
\end{tikzpicture}} & $\Phi^+_{H_3}$  \\ \hline
(I) & 
$\Lambda_{\QQ^{2n}}$ & \raisebox{-0.5\height}{\begin{tikzpicture}[scale=.35]
	\draw[thick,dotted] (1 cm,1) -- (3 cm,3);
    \draw[thick] (3 cm,3) -- (4 cm,4);
    \draw[thick] (2 cm,4) -- (3 cm,5);
    \draw[thick,dotted] (3 cm,5) -- (5 cm,7);
    \draw[thick] (2 cm,4) -- (3 cm,3);
    \draw[thick] (3 cm,5) -- (4 cm,4);
    \draw[thick,solid,fill=white] (1cm,1) circle (.5cm) ;
    \draw[thick,solid,fill=white] (5cm,7) circle (.5cm) ;
    \draw[thick,solid,fill=white] (3cm,3) circle (.5cm) ;
    \draw[thick,solid,fill=white] (4cm,4) circle (.5cm) ;
    \draw[thick,solid,fill=white] (2cm,4) circle (.5cm) ;
    \draw[thick,solid,fill=white] (3cm,5) circle (.5cm) ;
    \draw [decorate,decoration={brace,amplitude=10pt}] (5,3.5) -- (1.5,0) node [midway, below,yshift=-0.8em,xshift=1.2em] {$n$};
    \draw [decorate,decoration={brace,amplitude=10pt}] (1,4.5) -- (4.5,8) node [midway, above,yshift=0.7em,xshift=-1.3em] {$n$};
\end{tikzpicture}} &
\raisebox{-0.5\height}{\begin{tikzpicture}[scale=.35]
    \draw[thick] (2 cm,4) -- (4 cm,6);
    \draw[thick,dotted] (4 cm,6) -- (7 cm,9);
    \draw[thick] (3 cm,5) -- (4 cm,4);
    \draw[thick,solid,fill=white] (4cm,4) circle (.5cm) ;
    \draw[thick,solid,fill=white] (2cm,4) circle (.5cm) ;
    \draw[thick,solid,fill=white] (3cm,5) circle (.5cm) ;
    \draw[thick,solid,fill=white] (4cm,6) circle (.5cm) ;
    \draw[thick,solid,fill=white] (7cm,9) circle (.5cm) ;
    \draw [decorate,decoration={brace,amplitude=10pt}] (1,4.5) -- (6.5,10) node [midway, above,yshift=0.7em,xshift=-1.3em] {\scalebox{0.8}{$2n-1$}};
\end{tikzpicture}} & $\Phi^+_{I_2(2n)}$ \\ \hline
\end{tabular}
\end{center}
\caption{The names and Hasse diagrams of the six posets considered in~\Cref{thm:main_thm1}.  The first and last rows are infinite families, while the middle row is a single exceptional example.}
\label{fig:names}
\end{figure}

\begin{theorem}
\label{thm:main_thm1}
As defined in~\Cref{fig:names}, fix \[(\X,\Y) \in \Big\{\big(\Gr(k,n),B_{k,n}\big),\big(\OG(6,12),H_3\big),\big(\QQ^{2n},I_2(2n)\big) \Big\}.\]  Then, $\Lambda_\X$ and $\Phi^+_\Y$ are doppelg\"angers.
Moreover, there is an explicit, type-uniform bijection
\[ \PP^{[\p]}\left(\Lambda_\X\right) \simeq \PP^{[\p]}\left(\Phi^+_\Y\right). \]

\end{theorem}

Our arguments are usefully interpreted as statements about rational equivalence of certain generalized Schubert and Richardson subvarieties of minuscule flag varieties. That is, each of the bijections of \Cref{thm:main_thm1} corresponds to the fact that a certain Richardson variety represents the same element of the Chow ring as a certain Schubert variety. 
Nonetheless, our statements and arguments are purely combinatorial and do not logically depend on these geometric considerations.  In particular, the key proofs should be accessible to a reader who knows no geometry---such readers may skip reading \Cref{sec:definitions,sec:minuscule,sec:coincidental,sec:embedding,sec:ring}, referring to them only for notation.  

Although our bijections are type-uniform, our proofs are only partially so.  
It is an open problem to identify other interesting examples of doppelg\"angers, or to classify them in general.  

\begin{remark}
The equality \begin{equation}\left|\PP^{[\p]}\left(\Lambda_{\Gr(k,n)}\right)\right| = \left|\PP^{[\p]}\big(\trap{k}{n}\big)\right|\label{eq:rect_trap}\end{equation} was first proven nonbijectively by R.~Proctor in~\cite{proctor1983shifted} using a branching rule due to R.~King from the Lie algebra inclusion $\mathfrak{sp}_{2n}(\CC) \hookrightarrow \mathfrak{sl}_{2n}(\CC)$\cite{king1975branching,littlewood1950theory}.  Indeed, R.~Proctor remarks that ``the question of a combinatorial correspondence for [\Cref{eq:rect_trap}] seems to be a complete mystery.''

For the case $\p=1$ of~\Cref{eq:rect_trap}, J.~Stembridge produced a jeu-de-taquin bijection~\cite{stembridge1986trapezoidal}, while V.~Reiner gave an argument using type $B$ noncrossing partitions~\cite{reiner1997non}. For $\p\leq 2$, S.~Elizalde gave a bijection in the language of pairs of lattice paths~\cite{elizalde2015bijections}.   No bijection was previously known for $\p >2$, and the restriction of our bijection is not immediately equivalent to any of these known special cases for $\p\leq 2$.
\end{remark}

\begin{remark}
The other cases of \Cref{thm:main_thm1} are easy to establish directly~\cite[Theorems 3.1.24 and 3.1.27]{williams13cataland}; in these cases, we provide the first bijections and the first geometric interpretations. 
\end{remark}


%

\subsection{\K-Theoretic Schubert Calculus}

To interpret the examples of~\Cref{fig:names} as statements in Schubert calculus, the key observation we make is that for \[(\X,\Y) \in \Big\{\big(\Gr(k,n),B_{k,n}\big),\big(\OG(6,12),H_3\big),\big(\QQ^{2n},I_2(2n)\big) \Big\},\] both the posets $\Lambda_\X$ and $\Phi^+_\Y$ simultaneously occur as subposets of a larger, ambient, minuscule poset $\Lambda_\Z$.  
We denote these two embeddings by
\begin{equation}\label{eq:uvw}
\Theta(\Lambda_{\X}) \subseteq \Lambda_{\Z} \hspace{3em} \text{ and } \hspace{3em} \chi(\Phi^+_{\Y})  \subseteq \Lambda_{\Z}.
\end{equation}

As we explain in~\Cref{sec:embedding}, the embedding of $\Lambda_\X$ into $\Lambda_\Z$ comes from an embedding of one root system in another, while the duals of $\Phi^+_\Y$ appear as order ideals in $\Lambda_\Z$ and are posets on the positive roots of yet a third root system.\footnote{We thank Robert Proctor for pointing out this fact for $H_3$, long before we had any idea how to make sense of it.}  \Cref{fig:triples} lists the triples $(\Lambda_\X,\Lambda_\Z,\Phi^+_\Y)$ corresponding to the doppelg\"angers in~\Cref{fig:names}, while specific examples of these poset embeddings are illustrated in~\Cref{fig:coincidental_top_half_minuscules}.   In~\Cref{fig:triples}, the symbols $\X$ and $\Z$ specify a minuscule flag variety (\Cref{sec:minuscule}), while the symbol $\Y$ is a bookkeeping device that specifies a Coxeter-Cartan type (\Cref{sec:coincidental}).

\begin{figure}[htbp]
\[\begin{array}{c|ccccc} 
\text{Label} &
 \Lambda_\X & \hookrightarrow & \Lambda_\Z & \xmapsto{\rect} & \Phi^+_\Y \\ \hline
\mathrm{(B)} & 
\Lambda_{\Gr(k,n)}&   & \Lambda_{\mathsf{OG}(n,2n)}  & & \Phi^+_{B_{k,n}}\\ 
\mathrm{(H)} &
\Lambda_{\OG(6,12)}& & \Lambda_{\mathsf{G}_{\omega}(\mathbb{O}^3,\mathbb{O}^6)} &  & \Phi^+_{H_3}\\ 
\mathrm{(I)} &
 \Lambda_{\QQ^{2n}}&  & \Lambda_{\mathbb{Q}^{4n-2}} & & \Phi^+_{I_2(2n)} \\ 
\end{array}\]
\caption{Triples $(\X,\Y,\Z)$, where $\Lambda_\X$ and $\Phi^+(\Y)$ are doppelg\"angers.  The symbols $\X$ and $\Z$ specify minuscule flag varieties, while $\Y$ relates to a Coxeter-Cartan type.  As illustrated in~\Cref{fig:coincidental_top_half_minuscules}, both the poset $\Lambda_{\X}$ and the poset $\Phi^+_{\Y}$ embed in the ambient minuscule poset $\Lambda_{\Z}$; a tableau of shape $\Phi^+_{\Y}$ is then obtained from one of shape $\Lambda_\X$ by \K-rectification.}  
\label{fig:triples}
\end{figure}





\begin{figure}[htbp]
\[\begin{array}{ccc}
\raisebox{-0.5\height}{\begin{tikzpicture}[scale=.35]
	\draw[thick] (0 cm,0) -- (4 cm,4);
	\draw[thick] (0 cm,2) -- (3 cm,5);
    \draw[thick] (0 cm,4) -- (2 cm,6);
	\draw[thick] (0 cm,6) -- (1 cm,7);
    \draw[thick] (4 cm,4) -- (0 cm,8);
	\draw[thick] (3 cm,3) -- (0 cm,6);
	\draw[thick] (2 cm,2) -- (0 cm,4);
	\draw[thick] (1 cm,1) -- (0 cm,2);
    \draw[thick,solid,fill=gray] (0cm,0) circle (.5cm) ;
    \draw[thick,solid,fill=gray] (0cm,2) circle (.5cm) ;
    \draw[ultra thick,solid,fill=gray] (0cm,4) circle (.5cm) ;
    \draw[thick,solid,fill=white] (0cm,6) circle (.5cm) ;
    \draw[thick,solid,fill=white] (0cm,8) circle (.5cm) ;
    \draw[thick,solid,fill=gray] (1cm,1) circle (.5cm) ;
    \draw[ultra thick,solid,fill=gray] (1cm,3) circle (.5cm) ;
    \draw[ultra thick,solid,fill=white] (1cm,5) circle (.5cm) ;
    \draw[thick,solid,fill=white] (1cm,7) circle (.5cm) ;
    \draw[ultra thick,solid,fill=gray] (2cm,2) circle (.5cm) ;
    \draw[ultra thick,solid,fill=gray] (2cm,4) circle (.5cm) ;
    \draw[ultra thick,solid,fill=white] (2cm,6) circle (.5cm) ;
    \draw[ultra thick,solid,fill=gray] (3cm,3) circle (.5cm) ;
    \draw[ultra thick,solid,fill=white] (3cm,5) circle (.5cm) ;
    \draw[ultra thick,solid,fill=gray] (4cm,4) circle (.5cm) ;
\end{tikzpicture}}&
\raisebox{-0.5\height}{\begin{tikzpicture}[scale=.35]
	\draw[thick] (1 cm,1) -- (6 cm,6);
    \draw[thick] (3 cm,5) -- (5 cm,7);
	\draw[thick] (3 cm,7) -- (5 cm,9);
    \draw[thick] (2 cm,8) -- (6 cm,12);
    \draw[thick] (1 cm,9) -- (5 cm,13);
    \draw[thick] (3 cm,13) -- (4 cm,14);
    \draw[thick] (1 cm,17) -- (6 cm,12);
    \draw[thick] (3 cm,13) -- (5 cm,11);
    \draw[thick] (3cm,11) -- (5 cm,9);
    \draw[thick] (2cm,10) -- (6cm,6);
    \draw[thick] (1cm,9) -- (5cm,5);
    \draw[thick] (3cm,5) -- (4cm,4);
    \draw[thick,solid,fill=gray] (1cm,1) circle (.5cm) ;
    \draw[thick,solid,fill=gray] (2cm,2) circle (.5cm) ;
    \draw[thick,solid,fill=gray] (3cm,3) circle (.5cm) ;
    \draw[thick,solid,fill=gray] (4cm,4) circle (.5cm) ;
    \draw[thick,solid,fill=gray] (5cm,5) circle (.5cm) ;
    \draw[thick,solid,fill=gray] (6cm,6) circle (.5cm) ;
    \draw[ultra thick,solid,fill=gray] (3cm,5) circle (.5cm) ;
    \draw[ultra thick,solid,fill=gray] (4cm,6) circle (.5cm) ;
    \draw[ultra thick,solid,fill=gray] (5cm,7) circle (.5cm) ;
    \draw[ultra thick,solid,fill=gray] (3cm,7) circle (.5cm) ;
    \draw[ultra thick,solid,fill=gray] (4cm,8) circle (.5cm) ;
    \draw[ultra thick,solid,fill=gray] (5cm,9) circle (.5cm) ;
    \draw[ultra thick,solid,fill=gray] (2cm,8) circle (.5cm) ;
    \draw[ultra thick,solid,fill=gray] (3cm,9) circle (.5cm) ;
    \draw[ultra thick,solid,fill=white] (4cm,10) circle (.5cm) ;
    \draw[ultra thick,solid,fill=white] (5cm,11) circle (.5cm) ;
    \draw[thick,solid,fill=white] (6cm,12) circle (.5cm) ;
    \draw[ultra thick,solid,fill=gray] (1cm,9) circle (.5cm) ;
    \draw[ultra thick,solid,fill=white] (2cm,10) circle (.5cm) ;
    \draw[ultra thick,solid,fill=white] (3cm,11) circle (.5cm) ;
    \draw[ultra thick,solid,fill=white] (4cm,12) circle (.5cm) ;
    \draw[thick,solid,fill=white] (5cm,13) circle (.5cm) ;
    \draw[ultra thick,solid,fill=white] (3cm,13) circle (.5cm) ;
    \draw[thick,solid,fill=white] (4cm,14) circle (.5cm) ;
    \draw[thick,solid,fill=white] (3cm,15) circle (.5cm) ;
    \draw[thick,solid,fill=white] (2cm,16) circle (.5cm) ;
    \draw[thick,solid,fill=white] (1cm,17) circle (.5cm) ;
\end{tikzpicture}}&
\raisebox{-0.5\height}{\begin{tikzpicture}[scale=.35]
	\draw[thick] (-2 cm,-2) -- (4 cm,4);
    \draw[thick] (2 cm,4) -- (8 cm,10);
    \draw[thick] (2 cm,4) -- (3 cm,3);
    \draw[thick] (3 cm,5) -- (4 cm,4);
    \draw[thick,solid,fill=gray] (-1cm,-1) circle (.5cm) ;
   \draw[thick,solid,fill=gray] (-2cm,-2) circle (.5cm) ;
	\draw[thick,solid,fill=gray] (0cm,0) circle (.5cm) ;
    \draw[ultra thick,solid,fill=gray] (1cm,1) circle (.5cm) ;
    \draw[ultra thick,solid,fill=gray] (2cm,2) circle (.5cm) ;
    \draw[ultra thick,solid,fill=gray] (3cm,3) circle (.5cm) ;
    \draw[ultra thick,solid,fill=gray] (4cm,4) circle (.5cm) ;
    \draw[ultra thick,solid,fill=gray] (2cm,4) circle (.5cm) ;
    \draw[ultra thick,solid,fill=white] (3cm,5) circle (.5cm) ;
    \draw[ultra thick,solid,fill=white] (4cm,6) circle (.5cm) ;
    \draw[ultra thick,solid,fill=white] (5cm,7) circle (.5cm) ;
    \draw[thick,solid,fill=white] (6cm,8) circle (.5cm) ;
    \draw[thick,solid,fill=white] (7cm,9) circle (.5cm) ;
\draw[thick,solid,fill=white] (8cm,10) circle (.5cm) ;
\end{tikzpicture}} \\
\mathrm{(B)} & \mathrm{(H)} & \mathrm{(I)}
\end{array}\]

\caption{Examples of simultaneous embedding of the doppelg\"anger pairs $\Lambda_{\X}$ and (duals of) $\Phi^+_{\Y}$ in the minuscule posets $\Lambda_{\Z}$.  The vertices with thick borders correspond to $\Theta(\Lambda_{\X})$, while the gray vertices represent $\chi(\Phi^+_{\Y})$.}
\label{fig:coincidental_top_half_minuscules}
\end{figure}
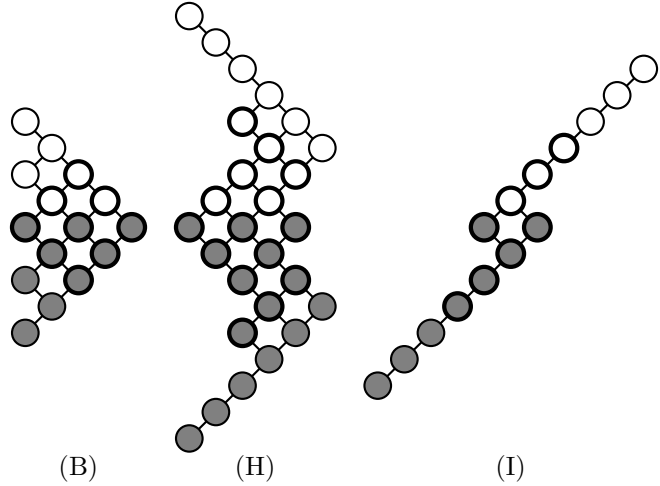

%






These embeddings allow us to prove \Cref{thm:main_thm1} using a combinatorial model of the structure coefficients of \K-theoretic Schubert calculus on minuscule varieties.  Our~\Cref{thm:gen_main_thm} and \Cref{thm:gen_main_thm_Ring} give a much more general bijective framework for proving similar results.

\medskip

In more detail, the \defn{increasing tableaux of shape $\Po$ and height $\kk$}---written $\IT^{[\kk]}(\Po)$---are the strictly order-preserving maps from $\Po \to \{1,2,\ldots,\kk\}$.  For a ranked poset $\Po$ whose maximal chains are all of the same length $\hgt(\Po)$ (and, in particular, for all the posets of \Cref{fig:names}), there is a simple bijection \[\IT^{[\kk]}(\Po) \simeq \PP^{[\p]}(\Po),\] where $\kk=\p+\hgt(\Po)$ (see \Cref{prop:bij_PP_IT}).   The significant advantage that increasing tableaux enjoy over plane partitions is a well-developed theory of \K-theoretic jeu-de-taquin that is particularly well-behaved on minuscule posets.  This theory, which we review in~\Cref{sec:ring,sec:structure_coeffs}, was introduced for geometric purposes by H.~Thomas and A.~Yong in~\cite{thomas2009jeu} and was further developed by A.~Buch, E.~Clifford, H.~Thomas, M.~Samuel, and A.~Yong in~\cite{clifford2014k,buch2014k}.  We  only need its combinatorial features to prove \Cref{thm:main_thm1}: using the identification of increasing tableaux with plane partitions and the embeddings in~\Cref{eq:uvw}, the bijections of~\Cref{thm:main_thm1} are uniformly written \begin{align*} \IT^{[\kk]}(\Lambda_{X}) &\to \IT^{[\kk]}(\Phi^+_{\Y})   \\ \T &\mapsto \chi^{-1}\left(\rect\left(\Theta(\T)\right)\right),\end{align*} where $\rect$ is an operation called \emph{\K-rectification} defined using \K-theoretic jeu-de-taquin.

\begin{example}
We continue \Cref{ex:ex1} to give the bijection of~\Cref{thm:main_thm1} in the case \[\PP^{[1]}\left(\Lambda_{\Gr(2,4)}\right) \simeq \PP^{[1]}\left(\Phi^+_{B_{2,4}}\right).\]  After embedding both posets in $\Lambda_{\mathsf{OG}(4,8)}$, the bijection is given by the jeu-de-taquin computation illustrated below (for more details, see~\Cref{sec:ktheorycomb}).  A larger example with additional details is given in~\Cref{sec:Rectangles_and_Trapezoids}, and the boxed fillings in the middle two rows are explained in~\Cref{rem:previous_work}.  

\begin{align*}\tiny
\raisebox{-0.5\height}{
}$
\label{ex:ex3}
\end{example}








\subsection{Example: Rectangles and Shifted Trapezoids}\label{sec:Rectangles_and_Trapezoids}

For the impatient reader, in this section we give a detailed example of the bijection of~\Cref{thm:main_thm1} \[\PP^{[4]}\left(\Lambda_{\Gr(4,8)}\right) \simeq \PP^{[4]}\left(\Phi^+_{B_{4,8}}\right).\]










A plane partition in $\PP^{[4]}\left(\Lambda_{\Gr(4,8)}\right)$ can be drawn as an ordinary plane partition whose $3$-dimensional Ferrers diagram fits in a $4 {\times} 4 {\times} 4$ box.  We encode this plane partition by labeling the vertices of the Hasse diagram of $\Lambda_{\Gr(4,8)}$ with their heights, and add $i$ to the label of each vertex of height $i$ to obtain an increasing tableau of shape $\Lambda_{\Gr(4,8)}$ (the \defn{rectangle}).  The embedding of $\Lambda_{\Gr(4,8)}$ inside $\Lambda_{\OG(8,16)}$ has the effect of adding two empty half-diamonds of vertices---one to the top of the labeled Hasse diagram, and one to the bottom.  To conserve space, we have only drawn the lower part of this embedding, since our operations won't interact with the top half-diamond.

\[\vcenter{\hbox{\protect \resizebox{!}{1in}{\begin{tikzpicture}
\planepartition{{4,4,3,2},{4,3,3,2},{4,3,2,1},{2,1,1}}
\end{tikzpicture}}}} \mapsto 
\vcenter{\hbox{\protect\begin{tikzpicture}[scale=.4]
	
	\node (g) at (3, 3) {};
	\node (h) at (4, 4) {};
	\node (i) at (5, 5) {};
	\node (j) at (6, 6) {};
	
	\node (k) at (2, 4) {};
	\node (l) at (3, 5) {};
	\node (m) at (4, 6) {};
	\node (n) at (5, 7) {};
	
	\node (o) at (1, 5) {};
	\node (p) at (2, 6) {};
	\node (q) at (3, 7) {};
	\node (r) at (4, 8) {};
	
	\node (s) at (0, 6) {};
	\node (t) at (1, 7) {};
	\node (u) at (2, 8) {};
	\node (v) at (3, 9) {};
	
	
	\draw (g) -- (h) -- (i) -- (j);
	\draw (k) -- (l) -- (m) -- (n);
	\draw (o) -- (p) -- (q) -- (r);
	\draw (s) -- (t) -- (u) -- (v);
	\draw (g) -- (k) -- (o) -- (s);
	\draw (h) -- (l) -- (p) -- (t);
	\draw (i) -- (m) -- (q) -- (u);
	\draw (j) -- (n) -- (r) -- (v);
	
	
	\draw[thick,fill=white] (g) circle [radius=0.5cm] node {0};
	\draw[thick,fill=white] (h) circle [radius=0.5cm] node {1};
	\draw[thick,fill=white] (i) circle [radius=0.5cm] node {2};
	\draw[thick,fill=white] (j) circle [radius=0.5cm] node {2};
	\draw[thick,fill=white] (k) circle [radius=0.5cm] node {1};
	\draw[thick,fill=white] (l) circle [radius=0.5cm] node {2};
	\draw[thick,fill=white] (m) circle [radius=0.5cm] node {3};
	\draw[thick,fill=white] (n) circle [radius=0.5cm] node {3};
	\draw[thick,fill=white] (o) circle [radius=0.5cm] node {1};
	\draw[thick,fill=white] (p) circle [radius=0.5cm] node {3};
	\draw[thick,fill=white] (q) circle [radius=0.5cm] node {3};
	\draw[thick,fill=white] (r) circle [radius=0.5cm] node {4};
	\draw[thick,fill=white] (s) circle [radius=0.5cm] node {2};
	\draw[thick,fill=white] (t) circle [radius=0.5cm] node {4};
	\draw[thick,fill=white] (u) circle [radius=0.5cm] node {4};
	\draw[thick,fill=white] (v) circle [radius=0.5cm] node {4};
\end{tikzpicture}}}\mapsto
\vcenter{\hbox{\protect\begin{tikzpicture}[scale=.4]
	
	\node (g) at (3, 3) {};
	\node (h) at (4, 4) {};
	\node (i) at (5, 5) {};
	\node (j) at (6, 6) {};
	
	\node (k) at (2, 4) {};
	\node (l) at (3, 5) {};
	\node (m) at (4, 6) {};
	\node (n) at (5, 7) {};
	
	\node (o) at (1, 5) {};
	\node (p) at (2, 6) {};
	\node (q) at (3, 7) {};
	\node (r) at (4, 8) {};
	
	\node (s) at (0, 6) {};
	\node (t) at (1, 7) {};
	\node (u) at (2, 8) {};
	\node (v) at (3, 9) {};
	
	
	\draw (g) -- (h) -- (i) -- (j);
	\draw (k) -- (l) -- (m) -- (n);
	\draw (o) -- (p) -- (q) -- (r);
	\draw (s) -- (t) -- (u) -- (v);
	\draw (g) -- (k) -- (o) -- (s);
	\draw (h) -- (l) -- (p) -- (t);
	\draw (i) -- (m) -- (q) -- (u);
	\draw (j) -- (n) -- (r) -- (v);
	
	
	\draw[thick,fill=white] (g) circle [radius=0.5cm] node {1};
	\draw[thick,fill=white] (h) circle [radius=0.5cm] node {3};
	\draw[thick,fill=white] (i) circle [radius=0.5cm] node {5};
	\draw[thick,fill=white] (j) circle [radius=0.5cm] node {6};
	\draw[thick,fill=white] (k) circle [radius=0.5cm] node {3};
	\draw[thick,fill=white] (l) circle [radius=0.5cm] node {5};
	\draw[thick,fill=white] (m) circle [radius=0.5cm] node {7};
	\draw[thick,fill=white] (n) circle [radius=0.5cm] node {8};
	\draw[thick,fill=white] (o) circle [radius=0.5cm] node {4};
	\draw[thick,fill=white] (p) circle [radius=0.5cm] node {7};
	\draw[thick,fill=white] (q) circle [radius=0.5cm] node {8};
	\draw[thick,fill=white] (r) circle [radius=0.5cm] node {\scalebox{.9}{10}};
	\draw[thick,fill=white] (s) circle [radius=0.5cm] node {6};
	\draw[thick,fill=white] (t) circle [radius=0.5cm] node {9};
	\draw[thick,fill=white] (u) circle [radius=0.5cm] node {\scalebox{.9}{10}};
	\draw[thick,fill=white] (v) circle [radius=0.5cm] node {\scalebox{.9}{11}};
\end{tikzpicture}}}
\mapsto 
\vcenter{\hbox{\protect\begin{tikzpicture}[scale=.4]
	\node (a) at (0, 0) {};
	\node (b) at (1, 1) {};
	\node (c) at (0, 2) {};
	\node (d) at (2, 2) {};
	\node (e) at (1, 3) {};
	\node (f) at (0, 4) {};
	
	\node (g) at (3, 3) {};
	\node (h) at (4, 4) {};
	\node (i) at (5, 5) {};
	\node (j) at (6, 6) {};
	
	\node (k) at (2, 4) {};
	\node (l) at (3, 5) {};
	\node (m) at (4, 6) {};
	\node (n) at (5, 7) {};
	
	\node (o) at (1, 5) {};
	\node (p) at (2, 6) {};
	\node (q) at (3, 7) {};
	\node (r) at (4, 8) {};
	
	\node (s) at (0, 6) {};
	\node (t) at (1, 7) {};
	\node (u) at (2, 8) {};
	\node (v) at (3, 9) {};
	
	\draw (a) -- (b) -- (d) -- (g);
	\draw (b) -- (c) -- (e) -- (k);
	\draw (d) -- (e) -- (f) -- (o);
	
	\draw (g) -- (h) -- (i) -- (j);
	\draw (k) -- (l) -- (m) -- (n);
	\draw (o) -- (p) -- (q) -- (r);
	\draw (s) -- (t) -- (u) -- (v);
	\draw (g) -- (k) -- (o) -- (s);
	\draw (h) -- (l) -- (p) -- (t);
	\draw (i) -- (m) -- (q) -- (u);
	\draw (j) -- (n) -- (r) -- (v);
	
	\draw[thick,fill=white] (a) circle [radius=0.5cm] node {};
	\draw[thick,fill=white] (b) circle [radius=0.5cm] node {};
	\draw[thick,fill=white] (c) circle [radius=0.5cm] node {};
	\draw[thick,fill=white] (d) circle [radius=0.5cm] node {};
	\draw[thick,fill=white] (e) circle [radius=0.5cm] node {};
	\draw[thick,fill=white] (f) circle [radius=0.5cm] node {};
	
	\draw[thick,fill=white] (g) circle [radius=0.5cm] node {1};
	\draw[thick,fill=white] (h) circle [radius=0.5cm] node {3};
	\draw[thick,fill=white] (i) circle [radius=0.5cm] node {5};
	\draw[thick,fill=white] (j) circle [radius=0.5cm] node {6};
	\draw[thick,fill=white] (k) circle [radius=0.5cm] node {3};
	\draw[thick,fill=white] (l) circle [radius=0.5cm] node {5};
	\draw[thick,fill=white] (m) circle [radius=0.5cm] node {7};
	\draw[thick,fill=white] (n) circle [radius=0.5cm] node {8};
	\draw[thick,fill=white] (o) circle [radius=0.5cm] node {4};
	\draw[thick,fill=white] (p) circle [radius=0.5cm] node {7};
	\draw[thick,fill=white] (q) circle [radius=0.5cm] node {8};
	\draw[thick,fill=white] (r) circle [radius=0.5cm] node {\scalebox{.9}{10}};
	\draw[thick,fill=white] (s) circle [radius=0.5cm] node {6};
	\draw[thick,fill=white] (t) circle [radius=0.5cm] node {9};
	\draw[thick,fill=white] (u) circle [radius=0.5cm] node {\scalebox{.9}{10}};
	\draw[thick,fill=white] (v) circle [radius=0.5cm] node {\scalebox{.9}{11}};
\end{tikzpicture}}}\]

Our bijection sequentially performs \K-theoretic jeu-de-taquin slides into the lower vacant half-diamond---these slides behave just like their ordinary jeu-de-taquin counterparts, except that when there are ties, ``everybody wins'' (see \Cref{sec:jdt} for more details).  For $i<j$, this rule is determined by the local pictures

\[
\begin{array}{c} \vcenter{\hbox{\protect\begin{tikzpicture}[scale=.5]
	\node (a) at (0, 0) {};
	\node (b) at (-1, 1) {};
	\node (c) at (1, 1) {};
	\draw (a) -- (b);
	\draw (a) -- (c);
	\draw[thick,fill=white] (a) circle [radius=0.5cm] node {$\bullet$};
	\draw[thick,fill=white] (b) circle [radius=0.5cm] node {$i$};
	\draw[thick,fill=white] (c) circle [radius=0.5cm] node {$j$};
\end{tikzpicture}}} \mapsto \vcenter{\hbox{\protect\begin{tikzpicture}[scale=.5]
	\node (a) at (0, 0) {};
	\node (b) at (-1, 1) {};
	\node (c) at (1, 1) {};
	\draw (a) -- (b);
	\draw (a) -- (c);
	\draw[thick,fill=white] (a) circle [radius=0.5cm] node {$i$};
	\draw[thick,fill=white] (b) circle [radius=0.5cm] node {$\bullet$};
	\draw[thick,fill=white] (c) circle [radius=0.5cm] node {$j$};
\end{tikzpicture}}} \\[.8cm] \vcenter{\hbox{\protect\begin{tikzpicture}[scale=.5]
	\node (a) at (0, 0) {};
	\node (b) at (-1, 1) {};
	\node (c) at (1, 1) {};
	\draw (a) -- (b);
	\draw (a) -- (c);
	\draw[thick,fill=white] (a) circle [radius=0.5cm] node {$\bullet$};
	\draw[thick,fill=white] (b) circle [radius=0.5cm] node {$j$};
	\draw[thick,fill=white] (c) circle [radius=0.5cm] node {$i$};
\end{tikzpicture}}} \mapsto \vcenter{\hbox{\protect\begin{tikzpicture}[scale=.5]
	\node (a) at (0, 0) {};
	\node (b) at (-1, 1) {};
	\node (c) at (1, 1) {};
	\draw (a) -- (b);
	\draw (a) -- (c);
	\draw[thick,fill=white] (a) circle [radius=0.5cm] node {$i$};
	\draw[thick,fill=white] (b) circle [radius=0.5cm] node {$j$};
	\draw[thick,fill=white] (c) circle [radius=0.5cm] node {$\bullet$};
\end{tikzpicture}}} \end{array} \hspace{1em} \text{and} \hspace{1em} 
\begin{array}{c}\vcenter{\hbox{\protect\begin{tikzpicture}[scale=.5]
	\node (a) at (0, 0) {};
	\node (b) at (-1, 1) {};
	\node (c) at (1, 1) {};
	\draw (a) -- (b);
	\draw (a) -- (c);
	\draw[thick,fill=white] (a) circle [radius=0.5cm] node {$\bullet$};
	\draw[thick,fill=white] (b) circle [radius=0.5cm] node {$i$};
	\draw[thick,fill=white] (c) circle [radius=0.5cm] node {$i$};
\end{tikzpicture}}} \mapsto \vcenter{\hbox{\protect\begin{tikzpicture}[scale=.5]
	\node (a) at (0, 0) {};
	\node (b) at (-1, 1) {};
	\node (c) at (1, 1) {};
	\draw (a) -- (b);
	\draw (a) -- (c);
	\draw[thick,fill=white] (a) circle [radius=0.5cm] node {$i$};
	\draw[thick,fill=white] (b) circle [radius=0.5cm] node {$\bullet$};
	\draw[thick,fill=white] (c) circle [radius=0.5cm] node {$\bullet$};
\end{tikzpicture}}} \\[.8cm] \vcenter{\hbox{\protect\begin{tikzpicture}[scale=.5]
	\node (a) at (0, 2) {};
	\node (b) at (-1, 1) {};
	\node (c) at (1, 1) {};
	\draw (a) -- (b);
	\draw (a) -- (c);
	\draw[thick,fill=white] (a) circle [radius=0.5cm] node {$i$};
	\draw[thick,fill=white] (b) circle [radius=0.5cm] node {$\bullet$};
	\draw[thick,fill=white] (c) circle [radius=0.5cm] node {$\bullet$};
\end{tikzpicture}}} \mapsto \vcenter{\hbox{\protect\begin{tikzpicture}[scale=.5]
	\node (a) at (0, 2) {};
	\node (b) at (-1, 1) {};
	\node (c) at (1, 1) {};
	\draw (a) -- (b);
	\draw (a) -- (c);
	\draw[thick,fill=white] (a) circle [radius=0.5cm] node {$\bullet$};
	\draw[thick,fill=white] (b) circle [radius=0.5cm] node {$i$};
	\draw[thick,fill=white] (c) circle [radius=0.5cm] node {$i$};
\end{tikzpicture}}} \end{array}\]

Unlike in the classical jeu-de-taquin theory, the order of these slides matters; as illustrated below, we perform slides row by row into the half-diamond, starting with the highest row.  

\[\vcenter{\hbox{\protect%
}}}.\]

In fact---no matter the starting increasing tableau of shape $\Lambda_{\Gr(4,8)}$---after applying this process, the labeled vertices are \emph{always} the dual of the shape $\trap{4}{8}$ (the \defn{shifted trapezoid}).  We obtain the desired bijection by subtracting $i$ from the labels of each vertex of height $i$, and interpreting the resulting labeling as a plane partition in $\PP^{[4]}\left(\trap{4}{8}\right)$.

\subsection{Cohomology}
\label{sec:previous_work}

It is natural to ask what happens to our theory in the context of ordinary cohomological Schubert calculus.  To address this, we define a linear extension of a poset $\Po$ with $n$ elements to be a strictly order-preserving bijection $\T: \Po \to [n]$, where $[n]:=\{1,2, \ldots, n\}$.  That is, if $p \succ q \in \Po$ then $\T(p) > \T(q).$    We call linear extensions of a general poset $\Po$ \defn{standard tableaux}, and we write $\SYT(\Po)$ for the set of all linear extensions of $\Po$.

Just as increasing tableaux govern the structure coefficients of \K-theoretic Schubert calculus of minuscule varieties, standard tableaux are used to compute the Schubert structure coefficients in the ordinary cohomology ring~\cite{thomas2009combinatorial}. We note that every standard tableau is also an increasing tableau, reflecting the fact that the \K-theoretic Schubert calculus is a richer theory (but see~\Cref{thm:knutson_mult_free}).

The order polynomial of $\Po$ encodes the number of its standard tableaux.  As in~\cite[Proposition 13.1]{stanley1972ordered}, the leading coefficient of the order polynomial $\left|\PP^{[\ell]}\left(\Po\right)\right| $ of $\Po$ is  $\frac{\ell^n}{n!}\left|\SYT(\Po)\right|.$  Hence, equating leading coefficients, two doppelg\"angers $\Po$ and $\mathcal{Q}$ must have the same number of standard tableaux.   In fact, our bijections in~\Cref{thm:main_thm1} restrict from \K-theory to ordinary cohomology, giving bijections between the standard tableaux for the doppelg\"angers of~\Cref{fig:names}.

\begin{theorem}
Fix \[(\X,\Y) \in \Big\{\big(\Gr(k,n),B_{k,n}\big),\big(\OG(6,12),H_3\big),\big(\QQ^{2n},I_2(2n)\big) \Big\}.\]  Then under the inclusion $\SYT(\Po) \subseteq \IT^{[n]}(\Po)$, the bijections of~\Cref{thm:main_thm1} restrict to bijections
\[ \SYT\left(\Lambda_\X\right) \simeq \SYT\left(\Phi^+_\Y\right). \]
\end{theorem}

\begin{remark}
In his study of dual equivalence~\cite{haiman1992dual}, M.~Haiman gave an elegant jeu-de-taquin bijection called \emph{rectification} for the identity 
 \begin{equation}\left|\SYT\left(\Lambda_{\Gr(k,n)}\right)\right| = \left|\SYT\left(\trap{k}{n}\right)\right|.\label{eq:syt_rect_trap}\end{equation}
Our bijection of~\Cref{thm:main_thm1} simultaneously generalizes M.~Haiman's bijection and provides the sought-after bijective proof of R.~Proctor's result for plane partitions.  In~\Cref{ex:ex3}, we have boxed the standard tableaux that correspond to the restriction of our bijection to M.~Haiman's result.
\label{rem:previous_work}
\end{remark}






\section{Philosophy and Outline}
\label{sec:philo}
\noindent
The general idea of our approach may be summarized as follows: 
\begin{quote}
Given a ring with a basis indexed by combinatorial objects and a combinatorial rule for structure coefficients, multiplicity-free products in the ring are equivalent to bijections. 
\end{quote}
 In this section, we give a short example of this philosophy by sketching the argument of a parallel result due to R.~Stanley~\cite{stanley1986symmetries}.  The organization of our paper then mirrors this example.

\begin{theorem}[{R.~Stanley~\cite[Section 3]{stanley1986symmetries}}]
The number of self-complementary plane partitions inside a $(2a) {\times} (2b) {\times} (2c)$ box is equal to the number of pairs of plane partitions, each fitting inside an $a {\times} b {\times} c$ box.
\label{thm:stanley}
\end{theorem}



\subsection{A Ring}
\label{sec:a_ring}
Recall that the \defn{ring of symmetric polynomials in $n$ variables} consists of those polynomials invariant under the natural action of the symmetric group $\mathfrak{S}_n$: \[\mathrm{Sym}_n:=\big\{f \in \mathbb{Q}[x_1,\ldots,x_{n}] : w \cdot f =f \text{ for all }w \in \mathfrak{S}_n\big\},\] where $w \cdot x_i=x_{w(i)}$.  

We fix some notation.  An \defn{integer partition} is a tuple \[\lambda=\left(\lambda_1 \geq \lambda_2 \geq \cdots \geq \lambda_n >0\right)\] with $\lambda_i \in \mathbb{Z}_{>0}$.  We say that $\lambda$ has $n$ \defn{parts}.  The representation theory of the symmetric and general linear groups assigns a poset to $\lambda$ as follows.  The \defn{Ferrers shape} of $\lambda$ is drawn in the quarter plane $\ZZ_{\geq_0} \times \ZZ_{\leq_0}$ as $k$ left-justified rows of axis-aligned unit squares (\defn{boxes}) of lengths $\lambda_i$.   We define a partial order on any collection of boxes in the quarter plane by letting a box be less than or equal to any box weakly below it \emph{and} weakly to its right.  When the collection of boxes defines a connected subset of the quarter plane, the Hasse diagram is obtained by vertically reflecting, then rotating the plane by $45^\circ$ counterclockwise, drawing edges between boxes that share an edge, and replacing all boxes by vertices.

For $\lambda$ a partition with at most $n$ parts, let $\mathrm{SSYT}_n(\lambda)$ be the set of \defn{semistandard tableaux of shape $\lambda$}---fillings of the Ferrers shape of $\lambda$ with numbers in $[n]$ that weakly increase across rows, and strictly increase down columns (see the top right tableau in~\Cref{fig:thmstanley} for an example).  The ring $\mathrm{Sym}_n$ has a linear basis of \defn{Schur polynomials} $\big\{s_\lambda : \lambda \text{ has at most } n \text{ parts}\big\}$, where  \[s_\lambda := \sum_{\T \in \mathrm{SSYT}_n(\lambda)} \, \prod_{i=1}^n x_i^{\text{number of times }i\text{ appears in }\T}.\]

\subsection{A Multiplicity-Free Product}
R.~Stanley's argument hinges on the following multiplicity-free identity in $\mathrm{Sym}_{b+c}$, expressing the square of a Schur polynomial indexed by the rectangular partition $\left(a^b\right):=\underbrace{(a,a,\ldots,a)}_{b \text{ parts}}$ in the Schur basis:
\begin{equation}
    s_{(a^b)}^2 = \sum_{\gamma} s_\gamma,
\label{eq:stanley}
\end{equation}
where $\gamma$ ranges over the set of partitions \[\bigg \{\big(a+\delta_1,\ldots,a+\delta_r,a-\delta_r,\ldots,a-\delta_1\big) : \delta = \big(\delta_1, \ldots,\delta_r\big) \subseteq (a^b) \bigg\}.\]  This expansion may be proven using a variant of the Littlewood-Richardson rule---more generally, products of two Schur polynomials indexed by different rectangular partitions are multiplicity free.

\begin{example}
\ytableausetup{boxsize=.5em}
For $a=b=2$, we have the $\binom{a+b}{a}=\binom{4}{2}=6$-term expansion
\[s_{\ydiagram{2,2}*[*(gray)]{2,2}}^2 = s_{\ydiagram{2,2,2,2}*[*(gray)]{2,2}}+s_{\ydiagram{3,2,2,1}*[*(gray)]{2,2}}+s_{\ydiagram{4,2,2}*[*(gray)]{2,2}}+s_{\ydiagram{3,3,1,1}*[*(gray)]{2,2}}+s_{\ydiagram{4,3,1}*[*(gray)]{2,2}}+s_{\ydiagram{4,4}*[*(gray)]{2,2}}.\]
\ytableausetup{boxsize=normal}
\label{ex:prod_sum}
\end{example}


\Cref{thm:stanley} follows from~\Cref{eq:stanley} as follows.  The terms in the product on the left-hand side of~\Cref{eq:stanley} are indexed by pairs of rectangular semistandard tableaux with entries in $[b+c]$ (top left of~\Cref{fig:thmstanley}).  By subtracting $i$ from the $i$th row, we produce a pair of plane partitions, each fitting inside an $a {\times} b {\times} c$ box, from this pair of semistandard tableaux (bottom left of~\Cref{fig:thmstanley}).

We need to do slightly more work to identify the right-hand side of \Cref{eq:stanley}.  For each term from the right-hand side, we again have a semistandard tableau (top right of \Cref{fig:thmstanley}), generally of nonrectangular shape. As before, we subtract $i$ from the $i$th row of the tableau to produce a plane partition.  Now observe that the partitions $\lambda$ occuring in the sum on the right-hand side of \Cref{eq:stanley} are exactly of the form required so that $\lambda$ and its rotation by $180^\circ$ may be placed together to form a rectangular partition of shape $(2a){\times} (2b)$.  The proof of \Cref{thm:stanley} is completed by noting that the filling of this rotation is specified by the self-complementarity condition (bottom right of~\Cref{fig:thmstanley}).


\subsection{A Bijection from a Rule for Structure Coefficients}\label{sec:a_bijection}
As sketched at the end of~\cite[Section 3]{stanley1986symmetries}, \Cref{thm:stanley} can be realized with a simple bijection.  Semistandard tableaux (unlike plane partitions) come with a theory of jeu-de-taquin. By a standard combinatorial version of the Littlewood-Richardson rule (see \cite[Chapter 5]{fulton1997young}), placing the initial pair of semistandard tableaux ``kitty-corner'' from each other and applying jeu-de-taquin until arriving at a north-west-justified (``straight'') shape gives a bijection from the pairs of tableaux representing the left-hand side of~\Cref{eq:stanley} to the semistandard tableaux representing the terms of the right-hand side.

\begin{figure}[htbp]
\[ \scalebox{0.8}{\xymatrix{ 
\vcenter{\hbox{\protect\begin{ytableau}
\none[\bullet] & \none[\bullet]  & \none[\bullet]  & \none[\bullet]  & 1 & 1 & 2 & 3 \\
\none[\bullet] & \none[\bullet]  & \none[\bullet]  & \none[\bullet]  & 2 & 4 & 5 & 5 \\
\none[\bullet] & \none[\bullet]  & \none[\bullet]  & \none[\bullet] & 6 & 6 & 6 & 6 \\
1 & 1 & 1 & 2 & \none & \none & \none & \none \\
3 & 3 & 4 & 5 & \none & \none & \none & \none \\
4 & 5 & 6 & 6 & \none & \none & \none & \none
\end{ytableau}}}\ar[rr]^-{\text{jeu-de-taquin}}
& &
\vcenter{\hbox{\protect\begin{ytableau}
1 & 1  &  1 & 1  & 1 & 2 & 3 & 5 \\
2 & 2 & 4  & 4  & 5 & 6 & 6 & 6 \\
3 & 3  & 5  & 6 & 6 & \none & \none & \none \\
4 & 5 & 6 & \none  & \none & \none & \none & \none 
\end{ytableau}}} \ar[d]^{\text{subtract }i\text{ from }i\text{th row}}_{\substack{\text{complete to}\\\text{self-complementary}}}\\
\vcenter{\hbox{\protect \resizebox{!}{1in}{
\begin{tikzpicture}
	\planepartition{{3,3,2,1},{3,2,1,1},{1}}
\end{tikzpicture}}}}\hspace{1ex}
\vcenter{\hbox{\protect \resizebox{!}{1in}{
\begin{tikzpicture}
	\planepartition{{3,3,3,3},{3,3,2},{2,1}}
\end{tikzpicture}}}}
 \ar[rr]_-{\text{\Cref{thm:stanley}}} \ar[u]^{\text{add }i\text{ to }i\text{th row}}_{\substack{\text{place tableaux}\\\text{kitty-corner}}} & & \vcenter{\hbox{\protect \resizebox{!}{1in}{\begin{tikzpicture}
	\planepartition{{6,6,6,6,6,5,4,2},{6,6,4,4,3,2,2,2},{6,6,4,3,3,2,1},{6,5,4,3,3,2},{4,4,4,3,2,2},{4,2,1}}
\end{tikzpicture}}}}} }\]
\caption{An illustration of the bijective proof of~\Cref{thm:stanley} between pairs of plane partitions in a $3{\times}4{\times}3$ box and self-complementary plane partitions in a $6{\times}8{\times}6$ box.}
\label{fig:thmstanley}
\end{figure}
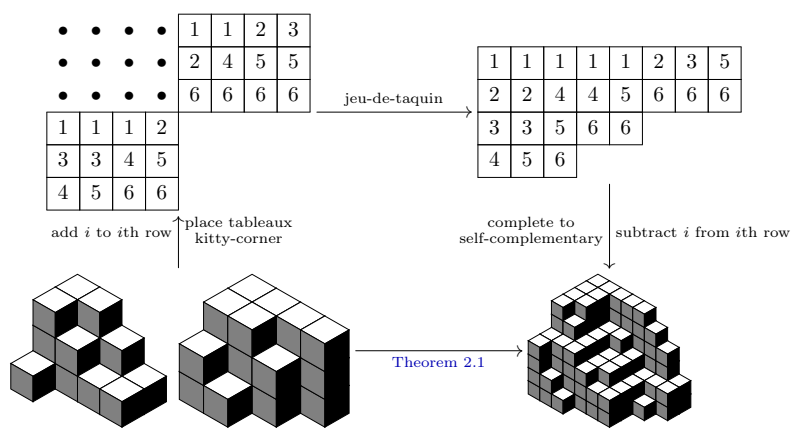

\subsection{Outline of the Paper}
In summary, our philosophy is that a multiplicity-free identity in a ring with a combinatorial rule for structure coefficients is \emph{equivalent} to a bijection.  In this paper, we apply this philosophy using the objects and tools of minuscule \K-theoretic Schubert calculus.  To obtain the bijections of~\Cref{thm:main_thm1}, we therefore need:
\begin{itemize}
    \item combinatorial objects (\Cref{sec:minuscule});
    \item rings with bases indexed by those objects (\Cref{sec:ring});
    \item combinatorial rules to compute structure coefficients in those rings (\Cref{sec:structure_coeffs}); which lead to
	\item bijections (\Cref{thm:gen_main_thm}), and their equivalent
    \item multiplicity-free identities (\Cref{thm:gen_main_thm_Ring}); which allow us to
	\item identify interesting special cases (\Cref{sec:coincidental,sec:embedding,sec:applications}).
\end{itemize}

The remainder of the paper is structured as follows.  In~\Cref{sec:definitions}, we review required background and fix notation for posets, root systems, reflection groups, and flag varieties.  In~\Cref{sec:minuscule}, we describe minuscule (co)weights and their associated posets.    In \Cref{sec:ring}, we recall the basic notions of (cohomological and \K-theoretic) Schubert calculus, building to the powerful combinatorial toolkit of \Cref{sec:structure_coeffs}.

In~\Cref{sec:results_and_conjectures}, we state and prove our main theorem (\Cref{thm:gen_main_thm}), which provides a bijective framework for doppelg\"angers.  We then turn to the task of applying \Cref{thm:gen_main_thm} to prove \Cref{thm:main_thm1}.  We recall the coincidental types and their root posets in~\Cref{sec:coincidental}, and \Cref{sec:embedding} is then devoted to embedding the posets in the first two columns of~\Cref{fig:triples} inside ambient minuscule posets.  Finally, we specialize~\Cref{thm:gen_main_thm} in~\Cref{sec:applications} to conclude \Cref{thm:main_thm1}.

In \Cref{sec:future_work}, we outline some related open problems.

\section{Background}
\label{sec:definitions}

In this section, we fix notation and review background on posets, root systems, real reflection groups, and flag varieties.  We refer the reader to~\cite{hiller1982geometry,humphreys1992reflection,kane} for more comprehensive treatments; we follow the notation of~\cite{humphreys1992reflection}.

\subsection{Posets}
\label{sec:posets}


 
An \defn{order ideal} of a poset $\Po$ is a subset $\w \subseteq \Po$ such that $y \succ x$ and $y \in \w$ together imply $x \in \w$.  We call an order ideal $\w$ of $\Po$ a \defn{straight shape}, and the difference of two straight shapes $\w \subseteq \vv$ a \defn{skew shape} $\vv/\w$.  Note that $\vv$ is the special case of a skew shape for $\w=\emptyset$.   Similarly, an \defn{order filter} is a subset of $\Po$ whose complement is an order ideal, and we call an order filter an \defn{anti-straight shape}.   As subposets of $\Po$, order ideals and filters inherit a partial order.  We write $J(\Po)$ for the set of all order ideals of a poset $\Po$, noting that $J(\Po)$ is itself a poset under inclusion of order ideals.  Observe that there is a natural correspondence $J(\Po) \simeq \PP^{[1]}(\Po)$ and, more generally, \[J(\Po \times [\ell]) \simeq \PP^{[\ell]}(\Po).\] 

A \defn{lattice} $\mathcal{L}$ is a poset such that any two $\w,\uu \in \mathcal{L}$ have both a least upper bound $\w \vee \uu$ and a greatest lower bound $\w \wedge \uu$. An element $\vv$ in a lattice $\mathcal{L}$ is \defn{join-irreducible} if $\vv = \w \vee \uu$ implies $\vv \in \{\w,\uu\}$.  Associated to a lattice $\mathcal{L}$ is its subposet $\mathcal{L}_{\rm ji}$ of join-irreducibles under the restiction of the partial order on $\mathcal{L}$.
A lattice $\mathcal{L}$ is a \defn{distributive lattice} if the binary operations $\vee$ and $\wedge$ distribute over each other.  The \defn{fundamental theorem of distributive lattices} states that for any finite distributive lattice $\mathcal{L}$, $J(\mathcal{L}_{\rm ji}) \cong \mathcal{L}$~\cite{birkhoff1937}.

\begin{remark}
	Recall from~\Cref{sec:a_ring} that an integer partition defines a Ferrers shape.  Similarly, a strict integer partition $\lambda=(\lambda_1 > \lambda_2 > \cdots > \lambda_k >0)$ defines a \defn{shifted Ferrers shape} by indenting the $i$th row $i$ steps to the right.   When working with Ferrers and shifted Ferrers shapes, we often find it convenient to switch to the English convention on tableau orientation. That is, we vertically reflect and then rotate our tableaux $135^\circ$ clockwise so ``gravity'' now points north-west.  \ytableausetup{boxsize=.5em} We differentiate shifted partitions from partitions using a subscripted ``$*$''---thus $(3,2,1)$ stands for the Ferrers shape $\raisebox{\height}{\ydiagram{3,2,1}}$, while $(3,2,1)_*$ is the shifted Ferrers shape $\raisebox{\height}{\ydiagram{3,1+2,2+1}}$.\ytableausetup{boxsize=normal}
	\label{rem:drawing}
\end{remark}





\subsection{Root Systems}\label{sec:rootsystems}
Let $V$ be a real Euclidean space of rank $n$ equipped with a nondegenerate symmetric inner product $\langle \cdot,\cdot\rangle$.  Let $\Phi \subset V$ be an irreducible (finite) root system; we do not assume $\Phi$ is crystallographic. Fix a choice of positive roots $\Phi^+$ and of simple roots $\Delta:=\{\alpha_1,\alpha_2,\ldots,\alpha_n\}$.  When we wish to differentiate types, we will write $\Phi=\Phi(X_n)$ for the root system of type $X_n$ (and similarly for other objects).  For a root $\alpha$, let $\alpha^\vee := 2 \frac{\alpha}{\langle \alpha,\alpha \rangle}$ be the corresponding \defn{coroot} and let $\Phi^\vee=\{\alpha^\vee:\alpha \in \Phi\}$ be the \defn{dual root system}.

The root system $\Phi$ is \defn{crystallographic} if $\langle \alpha, \beta^\vee \rangle \in \mathbb{Z}$ for all $\alpha,\beta \in \Phi$.  For $\Phi$ crystallographic, we let $Q:=\ZZ\Phi$ be the \defn{root lattice}, $Q^\vee:=\ZZ\Phi^\vee$ the \defn{coroot lattice}, and define the \defn{weight lattice} (whose elements are \defn{weights}) 
\[\Lambda:= \big\{\omega \in V : \langle \omega,\alpha^\vee \rangle \in \mathbb{Z} \text{ for all } \alpha \in \Delta\big\}
\]
and the \defn{coweight lattice} (whose elements are \defn{coweights})
 \[
 \Lambda^\vee:= \big\{\omega \in V : \langle \omega,\alpha \rangle \in \mathbb{Z} \text{ for all } \alpha \in \Delta\big\}.
 \]  As abelian groups, $\Lambda$ contains $Q$ as a subgroup and the finite number $f:=\left|\Lambda/Q\right|$ is called the \defn{index of connection}.    For each simple root $\alpha_i \in \Delta$, there is a corresponding \defn{fundamental weight} $\omega_i \in \Lambda$ defined by $(\omega_i, \alpha_j^\vee) = \delta_{ij}$, where $\delta_{ij}$ is the Kronecker delta.  The \defn{fundamental coweights} are defined analogously.   A weight $\omega$ is \defn{dominant} if $\langle \omega, \alpha_i^\vee\rangle \geq 0$ for all $\alpha_i\in\Delta$.  Each fundamental weight is dominant, and a weight is dominant exactly if it is a nonnegative linear combination of fundamental weights.

The \defn{dominance order} is the partial order on $\Lambda$ given by \begin{equation}\lambda \prec \omega \text{ if and only if } \omega-\lambda \text{ is a nonnegative sum of simple roots}.\label{eq:root_poset}\end{equation}    The \defn{positive root poset} is the restriction of dominance order to the positive roots $\Phi^+$.  The \defn{height} of a positive root $\alpha = \sum_{i=1}^n a_i \alpha_i$ is the postive integer $\hgt(\alpha)=\sum_{i=1}^n a_i$.  Abusing notation, we write $\Phi^+$ for the positive root poset, and we note that $\Phi^+$ has a unique maximal element $\widetilde{\alpha}$ called the \defn{highest root}.


\subsection{Reflection Groups}

\subsubsection{Classification}
For $\alpha \in \Phi^+$, we define the reflection $s_\alpha : V \to V$ by \[s_\alpha(v):=v-\langle v,\alpha\rangle \alpha^\vee\] for $v \in V$.  The \defn{(real) reflection group} of $\Phi$ is the group $W = \langle s_\alpha : \alpha \in \Phi^+ \rangle \subset \mathsf{O}(V)$ generated by these reflections.   Using our choice of simple roots $\Delta=\{\alpha_i\}_{i=1}^n$, $W$ is generated by the \defn{simple reflections} $S=\{s_{\alpha_i} : \alpha_i\in \Delta\}$.  We abbreviate $s_i:=s_{\alpha_i}$.

The simple reflections give $W$ a distinguished Coxeter presentation \[W = \big\langle s_1,s_2,\ldots,s_n : s_i^2=(s_is_j)^{m_{ij}}=e \big\rangle,\] where $e$ is the identity element of $W$ and $m_{ij}=m_{ji}$ are certain positive integers.  The pair $(W,S)$ of the finite reflection group and a choice of simple reflections is a finite \defn{Coxeter system}.  The relations $(s_is_j)^2=e$ are called \defn{commutation relations}, while higher-order relations $(s_is_j)^{m_{ij}}$ for $m_{ij}>2$ are called \defn{braid relations}.   The Coxeter presentation of $W$ is encoded by its \defn{Coxeter-Dynkin diagram}, the edge-labeled complete graph with vertex set $S$, where the edge from $s_i$ to $s_j$ is labeled by $m_{ij}$.  By a standard convention, we omit the edges with label $2$ (corresponding to commutations of simple reflections) and omit labels equal to $3$.  We also use a double edge to denote the label $4$ and a triple edge for the label $6$, with an arrow on each such edge pointing toward the shorter simple root.   \Cref{fig:min_classification} displays the Coxeter-Dynkin diagrams for the finite irreducible reflection groups (for now, ignore the vertices colored black).

With some low-dimensional redundancy, finite irreducible real reflection groups are classified as the crystallographic types $A_n, B_n, D_n, E_6, E_7, E_8, F_4,$ and $G_2$ (which come from crystallographic root systems), and the noncrystallographic types $H_3, H_4,$ and $I_2(m)$ (for $m \neq 1,2,3,4,6$).  We shall refer to these symbols as \defn{Coxeter-Cartan types}.

 Each finite crystallographic real reflection group $W$ has an affine extension $\widetilde{W}=W \ltimes Q^\vee$ obtained by adding a new ``affine'' simple reflection $s_0$ across an affine hyperplane perpendicular to $\widetilde{\alpha}$.
That is, \[\widetilde{W} = \big\langle s_0, s_1, \dots, s_n \big \rangle,\] where $s_0 : V \to V$ is defined by 
\[
s_0(v) := v - \big( \langle v, \widetilde{\alpha} \rangle - 1\big) \widetilde{\alpha}^\vee.
\]  The affine roots correspond to the vertices colored black in~\Cref{fig:min_classification}.

\begin{figure}[htbp]
\begin{center}
\begin{tabular}{cccc}
$\substack{\text{Coxeter-Cartan type} \\ \text{and index}}$ & $\GP$ & Coxeter-Dynkin diagram \\
$(A_1,1)$ & $\Gr(1,2)$ &
  $\vcenter{\hbox{\protect\begin{tikzpicture}[scale=.3]
    \draw[thick] (0 cm,0) -- (2 cm,0)  node [midway, above, sloped] (TextNode) {\tiny $\infty$} ;
    \draw[thick,solid,fill=gray] (0cm,0) circle (.6cm) node[white]  {$1$};
    \draw[thick,solid,fill=black] (2cm,0) circle (.6cm);
  \end{tikzpicture}}}$ \\
 $(A_n,k)$ & $\Gr(k,n+1)$ &
  $\vcenter{\hbox{\protect\begin{tikzpicture}[scale=.3]
    \draw[thick] (0 cm,0) -- (4 cm,0);
    \draw[dotted,thick] (4 cm,0) -- (8 cm,0);
    \draw[thick] (8 cm,0) -- (10 cm,0);
    \draw[thick] (0 cm,0) -- (5 cm,3);
    \draw[thick] (5 cm,3) -- (10 cm,0);
    \draw[thick,solid,fill=gray] (0cm,0) circle (.6cm) node[white]  {$1$};
    \draw[thick,solid,fill=gray] (2cm,0) circle (.6cm) node[white]  {$2$};
    \draw[thick,solid,fill=gray] (4cm,0) circle (.6cm);
    \draw[thick,solid,fill=gray] (8cm,0) circle (.6cm);
    \draw[thick,solid,fill=gray] (10cm,0) circle (.6cm) node[white]  {$n$};
    \draw[thick,solid,fill=black] (5cm,3) circle (.6cm);
  \end{tikzpicture}}}$ 
  \\
  $(B_n,n)$ &  $\mathbb{P}^{2n-1}$ &
 $\vcenter{\hbox{\protect \begin{tikzpicture}[scale=.3]
    \draw[thick] (0 cm,-.2) -- (2 cm,-.2);
    \draw[thick] (0 cm,.2) -- (2 cm,.2);
    \draw[thin] (1.3 cm,.6) -- (.8 cm,0);
    \draw[thin] (1.3 cm,-.6) -- (.8 cm,0);
    \draw[thick] (2 cm,0) -- (4 cm,0);
    \draw[dotted,thick] (4 cm,0) -- (8 cm,0);
    \draw[thick] (8 cm,0) -- (10 cm,1);
    \draw[thick] (8 cm,0) -- (10 cm,-1);
    \draw[thick,solid,fill=white] (0cm,0) circle (.6cm) node  {$1$};
    \draw[thick,solid,fill=white] (2cm,0) circle (.6cm) node  {$2$};
    \draw[thick,solid,fill=white] (4cm,0) circle (.6cm);
    \draw[thick,solid,fill=white] (8cm,0) circle (.6cm);
    \draw[thick,solid,fill=gray] (10cm,1) circle (.6cm) node[white]  {$n$};
    \draw[thick,solid,fill=black] (10cm,-1) circle (.6cm);
  \end{tikzpicture}}}$
  \\
$(C_n,1)$ & $\LG(n,2n)$ &
  $\vcenter{\hbox{\protect\begin{tikzpicture}[scale=.3]
    \draw[thick] (0 cm,-.2) -- (2 cm,-.2);
    \draw[thick] (0 cm,.2) -- (2 cm,.2);
    \draw[thin] (.7 cm,.6) -- (1.2 cm,0);
    \draw[thin] (.7 cm,-.6) -- (1.2 cm,0);
    \draw[thick] (2 cm,0) -- (4 cm,0);
    \draw[dotted,thick] (4 cm,0) -- (8 cm,0);
    \draw[thick] (8 cm,-.2) -- (10 cm,-.2);
    \draw[thick] (8 cm,.2) -- (10 cm,.2);
    \draw[thin] (9.3 cm,.6) -- (8.8 cm,0);
    \draw[thin] (9.3 cm,-.6) -- (8.8 cm,0);
    \draw[thick,solid,fill=gray] (0cm,0) circle (.6cm) node[white]  {$1$};
    \draw[thick,solid,fill=white] (2cm,0) circle (.6cm) node[white]  {$2$};
    \draw[thick,solid,fill=white] (4cm,0) circle (.6cm);
    \draw[thick,solid,fill=white] (8cm,0) circle (.6cm) node  {$n$};
    \draw[thick,solid,fill=black] (10cm,0) circle (.6cm);
  \end{tikzpicture}}}$ 
  \\
  $\substack{(D_n,1) \\ (D_n,2)}$ & $\OG(n,2n)$ &
  $\vcenter{\hbox{\protect\begin{tikzpicture}[scale=.3]
    \draw[thick] (0 cm,-1) -- (2 cm,0);
    \draw[thick] (0 cm,1) -- (2 cm,0);
    \draw[thick] (2 cm,0) -- (4 cm,0);
    \draw[dotted,thick] (4 cm,0) -- (8 cm,0);
    \draw[thick] (8 cm,0) -- (10 cm,1);
    \draw[thick] (8 cm,0) -- (10 cm,-1);
    \draw[thick,solid,fill=gray] (0cm,-1) circle (.6cm) node[white]  {$1$};
    \draw[thick,solid,fill=gray] (0cm,1) circle (.6cm) node[white]  {$2$};
    \draw[thick,solid,fill=white] (2cm,0) circle (.6cm) node  {$3$};
    \draw[thick,solid,fill=white] (4cm,0) circle (.6cm);
    \draw[thick,solid,fill=white] (8cm,0) circle (.6cm);
    \draw[thick,solid,fill=white] (10cm,1) circle (.6cm) node[black]  {$n$};
    \draw[thick,solid,fill=black] (10cm,-1) circle (.6cm);
  \end{tikzpicture}}}$ \\
   $(D_n,n)$ & $\QQ^{2n-2}$ &
  $\vcenter{\hbox{\protect\begin{tikzpicture}[scale=.3]
    \draw[thick] (0 cm,-1) -- (2 cm,0);
    \draw[thick] (0 cm,1) -- (2 cm,0);
    \draw[thick] (2 cm,0) -- (4 cm,0);
    \draw[dotted,thick] (4 cm,0) -- (8 cm,0);
    \draw[thick] (8 cm,0) -- (10 cm,1);
    \draw[thick] (8 cm,0) -- (10 cm,-1);
    \draw[thick,solid,fill=white] (0cm,-1) circle (.6cm) node[black]  {$1$};
    \draw[thick,solid,fill=white] (0cm,1) circle (.6cm) node[black]  {$2$};
    \draw[thick,solid,fill=white] (2cm,0) circle (.6cm) node  {$3$};
    \draw[thick,solid,fill=white] (4cm,0) circle (.6cm);
    \draw[thick,solid,fill=white] (8cm,0) circle (.6cm);
    \draw[thick,solid,fill=gray] (10cm,1) circle (.6cm) node[white]  {$n$};
    \draw[thick,solid,fill=black] (10cm,-1) circle (.6cm);
  \end{tikzpicture}}}$ \\
   $\substack{(E_61) \\ (E_6,6)}$ & $\mathbb{OP}^2$ &
 $\vcenter{\hbox{\protect \begin{tikzpicture}[scale=.3]
    \draw[thick] (0 cm,0) -- (8 cm,0);
    \draw[thick] (4 cm,0) -- (4 cm,4);
    \draw[thick,solid,fill=gray] (0cm,0) circle (.6cm) node[white]  {$1$};
    \draw[thick,solid,fill=white] (2cm,0) circle (.6cm) node  {$3$};
    \draw[thick,solid,fill=white] (4cm,0) circle (.6cm) node  {$4$};
    \draw[thick,solid,fill=white] (6cm,0) circle (.6cm) node  {$5$};
    \draw[thick,solid,fill=gray] (8cm,0) circle (.6cm) node[white]  {$6$};
    \draw[thick,solid,fill=white] (4cm,2) circle (.6cm) node  {$2$};
    \draw[thick,solid,fill=black] (4cm,4) circle (.6cm);
  \end{tikzpicture}}}$
  \\
  $(E_7,1)$ & $G_\omega(\mathbb{O}^3,\mathbb{O}^6)$ &
   $\vcenter{\hbox{\protect \begin{tikzpicture}[scale=.3]
    \draw[thick] (0 cm,0) -- (12 cm,0);
    \draw[thick] (6 cm,0) -- (6 cm,2);
    \draw[thick,solid,fill=gray] (0cm,0) circle (.6cm) node[white]  {$1$};
    \draw[thick,solid,fill=white] (2cm,0) circle (.6cm) node  {$3$};
    \draw[thick,solid,fill=white] (4cm,0) circle (.6cm) node  {$4$};
    \draw[thick,solid,fill=white] (6cm,0) circle (.6cm) node  {$5$};
    \draw[thick,solid,fill=white] (8cm,0) circle (.6cm) node  {$6$};
    \draw[thick,solid,fill=white] (10cm,0) circle (.6cm) node  {$7$};
    \draw[thick,solid,fill=black] (12cm,0) circle (.6cm);
    \draw[thick,solid,fill=white] (6cm,2) circle (.6cm) node  {$2$};
  \end{tikzpicture}}}$
  \\
  $(H_3,3)$ & \text{none} & 
    $\vcenter{\hbox{\protect\begin{tikzpicture}[scale=.3]
    \draw[thick] (2 cm,0) -- (4 cm,0);
    \draw[thick] (0 cm,0) -- (2 cm,0) node [midway, above, sloped] (TextNode) {\tiny $5$} ;
    \draw[thick,solid,fill=white] (0cm,0) circle (.6cm) node  {$1$};
    \draw[thick,solid,fill=white] (2cm,0) circle (.6cm) node  {$2$};
    \draw[thick,solid,fill=gray] (4cm,0) circle (.6cm) node[white]  {$3$};
  \end{tikzpicture}}}$
  \\
  $\substack{(I_2(m),1) \\ (I_2(m),2)}$ & \text{none} &
  $\vcenter{\hbox{\protect\begin{tikzpicture}[scale=.3]
    \draw[thick] (0 cm,0) -- (2 cm,0)  node [midway, above, sloped] (TextNode) {\tiny $m$} ;
    \draw[thick,solid,fill=gray] (0cm,0) circle (.6cm) node[white]  {$1$};
    \draw[thick,solid,fill=gray] (2cm,0) circle (.6cm) node[white]  {$2$};
  \end{tikzpicture}}}$
  \end{tabular}
  
\end{center}
\caption{In crystallographic type, the roots $\alpha_i$ marked in gray have a corresponding cominuscule fundamental weight $\omega_i$; the affine simple root is marked in black.  For $H_3$ and $I_2(m)$, the roots marked in gray correspond to maximal parabolic quotients $W^{\langle i \rangle}$ whose longest element is fully commutative.} 
\label{fig:min_classification}
\end{figure}

\subsubsection{Weak Order}
For $w \in W$, we write $\len(w)$ for the \defn{length} of $w$, i.e. the least $\well$ such that $w$ can be written as $w=s_{i_1}s_{i_2}\cdots s_{i_\well}$ for some sequence of simple reflections and define \[\Red(w):=\{(s_{i_1},s_{i_2},\ldots, s_{i_\well}) : w=s_{i_1}s_{i_2}\cdots s_{i_\well} \text{ and } \well = \len(w)  \}\] for its set of \defn{reduced words}.  There is the structure of a graph on $\Red(w)$ by drawing edges between two reduced words when they differ by a commutation or braid relation; by Matsumoto's theorem, this graph is connected.   

We define the \defn{inversion set} of $w \in W$ by  \[\w:=\left(-w(\Phi^+) \right)\cap \Phi^+=\big\{\alpha_{i_1},s_{i_1}(\alpha_{i_2}),\ldots,s_{i_1}s_{i_2}\cdots s_{i_{\well-1}}(\alpha_{i_\well})\big\},\] where $s_{i_1}s_{i_2}\cdots s_{i_\well}$ is any reduced word for $w$.  Clearly, one has $|\w| = \len(w)$ for all $w \in W$.  We will often think of the inversion set of $w$ as a subposet of the positive root poset.

The \defn{Demazure product} on $W$ is defined by setting \[s_i \bullet w := \begin{cases} s_i w & \text{if } \len(s_i w)>\len(w) \\ w & \text{otherwise} \end{cases}\] and then extending to products of arbitrary elements. (The Demazure product is well-defined and associative---it corresponds to the product in the $0$-Hecke algebra, and yields a monoid structure on $W$.)

The \defn{weak order} is the partial order on $W$ defined by $w \leq v$ if and only if $\w \subseteq \vv$.  The group $W$ has a unique longest element $\wo$, which is maximal in the weak order.

\subsubsection{Parabolic Subgroups and Quotients}
A \defn{parabolic subgroup} $W_J \subset W$ is a group generated by a subset $J \subset S$ of the simple reflections, with (possibly reducible) subroot system $\Phi_J \subseteq \Phi$.  A \defn{maximal parabolic subgroup} is one for which $J=S\setminus \{s_i\}$---we shall denote such a subgroup by $W_{\langle i \rangle}$ and its root system by $\Phi_{\langle i \rangle}$.  The set $W^J:=W/W_J$ is called a \defn{parabolic quotient}; we identify $W^J$ with its minimal-length coset representatives and write $W^{\langle i \rangle} :=W/W_{\langle i \rangle}$.

Any $w\in W$ can be written as $w=w_Jw^J$ with $w_J \in W_J$ and $w^J \in W^J$.  We write $\wo(J)$ for the unique longest element of $W_J$.  Each parabolic quotient $W^J$ also has a longest element $\wo^J=\wo \wo(J)$, and $W^J$ consists of the elements in the closed interval $[e,\wo^J]$ of $W$.  As an interval in the weak order, the parabolic quotient $W^J$ therefore inherits the partial ordering from $W$.  The map \begin{align*} 
\check{\cdot} : W^J &\to W^J \\ w &\mapsto \check{w} = \wo w \wo(J)
\end{align*}
 is an antiautomorphism of the poset $W^J$.

\subsection{Flag Varieties}\label{sec:flag_varieties}
Let ${\sf G}$ be a semisimple complex Lie group. Inside ${\sf G}$, fix a Borel subgroup ${\sf B}$, an opposite Borel subgroup ${\sf B}_-$ and a maximal torus ${\sf T}:={\sf B} \cap {\sf B}_-$. For example, if ${\sf G} = {\sf SL}(n)$, we may choose ${\sf B}$ to be the subgroup of upper triangular matrices in ${\sf SL}(n)$ and ${\sf B}_-$ to be the lower triangular matrices in ${\sf SL}(n)$, whence ${\sf T}:={\sf B} \cap {\sf B}_-$ would be uniquely determined as the diagonal matrices in ${\sf SL}(n)$.

One recovers the data of~\Cref{sec:rootsystems} in the following way.  The \defn{Weyl group} of ${\sf G}$ is $W:={\sf N}({\sf T})/{\sf T}$, where ${\sf N}({\sf T})$ is the normalizer of ${\sf T}$ in ${\sf G}$.  Consider the complex Lie algebra $\mathfrak{g}$ of ${\sf G}$. The \defn{Cartan subalgebra} $\mathfrak{t}$ is the Lie algebra of the maximal torus ${\sf T}$. The adjoint action of $\mathfrak{t}$ on $\mathfrak{g}$ can be simultaneously diagonalized, so $\mathfrak{g}$ decomposes into a direct sum of weight spaces as
\[\mathfrak{g} = \bigoplus_{\omega \in \mathfrak{t}^*} \mathfrak{g}_\omega,\] where the direct sum is over linear functionals on $\mathfrak{t}$. In fact, $\mathfrak{g}_0 = \mathfrak{t}$, and this decomposition can be rewritten as
 \[\mathfrak{g}=\mathfrak{t} \oplus \bigoplus_{\alpha \in \Phi} \mathfrak{g}_\alpha,\] where $\Phi$ is a finite set of nonzero vectors in the dual space $\mathfrak{t}^*$ and each $\mathfrak{g}_\alpha$ is a one-dimensional subspace of $\mathfrak{g}$.  It turns out that $\Phi$ is a crystallographic root system. Let $\mathfrak{b}$ and $\mathfrak{b}_-$ be the Lie algebras of the Borel subgroup ${\sf B}$ and opposite Borel subgroup ${\sf B}_-$, respectively. For each $\alpha \in \Phi$, either $\mathfrak{g}_\alpha \subseteq \mathfrak{b}$ or $\mathfrak{g}_\alpha \subseteq \mathfrak{b}_-$. We obtain a choice of positive roots $\Phi^+$ from our choice of Borel subgroup by
 \[
 \Phi^+ = \left\{ \alpha \in \Phi : \mathfrak{g}_\alpha \subseteq \mathfrak{b} \right\}.\] The decomposition of $\Phi$ into positive and negative roots uniquely determines a set of simple roots $\Delta$.

Moreover, the Lie group ${\sf G}$ decomposes as the disjoint union (the \defn{Bruhat decomposition}) \[{\sf G} = \bigsqcup_{w \in W} {\sf B}_- w {\sf B}.\]

More generally, for ${\sf P} \supseteq {\sf B}$ a parabolic subgroup of ${\sf G}$, we write $W^{\sf P}=W/W_{\sf P}$ for the corresponding parabolic quotient and subgroup of $W$.  Then \[{\sf G} = \bigsqcup_{w \in W^{\sf P}} {\sf B}_- w {\sf P}\] and the \defn{generalized flag variety} $\GP$ has the \defn{Schubert cell decomposition} \begin{equation}\GP = \bigsqcup_{w \in W^{\sf P}} {\sf B}_- w {\sf P} / {\sf P}.\label{eq:bruhatdecomposition}\end{equation}%
With these conventions, the Schubert cell  ${\sf B}_- w {\sf P} / {\sf P}$ has codimension $\len(w)$.


\section{Minuscule Weights and Posets}
\label{sec:minuscule}





Let $\Phi$ be a crystallographic root system and $\Lambda$ its weight lattice. In this section, we recall the definition of a minuscule weight $\omega \in \Lambda$.  For each such minuscule weight $\omega$, we then define the minuscule poset associated to $\omega$ as a subset of the poset $\Phi^+$ of positive roots. After recalling some combinatorics associated to minuscule posets in~\Cref{sec:ring,sec:structure_coeffs}, we will use these posets to formulate our bijective framework for doppelg\"angers in~\Cref{sec:results_and_conjectures}.   For full background, we refer the reader to Bourbaki~\cite[Chapitre VIII, \S 7.3]{bourbaki1975}, H.~Hiller's text~\cite[Chapter V, \S 2]{hiller1982geometry}, and R.~Green's recent book~\cite{green2013combinatorics}.  The combinatorial aspects of minuscule weights that we recall here have been thoroughly studied, exploited, and extended by R.~Stanley, R.~Proctor, and J.~Stembridge~\cite{stanley1980weyl,proctor1984bruhat,stembridge1996fully,stembridge1998partial,stembridge2001minuscule}, among others.  
For discussion of the geometric aspects, see the book of S.~Billey and V.~Lakshmibai \cite{billey2000book}.

\subsection{Minuscule Weights}

\begin{definition}  A nonzero dominant weight $\omega \in \Lambda$ is called \defn{minuscule} if $\langle \omega,\alpha^\vee\rangle \in\{-1,0,1\}$ for $\alpha \in \Phi$.
\label{def:minuscule}
\end{definition}

A minuscule weight for the dual root system $\Phi^\vee$ is called a \defn{minuscule coweight}; the corresponding weight of the original root system is called \defn{cominuscule}.  The vertices of Dynkin diagrams whose corresponding coweight is minuscule are marked in gray in~\Cref{fig:min_classification}.   


\begin{theorem}
For $\omega$ a dominant coweight in crystallographic Coxeter-Cartan type, the following are equivalent:
\begin{enumerate}
  \item $\omega$ is minuscule---that is, $\omega \neq 0$ and $\langle \omega,\alpha\rangle \in\{-1,0,1\}$ for all $\alpha \in \Phi$;
    \item $\omega=\omega_i$ is a fundamental coweight, and $c_i=1$ in the expansion $\widetilde{\alpha}=\sum_{j=1}^n c_j\alpha_j$ of the highest root in the simple root basis;
      \item $\omega=\omega_i$ is a fundamental coweight, and there is an automorphism of the affine Dynkin diagram sending $\alpha_0$ to $\alpha_i$; and
        \item $\omega$ is a nonzero minimal representative of $\Lambda^\vee / Q^\vee$ in the dominance order.
\end{enumerate}
\label{thm:equ_min}
\end{theorem}



\subsection{Minuscule Posets}

The subgroup of the reflection group $W$ that stabilizes a fundamental weight $\omega_i$ is the maximal parabolic subgroup 
$W_{\langle i \rangle}$.
Let ${\sf P}$ be the corresponding maximal parabolic subgroup of ${\sf G}$.  The \defn{minuscule poset} for a minuscule coweight $\omega_i$  is the order filter in the root poset $\Phi^+$ generated by the corresponding simple root $\alpha_i$: \begin{equation}\Lambda_{{\sf G}/{{\sf P}}} := \left\{\alpha \in \Phi^+ : \text{ if } \alpha = \sum_{j=1}^n c_j \alpha_j, \text{ then } c_i \neq 0\right\} = \Phi^+ \setminus \Phi^+_{\langle i \rangle}.\label{eq:subposet}\end{equation}

By the orbit-stabilizer theorem, the minimal coset representatives $w$ of the parabolic quotient $W^{\langle i \rangle}:=W/W_{\langle i \rangle}$ are in bijection with the weights in the orbit $\{ w(\omega_i) : w \in W\}$ of $\omega_i$.  We now follow J.~Stembridge and R.~Proctor to give an explicit combinatorial description of these quotients in the case $\omega_i$ is a minuscule weight.

\medskip

 Fix $w \in W$ with $\len(w) = j$ and let $\mathbf{w} = (s_{k_1},s_{k_2}, \ldots, s_{k_\well})$ be a reduced word for $w$.  Define a partial order $\prec_{\mathbf{w}}$ on $[\well]$ by the transitive closure of the relations \begin{equation} i \prec_{\mathbf{w}} j \text{ if } i<j \text{ and } s_{k_i}s_{k_j} \neq s_{k_j} s_{k_i}.\label{eq:heap}\end{equation}  
This partial ordering defines an ordering on $[\well]$ called a \defn{heap}~\cite{viennot1986heaps,stembridge1996fully}, and hence gives an ordering of the roots in the inversion set $\w$ of $w$.

A \defn{fully commutative} element $w\in W$ is one whose graph $\Red(w)$ of reduced words is connected using only commutations (that is, without braid relations). For any two reduced words of a fully commutative $w$, it is then not difficult to see that the two induced partial orderings on $\w$ are isomorphic.  Therefore, when $w \in W$ is fully commutative, we may unambiguously refer to \textit{the} heap $\w$ of $w$.

\begin{figure}[htbp]
\begin{center}
  \includegraphics[height=1.5in]{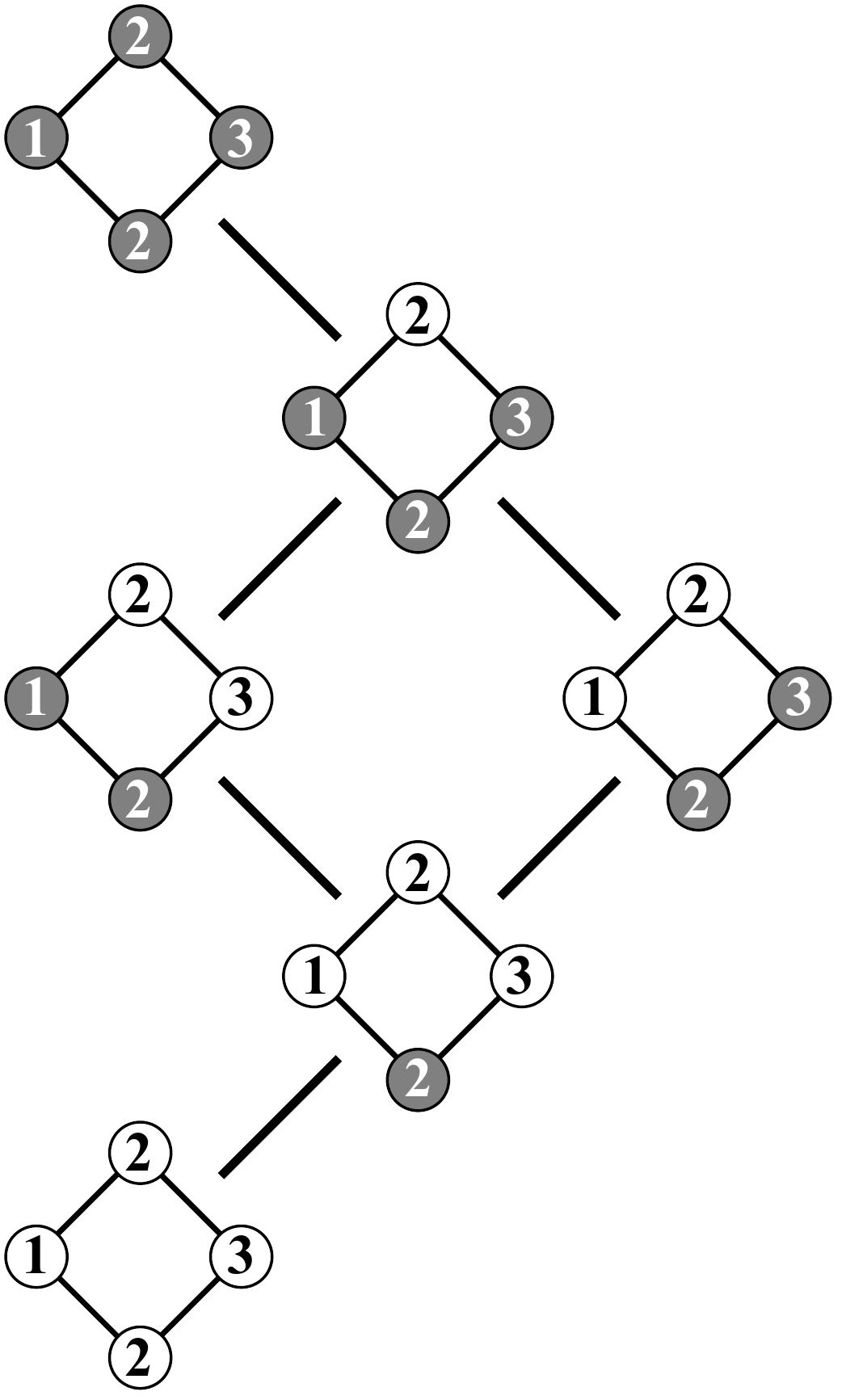}
 \end{center}
\caption{For $W=W(A_3)$, the minuscule weight $\omega_2$ is fixed by the parabolic subgroup $W_{\langle 2 \rangle}$.  The corresponding quotient $W^{\langle 2 \rangle}$ has a fully commutative longest element $\wo^{\langle 2 \rangle}=s_2 s_1 s_3 s_2$, whose heap is ${\sf w}_\circ^{\langle 2 \rangle}=[2] \times [2] \simeq \Lambda_{\Gr(2,4)}$.}
\label{ex:heap33}
\end{figure}

\begin{theorem}[{\cite[Proposition 2.2 and Lemma 3.1]{stembridge1996fully}}]
\label{thm:linandred}
For $w$ fully commutative, there is a bijection between standard tableaux of the heap of $w$ and reduced words for $w$: \[\SYT(\w) \simeq \Red(w).\]  This induces an order isomorphism between the distributive lattice of order ideals of the heap and the weak-order interval $[e,w]$: \[J(\w)\simeq [e,w].\]
\end{theorem}

In~\cite{stembridge1996fully}, J.~Stembridge classified all maximal parabolic quotients whose longest element $\wo^{\langle i \rangle}$ is fully commutative.  This classification is summarized in~\Cref{fig:min_classification}.  When $W$ is a Weyl group, this classification essentially coincides with the classification of minuscule representations of the corresponding Lie algebra.\footnote{The Weyl group is not sensitive to the difference between long and short roots, and so confuses types $B$ and $C$.}  By~\Cref{thm:linandred} and J.~Stembridge's classification, when $\omega_i$ is minuscule the inversion sets of the elements in $W^{\langle i \rangle}$ are \emph{order ideals} in the heap for the longest element $\wo^{\langle i \rangle}$ of $W^{\langle i \rangle}$.  This heap may now be simply described by~\Cref{eq:subposet} as the order filter in $\Phi^+$ generated by $\alpha_i$.

This discussion is summarized by the theorem below.

\begin{theorem}[R.~Proctor~\cite{proctor1984bruhat}]
When $\omega$ is a minuscule coweight, there is an order isomorphism
\begin{align*}
    W^{\sf P} &\simeq J\left(\Lambda_{\GP}\right) \\
    u &\mapsto \uu,
\end{align*}
where $W_{\sf P}:=\left\{w \in W : w(\omega) = \omega\right\}$ and ${\sf P}$ is the corresponding maximal parabolic subgroup of ${\sf G}$.
\label{thm:labelsofels}
\end{theorem}

In particular, for $\omega$ minuscule, the weak order on $W^{\sf P}$ is a distributive lattice.  When $\omega$ is a (co)minuscule weight, we shall also use the term \defn{(co)minuscule} to describe the corresponding flag variety $\GP$ and poset $\Lambda_{\GP}$.

\subsection{Explicit Constructions}
\label{sec:explicit}

We explicitly identify the minuscule posets from~\Cref{fig:names} by giving reduced words for $\wo^{\langle i \rangle}$. The corresponding posets can then be built as heaps using~\Cref{eq:heap}.

\subsubsection{$(A_{n-1},k)$: the \defn{Grassmannian} $\Gr(k,n)$}  In type $A_{n-1}$, any fundamental weight $\omega_k$ is minuscule.    For $W=W(A_{n-1})$, 
  the longest element $\wo^{\langle k \rangle} \in W^{\langle k \rangle}$ has reduced word \[\wo^{\langle k \rangle}=\prod_{j=1}^{n-k}\prod_{i=k-j+1}^{n-j} s_i.\]  The poset $\Lambda_{\Gr(k,n)}$ is commonly described as a $[k]\times[n-k]$ rectangle, represented as the partition $\left((n-k)^k\right)=\underbrace{(n-k,n-k,\ldots,n-k)}_{\text{$k$ parts}}$.  
Thus, as in~\Cref{rem:drawing}, an order ideal in $\Lambda_{\Gr(k,n)}$ (for any $k,n$) may be drawn as a Ferrers shape. The corresponding minuscule variety is a \defn{Grassmannian} $\Gr(k,n)$, a parameter space for $k$-dimensional linear subspaces of $\mathbb{C}^n$.


\subsubsection{$(C_{n},1)$: the \defn{Lagrangian Grassmanian} $\LG(n,2n)$}\label{sec:lg} In type $C_n$, $\omega_1$ is a cominuscule weight.  For $W=W(C_{n})$, 
the longest element $\wo^{\langle 1 \rangle}$ of $W^{\langle 1 \rangle}$ has reduced word \[\wo^{\langle 1 \rangle}=\prod_{i=1}^n \prod_{j=1}^{n-i+1} s_j,\] so that---when drawn as a shifted Ferrers shape as in~\Cref{rem:drawing}---$\Lambda_{\LG(n,2n)}$ is a shifted staircase of order $n$. We write this as the shifted partition $(n,n-1,\ldots,1)_*$.
The corresponding cominuscule variety $\LG(n,2n)$ is a \defn{Lagrangian Grassmannian}. It can be realized as the subvariety of the type $A$ Grassmannian $\Gr(n,2n)$ consisting of those points corresponding to the $n$-dimensional linear subspaces of $\mathbb{C}^{2n}$ that are isotropic with respect to a fixed nondegenerate symplectic form.

\subsubsection{$(D_{n},1)$ and $(D_{n},2)$: the \defn{Even Orthogonal Grassmanian} $\OG(n,2n)$}\label{sec:typeDels}  In type $D_n$, $\omega_1$ and $\omega_2$ are minuscule weights with isomorphic minuscule posets $\Lambda_{\OG(n,2n)}$.  For $W=W(D_{n})$,
write $s_{1,2}(j)=\begin{cases} s_1 & \text{if } j \text{ is odd} \\  s_2 & \text{if } j \text{ is even}\end{cases}.$  The longest elements of $W^{\langle i \rangle}$ for $i \in \{1,2\}$ have reduced words  
\begin{align*}
\wo^{\langle 1 \rangle}&=\prod_{j=1}^n \left( s_{1,2}(j) \prod_{k=3}^{n-j+1} s_k\right)\\
\wo^{\langle 2 \rangle}&=\prod_{j=1}^n \left( s_{1,2}(j+1) \prod_{k=3}^{n-j+1} s_k \right)
\end{align*}
When drawn as a shifted Ferrers shape, the poset $\Lambda_{\mathsf{OG}(n,2n)}$ is a shifted staircase of order $n-1$. 
The \defn{even orthogonal Grassmanian} $\OG(n,2n)$ is the minuscule subvariety of $\Gr(n,2n)$ parametrizing those $n$-dimensional linear subspaces of $\mathbb{C}^{2n}$ that are isotropic with respect to a fixed nondegenerate symmetric bilinear form.

\subsubsection{$(D_n,n)$: the \defn{Even Dimensional Quadric} $\mathbb{Q}^{2n-2}$}  Again in type $D_n$, the weight $\omega_n$ is also minuscule but with poset $\Lambda_{\QQ^{2n-2}}$.  
The corresponding longest element of $W^{\langle n \rangle}$ has reduced word
\begin{align*}
\wo^{\langle n \rangle}&=\left ( \prod_{j=3}^n s_j \right)^{-1} (s_1 s_2) \left( \prod_{j=3}^n s_j \right).
\end{align*}
 The poset $\Lambda_{\QQ^{2n-2}}$ can be compactly described as either the iterated distributive lattice of order ideals $\J^{n-3}([2]\times[2])$, or as the ordinal sum of a chain of length $n-2$, an antichain of size $2$, and a chain of length $n-2$: \[\Lambda_{\QQ^{2n-2}}:=[n-2]\oplus ([1]\sqcup [1]) \oplus [n-2].\] 
The corresponding minuscule variety is a quadric hypersurface $\mathbb{Q}^{2n-2}$.

\subsubsection{$(E_6,\omega_1)$ and $(E_6,\omega_6)$: the \defn{Cayley plane}}

In type $E_6$, $\omega_1$ and $\omega_6$ are minuscule weights with isomorphic minuscule posets $\Lambda_{\mathbb{OP}^2}$.  For $W=W(E_6)$,
the longest elements of $W^{\langle i \rangle}$ for $i \in \{1,6\}$ have reduced words  
\begin{align*}
\wo^{\langle 1 \rangle}&=s_1s_3s_4s_5s_6s_2s_4s_5s_3s_4s_1s_3s_2s_4s_5s_6\\
\wo^{\langle 6 \rangle}&=s_6s_5s_4s_3s_1s_2s_4s_3s_5s_4s_6s_5s_2s_4s_3s_1.
\end{align*}
The corresponding minuscule variety is the \defn{Cayley plane} $\mathbb{OP}^2$, the projective plane over the complexification of the octonions. This space was first constructed by R.~Moufang \cite{moufang}. For further discussion of the Cayley plane, see the thorough exposition of J.~Baez \cite{baez}.

\subsubsection{$(E_7,\omega_1)$: the \defn{Freudenthal variety} $\mathsf{G}_{\omega}(\mathbb{O}^3,\mathbb{O}^6)$}  In type $E_7$, only $\omega_1$ is a minuscule weight.   For $W=W(E_7)$, a reduced word for the longest element of $W^{{\langle 1 \rangle}}$ is \[\wo^{\langle 1 \rangle}=s_1s_3s_4s_5s_6s_2s_5s_4s_3s_1s_7s_6s_5s_4s_3s_2s_5s_4s_6s_5s_7s_6s_2s_4s_3s_1.\]  The poset $\Lambda_{\mathsf{G}_{\omega}(\mathbb{O}^3,\mathbb{O}^6)}$ is the second poset from the left in~\Cref{fig:coincidental_top_half_minuscules}.
We refer to the corresponding minuscule variety as the \defn{Freudenthal variety}. This space parametrizes certain copies of $\mathbb{O}^3$ in $\mathbb{O}^6$, where $\mathbb{O}$ denotes the complexification of the octonions; for details, see the work of J.~Tits \cite[\textsection E.5]{tits}.



\section{Schubert Calculus}
\label{sec:ring}

We now turn to the algebro-geometric context for minuscule posets.
  This section sets up the rings necessary to state~\Cref{thm:main1}, which establishes an equivalence between products in these rings and certain bijections.  The equivalence will follow from the combinatorics of the structure coefficients for the rings, which we review in~\Cref{sec:structure_coeffs}.  

\subsection{Cohomology}

Recall from~\Cref{sec:flag_varieties} that for ${\sf P}$ a parabolic subgroup of ${\sf G}$, the generalized flag variety $\GP$ has the Bruhat decomposition \[\GP = \bigsqcup_{w \in W^{\sf P}} {\sf B}_- w {\sf P} / {\sf P}.\]  For $w \in W^{\sf P}$, the \defn{Schubert variety} is the closures $X_w := \overline{ {\sf B}_- w {\sf P} / {\sf P}}$, while the \defn{opposite Schubert variety} is $X^w := \overline{ {\sf B}- w {\sf P} / {\sf P}}$. Schubert varieties intersect transversally with opposite Schubert varieties, and the intersection $X_u^w := X_u \cap X^w$ is called a \defn{Richardson variety}. 
The \defn{Schubert classes}  $\sigma_w$ are the Poincar\'e duals of the Schubert varieties. Since the Bruhat decomposition is a cell decomposition, the set $\{ \sigma_w \}_{w \in W^{\sf P}}$ is a $\ZZ$-linear basis of the cohomology ring $H^\star(\GP, \ZZ)$.
 As such, any cup product $\sigma_w \cdot \sigma_u$ of basis elements can be expressed in the basis:
\[\sigma_w \cdot \sigma_u = \sum_{v \in W^{\sf P}} c_{w,u}^v \sigma_v.\]
In this setting, $H^\star(\GP, \ZZ)$ is isomorphic to the Chow ring $A^\star(\GP)$ of subvarieties up to rational equivalence, where the ring product is given by transverse intersection.

The Borel homomorphism from $H^\star(\Gr(k,n))$ to the coinvariant ring identifies Schubert classes with Schur functions, and in this case the $c_{w,u}^v$ are known as the \defn{Littlewood-Richardson coefficients}~\cite{lesieur1947problemes}.  This setup therefore generalizes the specific example discussed in~\Cref{sec:philo}.   

For $\GP$ minuscule, H.~Thomas and A.~Yong gave a uniform combinatorial formula for  $c_{w,u}^v$~\cite{thomas2009combinatorial}.     Their formula generalizes M.-P.~Sch\"utzenberger's well-known rule for $\GP=\Gr(k,n)$.   Given a standard tableau $\T \in \SYT({\sf v}/{\sf w})$, there is a \defn{rectification} map (whose definition we defer until~\Cref{sec:rectification}, where it will be given in greater generality) that produces a tableau $\rect(\T) \in \SYT({\sf u'})$, for some $u' \in W^{\sf P}$.

\begin{theorem}[\cite{thomas2009combinatorial}]  For $\GP$ minuscule,
the coefficient $c_{w,u}^v$ equals the number of standard tableaux $\T \in \SYT({\sf v}/{\sf w})$ whose rectification is any fixed standard tableau of shape ${\sf u}$.
\label{thm:coefficients_LR}
\end{theorem}

\subsection{\K-Theory}
\label{sec:k_theory_intro}

\K-theoretic Schubert calculus turns to the Grothendieck ring $K(\GP)$ of algebraic vector bundles over $\GP$ as a richer variant of the ordinary cohomology ring $H^\star(\GP)$. The \K-theory ring $K(\GP)$ has a $\mathbb{Z}$-linear basis given by the classes of the Schubert varieties' structure sheaves $\{[\mathcal{O}_{X_w}]\}_{w \in W^{\sf P}}$.  As before, we have an expansion:
\begin{equation}[\mathcal{O}_{X_w}] \cdot [\mathcal{O}_{X_u}] = \sum_{v \in W^{\sf P}} C_{w,u}^v [\mathcal{O}_{X_v}],\label{eq:coefficients}\end{equation} where now \[(-1)^{|\vv|-|\w|-|\uu|} C_{w,u}^v \in \mathbb{Z}_{\geq 0}\] (as shown in greater generality by M.~Brion~\cite{brion2002positivity}).
These \K-theoretic structure constants generalize their cohomological counterparts---$C_{w,u}^v=c_{w,u}^v$ whenever $|\vv|=|\w|+|\uu|$, but when $|\vv| > |\w| + |\uu|$, $c_{w,u}^v=0$ while $C_{w,u}^v$ can be \emph{nonzero}.

The first part of \Cref{sec:structure_coeffs} is devoted the combinatorics required to generalize~\Cref{thm:coefficients_LR} to a combinatorial model for the \K-theoretic structure constants $C_{w,u}^v$.  But even without such an explicit rule, we can already state a specialized result that allows the determination of the $C_{w,u}^v$ from their cohomological analogues.
%
%

\medskip

When $\sigma_w \cdot \sigma_u$ expands as a multiplicity-free sum of Schubert classes in $H^\star(\GP)$, a result of A.~Knutson determines the corresponding expansion of $[\mathcal{O}_{X_w}] \cdot [\mathcal{O}_{X_u}]$ in $K(\GP)$.  Recall that the \defn{M\"{o}bius function} of a poset $\Po$ is the function $\mu_\Po: \Po \times \Po \to \mathbb{Z}$ uniquely characterized by $\mu_\Po(x,x)=1$ and the fact that for all $x\prec y \in \Po$, \begin{equation}\sum_{x \preceq z \preceq y} \mu_\Po(x,z)=0.\label{eq:mobius}\end{equation}  Given a poset $\Po$, we shall adjoin a minimal element $\hat{0}$ and write $\hat{\mu}_\Po(x):=-\mu_\Po(\hat{0},x)$.

\begin{theorem}[{A.~Knutson~\cite[Theorem 3]{knutson2009frobenius}}]
Suppose
\[\sigma_w \cdot \sigma_u = \sum_{v \in D} \sigma_v\] is a multiplicity-free product in $H^\star(\GP)$, where $D \subseteq W^{\sf P}$ represents the set of $v \in W^{\sf P}$ that appear. Write $\Po:=\{ y \in W^{\sf P} : y \geq v, \text{ for some } v \in D\}$.  Then the corresponding expansion in $K(\GP)$ is
\[[\mathcal{O}_{X_w}] \cdot [\mathcal{O}_{X_u}] = \sum_{y \in \Po} \hat{\mu}_{\Po}(y) [\mathcal{O}_{X_y}].\]
\label{thm:knutson_mult_free}
\end{theorem}

\begin{example}
\ytableausetup{boxsize=.5em}
Continuing~\Cref{ex:prod_sum}, for $a=b=2$, we have 
\begin{align*}\left[\raisebox{\height}{$\mathcal{O}_{\ydiagram{2,2}*[*(gray)]{2,2}}$}\right]^2 = \left[\raisebox{\height}{$\mathcal{O}_{\ydiagram{2,2,2,2}*[*(gray)]{2,2}}$}\right]+\left[\raisebox{\height}{$\mathcal{O}_{\ydiagram{3,2,2,1}*[*(gray)]{2,2}}$}\right]+\left[\raisebox{\height}{$\mathcal{O}_{\ydiagram{4,2,2}*[*(gray)]{2,2}}$}\right]&+\left[\raisebox{\height}{$\mathcal{O}_{\ydiagram{3,3,1,1}*[*(gray)]{2,2}}$}\right]+\left[\raisebox{\height}{$\mathcal{O}_{\ydiagram{4,3,1}*[*(gray)]{2,2}}$}\right]+\left[\raisebox{\height}{$\mathcal{O}_{\ydiagram{4,4}*[*(gray)]{2,2}}$}\right] \\
-\left[\raisebox{\height}{$\mathcal{O}_{\ydiagram{3,2,2,2}*[*(gray)]{2,2}}$}\right]-\left[\raisebox{\height}{$\mathcal{O}_{\ydiagram{3,3,2,1}*[*(gray)]{2,2}}$}\right]-\left[\raisebox{\height}{$\mathcal{O}_{\ydiagram{4,2,2,1}*[*(gray)]{2,2}}$}\right]&-\left[\raisebox{\height}{$\mathcal{O}_{\ydiagram{4,3,1,1}*[*(gray)]{2,2}}$}\right]-\left[\raisebox{\height}{$\mathcal{O}_{\ydiagram{4,3,2}*[*(gray)]{2,2}}$}\right]-\left[\raisebox{\height}{$\mathcal{O}_{\ydiagram{4,4,1}*[*(gray)]{2,2}}$}\right] \\ &+\left[\raisebox{\height}{$\mathcal{O}_{\ydiagram{4,3,2,1}*[*(gray)]{2,2}}$}\right].\end{align*}
\ytableausetup{boxsize=normal}
In fact, as we will prove in a forthcoming paper, the square of a structure sheaf indexed by a rectangle is always multiplicity-free.
\label{ex:prod_sum2}
\end{example}

We will be particularly interested in multiplicity-free products in $K(\GP)$. Conveniently, by~\Cref{thm:knutson_mult_free}, to determine multiplicity-freeness in $K(\GP)$, it suffices to check the corresponding statement in $H^\star(\GP)$ and then apply the theorem.

\begin{remark}\rm
\ytableausetup{boxsize=.5em}
It is \emph{not} the case that a multiplicity-free product in cohomology necessarily yields a multiplicity-free product in \K-theory.  For example, in $\Gr(3,6)$, we have
\begin{align*} \sigma_{\ydiagram{2,1}} \cdot \sigma_{\ydiagram{1,1}} = \sigma_{\ydiagram{2,2,1}}+\sigma_{\ydiagram{3,1,1}}+\sigma_{\ydiagram{3,2}}, \end{align*}
but
\begin{align*} \left[\raisebox{\height}{$\mathcal{O}_{\ydiagram{2,1}}$}\right]\cdot \left[\raisebox{\height}{$\mathcal{O}_{\ydiagram{1,1}}$}\right] = \left[\raisebox{\height}{$\mathcal{O}_{\ydiagram{2,2,1}}$}\right]+\left[\raisebox{\height}{$\mathcal{O}_{\ydiagram{3,1,1}}$}\right]+\left[\raisebox{\height}{$\mathcal{O}_{\ydiagram{3,2}}$}\right]-2\left[\raisebox{\height}{$\mathcal{O}_{\ydiagram{3,2,1}}$}\right]. \end{align*}
However, if the cohomological product is multiplicity-free \emph{with a single term}, i.e., 
\[
\sigma_w \cdot \sigma_u = \sigma_v,
\]
then it is immediate from \Cref{thm:knutson_mult_free} that the \K-theoretic product will also have a single term:
\[ [\mathcal{O}_{X_w}] \cdot [\mathcal{O}_{X_u}] = [\mathcal{O}_{X_v}]. \]
\ytableausetup{boxsize=normal}
\label{rem:notnecessarily}
\end{remark}

\section{Combinatorics of Structure Coefficients}
\label{sec:structure_coeffs}

In this section we introduce the combinatorial tools developed in the study of \K-theoretic Schubert calculus by H.~Thomas and A.~Yong~\cite{thomas2009jeu}, E.~Clifford, H.~Thomas, and A.~Yong~\cite{clifford2014k}, and A.~Buch and M.~Samuel~\cite{buch2014k}.  The main result we wish to review is a combinatorial formula for the \K-theoretic structure coefficients $C_{w,u}^v$ in the style of~\Cref{thm:coefficients_LR}.    The role of standard tableaux is now played by \emph{increasing tableaux}.

\subsection{Increasing Tableaux}

Following~\cite{buch2014k}, we generalize the language of increasing tableaux from the introduction.
Fix a finite poset $\Po$ with order relation $\prec$ and an alphabet $\mathcal{A}$ (assume the symbol $\bullet \not \in \mathcal{A}$).   For a skew shape $\vv/\w$, a \defn{tableau of shape $\vv/\w$ on the alphabet $\mathcal{A}$} is a map $\T: \vv/\w \to \mathcal{A}$.

\begin{definition}
\label{def:increasing_tableau}
  Let $\mathcal{A}$ be a totally-ordered alphabet with order relation $<$.  An \defn{increasing tableau of shape $\vv/\w$ on the alphabet $\mathcal{A}$} is a strictly order-preserving map $\T: \vv/\w \to \mathcal{A}$, that is if $\alpha\prec \beta$ in $\vv/\w$ then $\T(\alpha)<\T(\beta)$.  We write $\IT^{\mathcal{A}}(\vv/\w)$ for the set of all such maps.
\end{definition}

For two disjoint alphabets $\mathcal{A},\mathcal{B}$ with $\T \in \IT^{\mathcal{B}}(\w)$ (``$\mathcal{B}$'' for below) and $\U \in \IT^{\mathcal{A}}(\vv/\w)$ (``$\mathcal{A}$'' for above), we write $\T \sqcup \U$ for the increasing tableau in $\IT^{\mathcal{B}\sqcup \mathcal{A}}(\vv)$, where $\mathcal{B}\sqcup \mathcal{A}$ is totally ordered so that $b<a$ for all $b \in \mathcal{B}$ and $a \in \mathcal{A}$.  We define $\IT(\vv/\w):=\bigcup_{\kk=1}^\infty  \IT^{[\kk]}(\vv/\w)$ and set $\T^{\rm min}_{\vv/\w}$ to be the componentwise minimal increasing tableau in $\IT(\vv/\w)$.  We call $\T^{\rm min}_{\vv/\w}$ the \defn{minimal increasing tableau of shape $\vv/\w$} (see~\Cref{fig:double}(a) for an example); it will play an important role in the sequel.

\medskip

In special cases, the notions of increasing tableaux and $\Po$-partitions are simply related, as was first observed in~\cite{DPS}.

\begin{proposition}[{\cite[Theorem 4.1]{DPS}}]
For a ranked poset $\Po$ with all maximal chains of the same length $\hgt(\Po)$, there is a bijection between plane partitions of height $p$ and increasing tableaux in the alphabet $[1,\p+\hgt(\Po)]$: $$\PP^{[\p]}(\Po) \simeq \IT^{[\p+\hgt(\Po)]}(\Po).$$
\label{prop:bij_PP_IT}
\end{proposition}
\begin{proof}
With our conventions, a bijection from plane partitions to increasing tableaux is evidently given by adding $i$ to the labels of the elements on the $i$th rank.
\end{proof}

Since all of the posets in~\Cref{thm:main_thm1} are of the required form, by~\Cref{prop:bij_PP_IT} we may henceforth deal only with increasing tableaux.  The significant advantage that increasing tableaux enjoy over $\Po$-partitions is that increasing tableaux are equipped with a well-developed theory of \K-theoretic jeu-de-taquin~\cite{thomas2009combinatorial,thomas2009jeu,clifford2014k,buch2014k}, a theory we now explore.

\subsection{Jeu-de-Taquin and Other Games}
\label{sec:ktheorycomb}

\subsubsection{Jeu-de-Taquin}
\label{sec:jdt}
Given a shape $\vv/\w \subseteq \Po$, a tableau $\T$ of shape $\vv/\w$ on $\mathcal{A}$, and $a \in \mathcal{A}$, we let
\[\T_a:=\{\alpha \in \vv/\w : \alpha \text{ covers or is covered by some } \beta \text{ for which } \T(\beta)=a\}.\]
For two letters $a,b \in \mathcal{A}$, we may ``exchange'' them in $\T$ to obtain a new tableau 
\[ \swap_{a,b}(\T)(\alpha) := \begin{cases} a & \text{if } \T(\alpha)=b \text{ and } \alpha \in \T_a; \\  b & \text{if } \T(\alpha)=a \text{ and } \alpha \in \T_b; \\ \T(\alpha) & \text{ otherwise.} \end{cases}\]

If we remove a set of maximal elements from $\w$ to obtain $\w'$, we may extend the tableau $\T$ on $\vv/\w$ to a tableau $\T'$ of shape $\vv/\w'$ by setting $\T'(\alpha):=\bullet$ for $\alpha \in \w/\w'$.  Given an increasing tableau $\T$ of shape $\vv/\w$ on the totally-ordered alphabet $\mathcal{A}$, the \defn{slide} of $\T$ into $\w/\w'$ is given by 
\[ \jdt_{\w/\w'}(\T) := \left(\prod_{a \in \mathcal{A}} \swap_{a,\bullet} \right)(\T'),  \]
where the product is in the given linear ordering for $\mathcal{A}$, and
where we restrict the domain of $\jdt_{\w/\w'}(\T)$ to the subset $\vv'/\w':=\{\alpha \subseteq \vv/\w' : \jdt_{\w/\w'}(\T)(\alpha)\neq \bullet\}$.  This procedure is invertible and hence bijective.

\begin{example}\rm The following illustration is an example of a slide for $\mathcal{A}=1<2<3<4<5<6$ (see~\Cref{ex:ex3} for several other illustrations).  Here, we supress mention of $\swap_{4,\bullet}$ and $\swap_{5,\bullet}$, as they act trivially in this example.

\begin{align*}\tiny
 \raisebox{-0.5\height}{
  \begin{tikzpicture}[scale=.3]
    \draw[thick] (-1 cm,-1) -- (2 cm,2) -- (-1cm,5);
        \draw[thick] (0 cm,0) -- (-1 cm,1);
        \draw[thick] (1 cm,1) -- (0 cm,2) -- (1,3);
        \draw[thick] (-1 cm,1) -- (0 cm,2)--(-1 cm,3)--(0,4);
        \draw[thick,solid,fill=white] (-1cm,-1) circle (.5cm) node {$\bullet$};
    \draw[thick,solid,fill=white] (0cm,0) circle (.5cm) node {1};
    \draw[thick,solid,fill=white] (1cm,1) circle (.5cm) node {2};
    \draw[thick,solid,fill=white] (2cm,2) circle (.5cm) node {4};
        \draw[thick,solid,fill=white] (-1cm,1) circle (.5cm) node {2};
        \draw[thick,solid,fill=white] (0cm,2) circle (.5cm) node {3};
        \draw[thick,solid,fill=white] (1cm,3) circle (.5cm) node {6};
        \draw[thick,solid,fill=white] (-1cm,3) circle (.5cm) node {};
        \draw[thick,solid,fill=white] (-1cm,5) circle (.5cm) node {};
        \draw[thick,solid,fill=white] (0cm,4) circle (.5cm) node {};
  \end{tikzpicture}}\xrightarrow{\swap_{1,\bullet}} \raisebox{-0.5\height}{
  \begin{tikzpicture}[scale=.3]
    \draw[thick] (-1 cm,-1) -- (2 cm,2) -- (-1cm,5);
        \draw[thick] (0 cm,0) -- (-1 cm,1);
        \draw[thick] (1 cm,1) -- (0 cm,2) -- (1,3);
        \draw[thick] (-1 cm,1) -- (0 cm,2)--(-1 cm,3)--(0,4);
        \draw[thick,solid,fill=white] (-1cm,-1) circle (.5cm) node {1};
    \draw[thick,solid,fill=white] (0cm,0) circle (.5cm) node {$\bullet$};
    \draw[thick,solid,fill=white] (1cm,1) circle (.5cm) node {2};
    \draw[thick,solid,fill=white] (2cm,2) circle (.5cm) node {4};
        \draw[thick,solid,fill=white] (-1cm,1) circle (.5cm) node {2};
        \draw[thick,solid,fill=white] (0cm,2) circle (.5cm) node {3};
        \draw[thick,solid,fill=white] (1cm,3) circle (.5cm) node {6};
        \draw[thick,solid,fill=white] (-1cm,3) circle (.5cm) node {};
        \draw[thick,solid,fill=white] (-1cm,5) circle (.5cm) node {};
        \draw[thick,solid,fill=white] (0cm,4) circle (.5cm) node {};
  \end{tikzpicture}}\xrightarrow{\swap_{2,\bullet}} \raisebox{-0.5\height}{
  \begin{tikzpicture}[scale=.3]
    \draw[thick] (-1 cm,-1) -- (2 cm,2) -- (-1cm,5);
        \draw[thick] (0 cm,0) -- (-1 cm,1);
        \draw[thick] (1 cm,1) -- (0 cm,2) -- (1,3);
        \draw[thick] (-1 cm,1) -- (0 cm,2)--(-1 cm,3)--(0,4);
        \draw[thick,solid,fill=white] (-1cm,-1) circle (.5cm) node {1};
    \draw[thick,solid,fill=white] (0cm,0) circle (.5cm) node {2};
    \draw[thick,solid,fill=white] (1cm,1) circle (.5cm) node {$\bullet$};
    \draw[thick,solid,fill=white] (2cm,2) circle (.5cm) node {4};
        \draw[thick,solid,fill=white] (-1cm,1) circle (.5cm) node {$\bullet$};
        \draw[thick,solid,fill=white] (0cm,2) circle (.5cm) node {3};
        \draw[thick,solid,fill=white] (1cm,3) circle (.5cm) node {6};
        \draw[thick,solid,fill=white] (-1cm,3) circle (.5cm) node {};
        \draw[thick,solid,fill=white] (-1cm,5) circle (.5cm) node {};
        \draw[thick,solid,fill=white] (0cm,4) circle (.5cm) node {};
  \end{tikzpicture}}\xrightarrow{\swap_{3,\bullet}} \raisebox{-0.5\height}{
  \begin{tikzpicture}[scale=.3]
    \draw[thick] (-1 cm,-1) -- (2 cm,2) -- (-1cm,5);
        \draw[thick] (0 cm,0) -- (-1 cm,1);
        \draw[thick] (1 cm,1) -- (0 cm,2) -- (1,3);
        \draw[thick] (-1 cm,1) -- (0 cm,2)--(-1 cm,3)--(0,4);
        \draw[thick,solid,fill=white] (-1cm,-1) circle (.5cm) node {1};
    \draw[thick,solid,fill=white] (0cm,0) circle (.5cm) node {2};
    \draw[thick,solid,fill=white] (1cm,1) circle (.5cm) node {3};
    \draw[thick,solid,fill=white] (2cm,2) circle (.5cm) node {4};
        \draw[thick,solid,fill=white] (-1cm,1) circle (.5cm) node {3};
        \draw[thick,solid,fill=white] (0cm,2) circle (.5cm) node {$\bullet$};
        \draw[thick,solid,fill=white] (1cm,3) circle (.5cm) node {6};
        \draw[thick,solid,fill=white] (-1cm,3) circle (.5cm) node {};
        \draw[thick,solid,fill=white] (-1cm,5) circle (.5cm) node {};
        \draw[thick,solid,fill=white] (0cm,4) circle (.5cm) node {};
  \end{tikzpicture}} \xrightarrow{\swap_{6,\bullet}} \raisebox{-0.5\height}{
  \begin{tikzpicture}[scale=.3]
    \draw[thick] (-1 cm,-1) -- (2 cm,2) -- (-1cm,5);
        \draw[thick] (0 cm,0) -- (-1 cm,1);
        \draw[thick] (1 cm,1) -- (0 cm,2) -- (1,3);
        \draw[thick] (-1 cm,1) -- (0 cm,2)--(-1 cm,3)--(0,4);
        \draw[thick,solid,fill=white] (-1cm,-1) circle (.5cm) node {1};
    \draw[thick,solid,fill=white] (0cm,0) circle (.5cm) node {2};
    \draw[thick,solid,fill=white] (1cm,1) circle (.5cm) node {3};
    \draw[thick,solid,fill=white] (2cm,2) circle (.5cm) node {4};
        \draw[thick,solid,fill=white] (-1cm,1) circle (.5cm) node {3};
        \draw[thick,solid,fill=white] (0cm,2) circle (.5cm) node {6};
        \draw[thick,solid,fill=white] (1cm,3) circle (.5cm) node {$\bullet$};
        \draw[thick,solid,fill=white] (-1cm,3) circle (.5cm) node {};
        \draw[thick,solid,fill=white] (-1cm,5) circle (.5cm) node {};
        \draw[thick,solid,fill=white] (0cm,4) circle (.5cm) node {};
  \end{tikzpicture}}\end{align*}
\label{ex:jdt_slide}
\end{example}

When $\T \in \SYT(\vv/\w)$ and $\w/\w'$ is a single box, this process recovers the usual notion of \defn{jeu-de-taquin} introduced by M.-P.~Sch\"{u}tzenberger.  Two tableaux $\T$ and $\T'$ are called \defn{jeu-de-taquin equivalent} if they are related by a sequence of slides and inverse slides.

\subsubsection{Rectification}
\label{sec:rectification}
Let $\T \in \IT^{[\kk]}(\w)$ be an increasing tableau and set \[\w_i:=\{\alpha \in \w : \T(\alpha)\leq \kk-i\},\] so that $\w_0=\w$ and $\w_\kk=\emptyset$.  The \defn{$\T$-rectification} of the skew increasing tableau $\U \in \IT^{\mathcal{A}}(\vv/\w)$ is the straight-shaped tableau 
\[ \rect_{\T}(\U) := \left(\prod_{i=0}^{\kk-1} \jdt_{\w_i/\w_{i+1}} \right) (\U).\]In other words, the $\T$-rectification of $\U$ uses $\T$ to determine the \textit{rectification order}: which elements of $\w$ become $\bullet$s at each stage in the rectification. \Cref{ex:jdt_slide} is an example of rectification, as is the sequence of slides in \Cref{sec:Rectangles_and_Trapezoids}. In the latter example, the rectification order is determined by the minimal tableau of shifted shape $\w = (3,2,1)_*$.

Given $\U \in \IT^{\mathcal{A}}(\vv/\w)$ and $\T, \T'\in \IT^{[\kk]}(\w)$, it is possible that $\rect_{{\sf T}}(\U)\neq \rect_{{\sf T'}}(\U)$. A \defn{unique rectification target (URT)} is an increasing tableau ${\sf R}$ of straight shape such that if $\rect_{\sf T}(\U)={\sf R}$ for \emph{some} ${\sf T} \in \IT({\sf w})$, then $\rect_{\T'}(\U)={\sf R}$ for \emph{all} ${\sf T'} \in \IT({\sf w})$.  In such cases where the tableau prescribing rectification order does not matter, we may simply write $\rect(\U) = R$.  

\begin{theorem}[{\cite[Theorem 3.12]{buch2014k}}] For ${\sf w}$ any straight shape in a minuscule poset, the minimal tableau $\T^{\rm min}_{\sf w}$ is a URT.
\label{thm:urt}
\end{theorem}

For $\GP$ minuscule and $w,u \leq v \in W^{\sf P}$, define
\begin{equation}
 	\R_{w,u}^v = \R_{\w,\uu}^\vv := \big\{ \T \in \IT(\vv/\w) : \rect(\T)=\T^{\rm min}_\uu \big\}
\label{eq:setr}
\end{equation}
to be the set of increasing tableaux of shape $\vv/\w$ that rectify to the minimal tableau of shape $\uu$.  We can now state A.~Buch and M.~Samuel's elegant combinatorial rule---building on seminal work of H.~Thomas and A.~Yong~\cite{thomas2009jeu} to generalize~\Cref{thm:coefficients_LR}---for the structure coefficients $C_{w,u}^v$ of~\Cref{eq:coefficients}.

\begin{theorem}[{\cite[Corollary 4.8]{buch2014k}}]  For $\GP$ minuscule, \[(-1)^{|\vv|-|\w|-|\uu|} C_{w,u}^v = \Big| \R_{w,u}^v \Big|.\]
\label{thm:k_theory_comb_structure}
\end{theorem}

\subsubsection{The Infusion Involution}
\label{sec:infusion_involution}

Instead of discarding the rectification order $\T$ when performing rectification, we can consider what happens to the \emph{pair} $(\T,\U)$ as we move ${\sf U}$ past ${\sf T}$.  We keep track of the two tableaux using two disjoint alphabets: $[\kk]$, and $[\overline{\kk}]:=\{\overline{1}<\overline{2}<\cdots<\overline{k}\}$.  
For $\U \in \IT^{[\kk]}(\vv/\w)$, we will write $\overline{\U}$ to denote the increasing tableau in $\IT^{[\overline{\kk}]}(\vv/\w)$ obtained by sending $i \mapsto \overline{i}$.

Let ${\sf T} \in \IT^{[i]}(\w)$ and ${\sf U} \in \IT^{[j]}(\vv/\w)$.  Informally, we will glue the tableau $\overline{\T}$ on the alphabet $[\overline{i}]$ to the bottom of the tableau $\U$ on $[j]$ and then slide one alphabet past the other, so that the total ordering \[ [\overline{i}]\sqcup [j]= \overline{1}<\overline{2}<\cdots<\overline{i}<1<2<\cdots<j\] becomes \[[j] \sqcup [\overline{i}] = 1<2<\cdots<j<\overline{1}<\overline{2}<\cdots<\overline{i}.\]
Formally, the \defn{infusion involution} of $(\T,\U) \in \IT^{[i]}(\w) \times \IT^{[j]}(\vv/\w)$ is the pair of tableaux $(\U',\T') \in \IT^{[j]}(\uu) \times \IT^{[i]}(\vv/\uu)$ defined by \[\U'\sqcup \overline{\T'}=\left( \prod_{a=1}^i \prod_{b=j}^1 \swap_{\overline{a},b} \right) (\overline{\T} \sqcup \U).\] See \Cref{fig:infinv} for an example.

\begin{theorem}[{\cite[Theorem 3.1]{thomas2009jeu}}]
Infusion is an involution. That is,
for ${\sf T} \in \IT(\w)$ and ${\sf U} \in \IT(\vv/\w)$, we have \[(\Kinf \circ \Kinf)({\sf T},{\sf U}) = ({\sf T},{\sf U}).\]
\label{thm:infusion_involution}
\end{theorem}

\begin{figure}
\[\vcenter{\hbox{\protect\begin{tikzpicture}[scale=.5]
	\node (a) at (0, 0) {};
	\node (b) at (1, 1) {};
	\node (c) at (0, 2) {};
	\node (d) at (2, 2) {};
	\node (e) at (1, 3) {};
	\node (f) at (0, 4) {};
	
	\node (g) at (3, 3) {};
	\node (h) at (4, 4) {};
	\node (i) at (5, 5) {};
	\node (j) at (6, 6) {};
	
	\node (k) at (2, 4) {};
	\node (l) at (3, 5) {};
	\node (m) at (4, 6) {};
	\node (n) at (5, 7) {};
	
	\node (o) at (1, 5) {};
	\node (p) at (2, 6) {};
	\node (q) at (3, 7) {};
	\node (r) at (4, 8) {};
	
	\node (s) at (0, 6) {};
	\node (t) at (1, 7) {};
	\node (u) at (2, 8) {};
	\node (v) at (3, 9) {};
	
	\draw (a) -- (b) -- (d) -- (g);
	\draw (b) -- (c) -- (e) -- (k);
	\draw (d) -- (e) -- (f) -- (o);
	
	\draw (g) -- (h) -- (i) -- (j);
	\draw (k) -- (l) -- (m) -- (n);
	\draw (o) -- (p) -- (q) -- (r);
	\draw (s) -- (t) -- (u) -- (v);
	\draw (g) -- (k) -- (o) -- (s);
	\draw (h) -- (l) -- (p) -- (t);
	\draw (i) -- (m) -- (q) -- (u);
	\draw (j) -- (n) -- (r) -- (v);
	
	\draw[thick,fill=white] (a) circle [radius=0.5cm] node {\textcolor{red}{$\overline{1}$}};
	\draw[thick,fill=white] (b) circle [radius=0.5cm] node {\textcolor{red}{$\overline{2}$}};
	\draw[thick,fill=white] (c) circle [radius=0.5cm] node {\textcolor{red}{$\overline{3}$}};
	\draw[thick,fill=white] (d) circle [radius=0.5cm] node {\textcolor{red}{$\overline{3}$}};
	\draw[thick,fill=white] (e) circle [radius=0.5cm] node {\textcolor{red}{$\overline{4}$}};
	\draw[thick,fill=white] (f) circle [radius=0.5cm] node {\textcolor{red}{$\overline{5}$}};
	
	\draw[thick,fill=white] (g) circle [radius=0.5cm] node {1};
	\draw[thick,fill=white] (h) circle [radius=0.5cm] node {3};
	\draw[thick,fill=white] (i) circle [radius=0.5cm] node {5};
	\draw[thick,fill=white] (j) circle [radius=0.5cm] node {6};
	\draw[thick,fill=white] (k) circle [radius=0.5cm] node {3};
	\draw[thick,fill=white] (l) circle [radius=0.5cm] node {5};
	\draw[thick,fill=white] (m) circle [radius=0.5cm] node {7};
	\draw[thick,fill=white] (n) circle [radius=0.5cm] node {8};
	\draw[thick,fill=white] (o) circle [radius=0.5cm] node {4};
	\draw[thick,fill=white] (p) circle [radius=0.5cm] node {7};
	\draw[thick,fill=white] (q) circle [radius=0.5cm] node {8};
	\draw[thick,fill=white] (r) circle [radius=0.5cm] node {\scalebox{.9}{10}};
	\draw[thick,fill=white] (s) circle [radius=0.5cm] node {6};
	\draw[thick,fill=white] (t) circle [radius=0.5cm] node {9};
	\draw[thick,fill=white] (u) circle [radius=0.5cm] node {\scalebox{.9}{10}};
	\draw[thick,fill=white] (v) circle [radius=0.5cm] node {\scalebox{.9}{11}};
\end{tikzpicture}}} \mapsto
\vcenter{\hbox{\protect\begin{tikzpicture}[scale=.5]
	\node (a) at (0, 0) {};
	\node (b) at (1, 1) {};
	\node (c) at (0, 2) {};
	\node (d) at (2, 2) {};
	\node (e) at (1, 3) {};
	\node (f) at (0, 4) {};
	
	\node (g) at (3, 3) {};
	\node (h) at (4, 4) {};
	\node (i) at (5, 5) {};
	\node (j) at (6, 6) {};
	
	\node (k) at (2, 4) {};
	\node (l) at (3, 5) {};
	\node (m) at (4, 6) {};
	\node (n) at (5, 7) {};
	
	\node (o) at (1, 5) {};
	\node (p) at (2, 6) {};
	\node (q) at (3, 7) {};
	\node (r) at (4, 8) {};
	
	\node (s) at (0, 6) {};
	\node (t) at (1, 7) {};
	\node (u) at (2, 8) {};
	\node (v) at (3, 9) {};
	
	\draw (a) -- (b) -- (d) -- (g);
	\draw (b) -- (c) -- (e) -- (k);
	\draw (d) -- (e) -- (f) -- (o);
	
	\draw (g) -- (h) -- (i) -- (j);
	\draw (k) -- (l) -- (m) -- (n);
	\draw (o) -- (p) -- (q) -- (r);
	\draw (s) -- (t) -- (u) -- (v);
	\draw (g) -- (k) -- (o) -- (s);
	\draw (h) -- (l) -- (p) -- (t);
	\draw (i) -- (m) -- (q) -- (u);
	\draw (j) -- (n) -- (r) -- (v);
	
	\draw[thick,fill=white] (a) circle [radius=0.5cm] node {1};
	\draw[thick,fill=white] (b) circle [radius=0.5cm] node {3};
	\draw[thick,fill=white] (c) circle [radius=0.5cm] node {4};
	\draw[thick,fill=white] (d) circle [radius=0.5cm] node {4};
	\draw[thick,fill=white] (e) circle [radius=0.5cm] node {6};
	\draw[thick,fill=white] (f) circle [radius=0.5cm] node {8};
	
	\draw[thick,fill=white] (g) circle [radius=0.5cm] node {5};
	\draw[thick,fill=white] (h) circle [radius=0.5cm] node {6};
	\draw[thick,fill=white] (i) circle [radius=0.5cm] node {8};
	\draw[thick,fill=white] (j) circle [radius=0.5cm] node {\scalebox{.9}{10}};
	\draw[thick,fill=white] (k) circle [radius=0.5cm] node {7};
	\draw[thick,fill=white] (l) circle [radius=0.5cm] node {8};
	\draw[thick,fill=white] (m) circle [radius=0.5cm] node {\scalebox{.9}{11}};
	\draw[thick,fill=white] (n) circle [radius=0.5cm] node {\textcolor{red}{$\overline{3}$}};
	\draw[thick,fill=white] (o) circle [radius=0.5cm] node {9};
	\draw[thick,fill=white] (p) circle [radius=0.5cm] node {10};
	\draw[thick,fill=white] (q) circle [radius=0.5cm] node {\textcolor{red}{$\overline{2}$}};
	\draw[thick,fill=white] (r) circle [radius=0.5cm] node {\textcolor{red}{$\overline{4}$}};
	\draw[thick,fill=white] (s) circle [radius=0.5cm] node {\scalebox{.9}{11}};
	\draw[thick,fill=white] (t) circle [radius=0.5cm] node {\textcolor{red}{$\overline{1}$}};
	\draw[thick,fill=white] (u) circle [radius=0.5cm] node {\textcolor{red}{$\overline{3}$}};
	\draw[thick,fill=white] (v) circle [radius=0.5cm] node {\textcolor{red}{$\overline{5}$}};
\end{tikzpicture}}}\]
\caption{The infusion involution. On the left, $\U$ is shown with black entries and $\overline{\T}$ shown with red, barred entries. The infusion involution of $(\T,\U)$ is $(\U',\T')$, where $\U'$ is the straight tableau in black on the right and $\overline{\T}'$ is the skew tableau in red on the right. }\label{fig:infinv}
\end{figure}
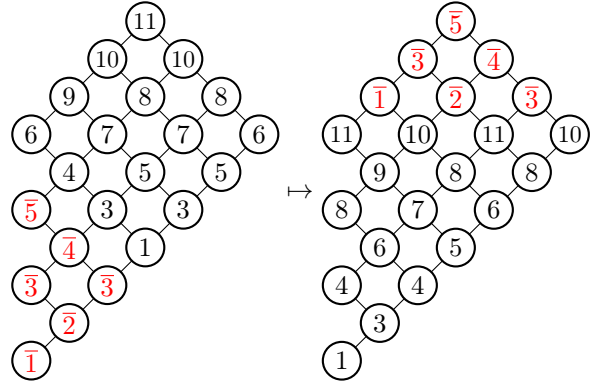


\subsection{Relations and Equivalence}\label{sec:kknuth}

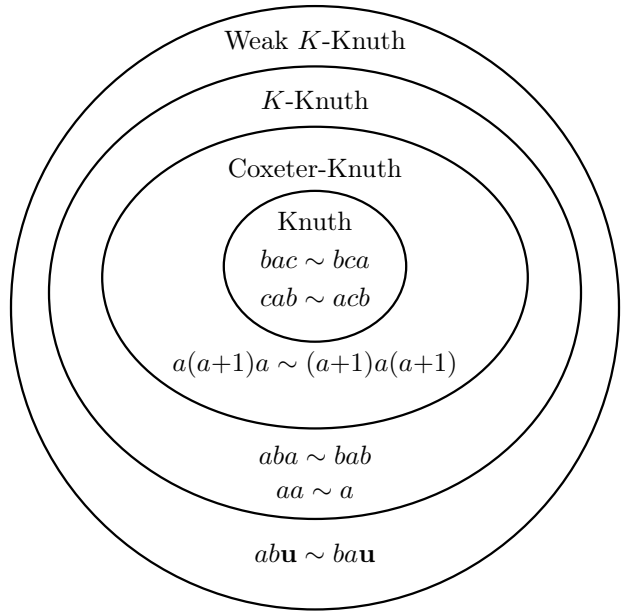
\begin{figure}[htbp]
\begin{tikzpicture}
\draw[line width=0.3mm] (0,0) circle [x radius=1.2cm, y radius=1cm];
\draw[line width=0.3mm] (0,-.15) circle [x radius=2.8cm, y radius=2cm];
\draw[line width=0.3mm] (0,-.35) circle [x radius=3.5cm, y radius=3cm];
\draw[line width=0.3mm] (0,-.55) circle [x radius=4cm, y radius=4cm];
\draw (0,.1) node {$bac \sim bca$};
\draw (0,-.4) node {$cab \sim acb$};
\draw (0,-1.3) node {$a(a{+}1)a \sim (a{+}1)a(a{+}1)$};
\draw (0,-2.5) node {$aba \sim bab$};
\draw (0,-3) node {$aa \sim a$};
\draw (0,-3.8) node {$ab \mathbf{u} \sim ba \mathbf{u}$};
\draw (0,.6) node {Knuth};
\draw (0,1.3) node {Coxeter-Knuth};
\draw (0,2.2) node {\K-Knuth};
\draw (0,3) node {Weak \K-Knuth};
\end{tikzpicture}
\caption{(Standard) Knuth-like relations, where $a<b<c$ are distinct positive integers and $\mathbf{u}$ is a word of positive integers.}
\label{fig:knuth_rels}
\end{figure}

For $\T$ a standard or increasing tableau (in English tableau orientation, as in \Cref{fig:thmstanley} or \Cref{fig:double}), let $\rhow(\T)$ be the \defn{column reading word} obtained by reading the entries in the columns of $\T$ from left to right and bottom to top; where it will not cause confusion, we will abbreviate this to \defn{reading word}.   We wish to consider the set of words on the alphabet of positive integers, up to Knuth, Coxeter-Knuth, \K-Knuth, or weak \K-Knuth equivalences---whose defining relations are given in~\Cref{fig:knuth_rels}.  We note that our ``Knuth equivalence'' in~\Cref{fig:knuth_rels} acts only on triples of \emph{distinct} letters; it is therefore weaker than the usual notion from the theory of semistandard tableaux, although it agrees with the usual notion for reading words of standard tableaux.

Remarkably, as summarized in~\Cref{thm:jdt_and_knuth_equiv}, these relations on reading words of tableaux exactly mirror jeu-de-taquin slides on the tableaux themselves. As in~\Cref{rem:drawing} and~\Cref{sec:explicit}, recall that ``Ferrers tableau'' means a tableau of straight shape in $\Lambda_{\Gr(k,n)}$, while ``shifted Ferrers tableau'' means a tableau of straight shape in $\Lambda_{\OG(n,2n)}$.

\begin{theorem}[{\cite[Theorems~6.2 and~7.8]{buch2014k}}]
\begin{align*}\text{Two increasing (skew)} \left\{\begin{matrix} \text{Ferrers} \\ \text{shifted Ferrers} \end{matrix}\right\} \text{ tableaux } \T,\T' \text{ are jeu-de-taquin equivalent}  \\
 \text{ if and only if } \rhow(\T) \text{ and } \rhow(\T') \text{ are } \left\{\begin{matrix} \text{\K-Knuth} \\ \text{weakly \K-Knuth} \end{matrix}\right\} \text{ equivalent}.
\end{align*}
\label{thm:jdt_and_knuth_equiv}
\end{theorem}

The following proposition records two facts about \K-Knuth equivalence for use in~\Cref{sec:b_rect_and_trap}.

\begin{proposition}[{\cite[Theorem~6.1]{thomas2009jeu}} and {\cite[Lemma~5.4 and~Corollary~6.8]{buch2014k}}]
\gap
\begin{enumerate}
\item The longest strictly increasing subsequences of \K-Knuth equivalent words have the same length.
\item \label{item:linc} The length of the first row of an increasing Ferrers tableau $\T$ of straight shape is the length of the longest strictly increasing subsequence of $\rhow(\T)$.
\end{enumerate}
\label{prop:kknuth}
\end{proposition}

The \defn{doubling} $\T^D$ of a shifted Ferrers tableau $\T$ is the Ferrers tableau obtained by reflecting $\T$ across the shifted diagonal---note that the shifted diagonal itself is \emph{not} duplicated~\cite[Section 7.1]{buch2014k}.  This construction is illustrated in~\Cref{fig:double}.

\begin{figure}[ht]
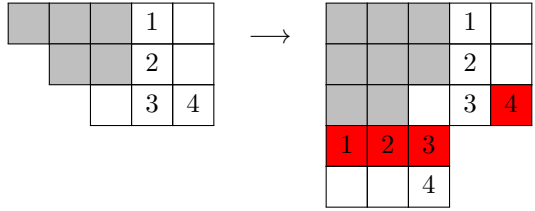

\begin{description}
\item[(a)] For $\uu=(3,2,1)_*$, the minimal increasing tableau $\T_\uu^{\text{min}}$ and $(\T_\uu^{\text{min}})^D$:
\[
\begin{ytableau}
1 & 2 & 3\\
\none & 3 & 4\\
\none & \none & 5
\end{ytableau}
\ \ \ \longrightarrow \ \ \ 
\begin{ytableau}
1& 2 & 3\\
2 & 3 & 4\\
3 & 4 & 5\\
\end{ytableau}
\]
\item[(b)] A partially filled skew shape $\widetilde U$ and its doubling $\widetilde \U^D$:
\[
\begin{ytableau}
*(lightgray) \blank &
*(lightgray) \blank & *(lightgray) \blank & 1 & \\
\none &
*(lightgray) \blank & *(lightgray) \blank & 2 & \\
\none & \none  & \blank & 3 & 4
\end{ytableau}
\ \ \ \longrightarrow \ \ \ 
\begin{ytableau}
*(lightgray) \blank & *(lightgray) \blank & *(lightgray) \blank & 1 &\\
*(lightgray) \blank & *(lightgray) \blank & *(lightgray) \blank & 2 &\\
 *(lightgray) \blank & *(lightgray) \blank &  & 3 & *(red) 4\\
*(red) 1 & *(red) 2 & *(red) 3\\
& & 4
\end{ytableau}
\]
\end{description}
\caption{Examples of doubling. In (b), we have marked in {\color{red} red} the strictly increasing subsequence of length at least $k$ from which we derive a contradiction in~\Cref{prop:B_mult-free}.}
\label{fig:double}
\end{figure}

The operation of doubling allows us to relate weak \K-Knuth equivalence to \K-Knuth equivalence.

\begin{proposition}[{\cite[Proposition 7.1]{buch2014k}}]
Let $\T$ and $\U$ be shifted Ferrers tableaux. If $\rhow(\T)$ and $\rhow(\U)$ are weakly \K-Knuth equivalent, then $\rhow(\T^D)$ and $\rhow(\U^D)$ are \K-Knuth equivalent.
\label{prop:doubling}
\end{proposition}

\section{A Bijective Framework for Doppelg\"angers}
\label{sec:results_and_conjectures}

Our bijective framework for doppelg\"angers is based on the equivalence of a product in $K(\GP)$ and a bijection between a set and a multiset of increasing tableaux.  In this section we introduce this framework; \Cref{sec:coincidental,sec:embedding,sec:applications} will then be devoted to specializing~\Cref{thm:main1} to prove~\Cref{thm:main_thm1}. 

%
\begin{theorem}
Suppose $\GP$ is a minuscule variety and $u \leq v \in W^{\mathsf P}$. If
\begin{equation}
[\mathcal{O}_{X_{\check{v}}}] \cdot [\mathcal{O}_{X_u}] = \sum_{\check{x}} C_{\check{v}, u}^{\check{x}} [\mathcal{O}_{X_{\check{x}}}]
,
\label{eq:sym}
\end{equation}
then $C_{\check{v}, u}^{\check{x}} = C_{x,u}^v$ and, for any $\kk \in \mathbb{Z}_{>0}$, $\rect_{\T^{\rm min}_{\sf u}}$ gives a bijection
\[\IT^{[\kk]}({\sf v}/{\sf u}) \simeq \bigsqcup_{{\sf x} \subseteq {\sf v}} \left( \R^{\sf v}_{{\sf x}, {\sf u}} \times \IT^{[\kk]}({\sf x}) \right).\]
\label{thm:main1}
\label{thm:gen_main_thm_Ring}
\end{theorem}

The idea behind the proof of~\Cref{thm:main1} is illustrated in~\Cref{fig:bijections}.

\begin{figure}[htbp]
\[\raisebox{-0.5\height}{\begin{tikzpicture}
\draw[clip, preaction={draw, ultra thick, double distance=0pt}]
  plot[smooth cycle]
  coordinates{
   (0,0) (-1,2) (0,4) (1,2)
  };
\draw[thick, draw, fill=blue]  plot [smooth cycle] coordinates {(0,0) (-1.5,2) (-1,2.1) (-.7,2.2) (-.4,2.1) (.7,1.9) (1,1.3)};
\draw[thick, draw, fill=red]  plot [smooth cycle] coordinates {(0,0) (-1,1.3) (-.5,1) (.3,1.4) (.7,1) (1,.5)};
\node at (0,.6) {${\sf x}$};
\node at (0,1.7) {${\sf U}$};
\end{tikzpicture}} 
\xleftrightarrow{{\sf Kinfusion}}
\raisebox{-0.5\height}{\begin{tikzpicture}
\draw[clip, preaction={draw, ultra thick, double distance=0pt}]
  plot[smooth cycle]
  coordinates{
   (0,0) (-1,2) (0,4) (1,2)
  };
\draw[thick, draw, fill=red]  plot [smooth cycle] coordinates {(0,0) (-1.5,2) (-1,2.1) (-.7,2.2) (-.4,2.1) (.7,1.9) (1,1.3)};
\draw[thick, draw, fill=blue]  plot [smooth cycle] coordinates {(0,0) (-1,1.5) (-.5,1.3) (.3,1) (.7,.6) (1,.7)};
\draw[thick, draw]  plot [smooth] coordinates {(-1,1.3) (-.5,1) (.3,1.4) (.7,1) (1,.5)};
\node at (0,.6) {${\sf T}_{\sf u}^{\rm min}$};
\node at (0,1.7) {${\sf v} /{\sf u}$};
\end{tikzpicture}}
\]
\caption{An illustration of the idea (and notation) behind the bijections of~\Cref{thm:main1}.  To conclude~\Cref{thm:main_thm1} from~\Cref{thm:main1}, we must show that $\x=\chi(\Phi^+_{\Y})$ and that $\U$ is unique.}
\label{fig:bijections}
\end{figure}
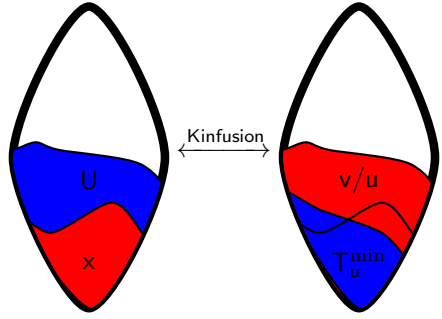

\begin{proof}
We first note that $C_{\check{v}, u}^{\check{x}} = C_{x, u}^{v}$ by the $\mathfrak{S}_3$-symmetry of \K-theoretic Littlewood-Richardson coefficients for minuscule varieties (cf.~\cite{knutson2010KT}).

Consider $\T' \in \IT^{[\kk]}({\sf v}/{\sf u})$. We have \[\Kinf(\T^{\rm min}_{\sf u},\T') = (\T,\U),\] for some ${\sf x} \subseteq {\sf v}$, some 
$\T \in \IT^{[\kk]}({\sf x})$ and some $\U \in \IT^{[\kk]}({\sf v}/{\sf x})$. Since $\rect_{\T}({\sf U})={\sf T}^{\rm min}_{\sf u}$, we have $\U \in \R^{\sf v}_{{\sf x}, {\sf u}}$. 
But infusion is an involution by~\Cref{thm:infusion_involution}, so that $\rect_{\T^{\rm min}_{\sf u}}$ is an injection \[\IT^{[\kk]}({\sf v}/{\sf u}) \hookrightarrow \bigsqcup_{{\sf x} \subseteq {\sf v}} \left( \R^{\sf v}_{{\sf x}, {\sf u}} \times \IT^{[\kk]}({\sf x}) \right).\]

Conversely, for ${\sf x} \subseteq {\sf v}$, $\U \in \R^{\sf v}_{{\sf x}, {\sf u}}$ and $\T \in \IT^{[\kk]}({\sf x})$, we have
\[\Kinf(\T,\U) = (\T^{\rm min}_{\sf u},\T') \in \IT({\sf u}) \times \IT^{[\kk]}({\sf v}/{\sf u}),\] for a unique $\T'$.
This shows that $\rect_{\T^{\rm min}_{\sf u}}$ is also surjective, and hence bijective.
\end{proof}

In the special case when the product in $K(\GP)$ is multiplicity-free, so that each $\R^{\sf v}_{{\sf x}, {\sf u}}$ only has one element, \Cref{thm:main1} specializes to a bijective statement involving only sets of increasing tableaux.  This is illustrated in~\Cref{ex:ex3} and~\Cref{fig:bijections}.

\begin{theorem} Fix $\kk \in \mathbb{Z}_{>0}$ a positive integer and ${\sf u} \subseteq {\sf v}$ order ideals in a minuscule poset. Suppose that the product $[\mathcal{O}_{X_{\check{v}}}] \cdot [\mathcal{O}_{X_u}]$ is multiplicity-free or equivalently that $|\R^{\sf v}_{{\sf x}, {\sf u}}| \leq 1$ for every ${\sf x} \subseteq {\sf v}$.  Then $\rect_{\T^{\rm min}_{\sf u}}$ gives a bijection 
\[\IT^{[\kk]}({\sf v}/{\sf u}) \simeq \bigsqcup_{{\sf x} : C^{\sf v}_{{\sf x}, {\sf u}} \neq 0}  \IT^{[\kk]}({\sf x}).\]

Let $\p = \kk - \hgt({\sf v}/{\sf u})$.  If all maximal chains of ${\sf v}/{\sf u}$ are of equal length, and the same is true for each ${\sf x}$ with $C^{\sf v}_{{\sf x}, {\sf u}} \neq 0$, then there is a bijection 
\[\PP^{[\p]}({\sf v}/{\sf u}) \simeq \bigsqcup_{{\sf x} :  C^{\sf v}_{{\sf x}, {\sf u}} \neq 0} \PP^{[\p + \hgt({\sf v}/{\sf u}) - \hgt({\sf x})]}({\sf x}).\]
\label{thm:gen_main_thm}
\end{theorem}
\begin{proof}
The first statement is immediate from \Cref{thm:main1}, and the second statement follows from the first by \Cref{prop:bij_PP_IT}.
\end{proof}

Specializing further to the case when the product is multiplicity-free with a single term, we obtain the following geometric corollary.
\begin{corollary}
In a minuscule variety $\GP$, if the Richardson variety $X_u^v$ is rationally equivalent to the opposite Schubert variety $X^x$ (or, equivalently, if both represent the same class in the Chow ring $A^\star(\GP)$), then there is a bijection \[\IT^{[\kk]}({\sf v}/{\sf u}) \simeq \IT^{[\kk]}({\sf x}).\]
\end{corollary}
\begin{proof}
In this case, the product $\sigma_{\check{v}} \cdot \sigma_u$ expands in the 
Schubert basis of $H^\star(\GP)$ as the single term $\sigma_{\check{x}}$. Hence, by \Cref{rem:notnecessarily}, we also have 
\[
[\mathcal{O}_{X_{\check{v}}}] \cdot [\mathcal{O}_{X_u}] =  [\mathcal{O}_{X_{\check{x}}}].
\]
Now, conclude by \Cref{thm:gen_main_thm}.
\end{proof}

We finally may restrict~\Cref{thm:gen_main_thm} to the cohomology ring $H^\star(\GP)$, which gives bijections of standard tableaux.

\begin{corollary}
	Fix $\kk \in \mathbb{Z}_{>0}$ a positive integer and ${\sf u} \subseteq {\sf v}$ order ideals in a minuscule poset. Suppose $|\R^{\sf v}_{{\sf x}, {\sf u}}| \leq 1$ for every ${\sf x} \subseteq {\sf v}$.  Then $\rect_{\T^{\rm min}_{\sf u}}$ restricts to a bijection 
\[\SYT({\sf v}/{\sf u}) \simeq \bigsqcup_{{\sf x} : C^{\sf v}_{{\sf x}, {\sf u}} \neq 0}  \SYT({\sf x}).\]
\end{corollary}
\begin{proof}
	The bijection in~\Cref{thm:gen_main_thm} obviously restricts to the set of standard tableaux.
\end{proof}

\section{Coincidental Root Posets}
\label{sec:coincidental}

In this section, we develop the posets $\Phi^+_Y$ appearing in \Cref{thm:main_thm1}.

\begin{definition}
We call the Coxeter-Cartan types $A_n$, $B_n$, $H_3$, and $I_2(m)$ the \defn{coincidental types}.
\end{definition}

A.~Miller observed that these are exactly those types for which the degrees $d_1<d_2<\cdots<d_n$ of the Coxeter group $W$ form an arithmetic sequence~\cite{miller2015foulkes}.  The coincidental types have many remarkable properties, and many enumerative questions are  ``more uniform'' when restricted from all Coxeter-Cartan types to just the coincidental types.  Such enumerative results include:

\begin{itemize}
    \item the number of $k$-dimensional faces of the generalized cluster complex~\cite{fomin2005generalized};
    \item the number of saturated chains of length $k$ in the noncrossing partition lattice~\cite{reading2008chains};
    \item the number of reduced words for $\wo$~\cite{stanley1984number,edelman1987balanced,haiman1992dual,williams13cataland};
    \item the number of multitriangulations~\cite{ceballos2014subword}; and
    \item the Coxeter-biCatalan numbers~\cite{barnardreading}.
\end{itemize}

\subsection{Crystallographic Root Posets}
Since the root system of type $B_n$ is crystallographic, its root poset is defined by~\Cref{eq:root_poset}.  Examples are given in~\Cref{fig:names}; $\Phi^+_{B_n}$ is a \defn{shifted double staircase} when drawn as a shifted Ferrers shape. 

More generally, using the conventions of~\Cref{fig:min_classification}, let \[b_{k,n}:=(s_1s_2\cdots s_{n-k})^{k} \in W(B_{n-k})\] for $n\geq 2k$.    The poset $\trap{k}{n}$ is \[\trap{k}{n}:=\Phi^+_{B_{n-k}} \cap {\sf b}_{k,n},\]
  the restriction of $\Phi^+_{B_{n-k}}$ to the roots that are inversions of the Weyl group element $b_{k,n}$.  As a special case, since $b_{n,2n}=\wo \in W(B_{n})$, we have $\Phi^+_{B_n}=\trap{n}{2n}$.  When drawn as a shifted Ferrers shape, $\trap{k}{n}$ may be described as the \defn{shifted trapezoid}, with corresponding strict integer partition $(n-1,n-3,\ldots,n-2k+1)_*$.

In general, for $Y \in \{ B_{k,n}, H_3, I_2(2n) \}$, we will write \begin{equation}\wo(\Y):=\begin{cases} \wo \in W(\Y) & \text{if } \Y \in \{H_3,I_2(m)\} \\ b_{k,n} \in W(B_{n-k}) & \text{if }\Y=B_{k,n} \end{cases}.\label{eq:long_element}\end{equation}

\subsection{Non-Crystallographic Root Posets}\label{sec:noncrystroot}
It remains to construct ``root posets'' in the noncrystallographic types $H_3$ and $I_2(m)$ for $m\neq 2,3,4,6$.  In crystallographic types, it is a remarkable fact (uniformly proven by B.~Kostant~\cite{kostant1959principal}) that the sizes of the ranks of $\Phi^+$ and the degrees $d_1,d_2,\ldots,d_n$ of $W$ form conjugate partitions under the identity~\cite[Theorem 3.20]{humphreys1992reflection}
\begin{equation}\left|\big\{ \alpha \in \Phi^+ : \hgt(\alpha)=i \big\}\right| = \left|\big\{ j : d_j > i \big\} \right|.\label{eq:heights}\end{equation}

The obvious application of~\Cref{eq:root_poset} does not yield a root poset satisfying this condition in the non-crystallographic types.  For example, if $\phi:=\frac{1+\sqrt{5}}{2}$, then in the basis of simple roots, \[\Phi^+_{I_2(5)}=\left\{(1,0),(0,1),\left(\phi,1\right),\left(1,\phi\right),\left(\phi,\phi\right)\right\},\] which would be ordered by~\Cref{eq:root_poset} to have Hasse diagram $\raisebox{-0.5\height}{\begin{tikzpicture}[scale=.2]
    \draw[thick] (0 cm,0) -- (0 cm,1.5) -- (1cm,2.8);
    \draw[thick] (2 cm,0) -- (2 cm,1.5) -- (1cm,2.8);
    \draw[thick] (0 cm,0) -- (2 cm,1.5);
    \draw[thick] (2 cm,0) -- (0 cm,1.5);
    \draw[ thick,solid,fill=white] (0cm,0) circle (.5cm);
    \draw[ thick,solid,fill=white] (2cm,0) circle (.5cm);
    \draw[ thick,solid,fill=white] (0cm,1.5) circle (.5cm);
    \draw[ thick,solid,fill=white] (2cm,1.5) circle (.5cm);
    \draw[ thick,solid,fill=white] (1cm,2.8) circle (.5cm);
  \end{tikzpicture}}$, so that it has two elements of rank one, two elements of rank two, and one element of rank three.    On the other hand, since the degrees of $I_2(5)$ are $2$ and $5$,~\Cref{eq:heights} predicts two elements of rank one, and one element for each rank greater than one.

On the basis of~\Cref{eq:heights} and a few other criteria from Coxeter-Catalan combinatorics, D.~Armstrong constructed
surrogate root posets in types $H_3$ and $I_2(m)$ with desirable behavior~\cite{armstrong2009generalized}. For more details, see~\cite[Section 3]{cuntz2015root} (which includes a construction of $\Phi^+_{H_3}$ using a folding argument).
\medskip

We will construct these root posets using the fully commutative theory reviewed in~\Cref{sec:minuscule}, and refer the reader to the noncrystallographic part of~\Cref{fig:min_classification} for the labeling conventions of the Coxeter-Dynkin diagram.  For convenience, we now work with reflections instead of roots.

\subsubsection{$I_2(m)$}
\label{sec:root_poset_dihedral}
In type $I_2(m)$, the root poset is a natural generalization of the root posets for the crystallographic
dihedral types $A_1 \times A_1$, $A_2$, $B_2$, and $G_2$.

The Coxeter group $W=W(I_2(m))$ has two generators, $s=s_1$ and $t=s_2$.  $W$ has a fully-commutative maximal parabolic quotient $W^J$, where $J=\{t\}$.  The longest element of $W^J$ has one reduced word: $\woJ=\underbrace{sts\cdots}_{m-1 \text{ letters}}.$  The heap for $\woJ$ is therefore a chain of length $m-1$, whose vertices are canonically labeled by the reflections coming from the corresponding letter of the word for $\woJ$: $s,sts,ststs,\ldots$.  This leaves only the reflection $t$ unspecified---but in order for the full poset to satisfy~\Cref{eq:heights}, we conclude that $sts$ covers $t$. See \Cref{fig:names} for an example.

\subsubsection{$H_3$}
\label{sec:h3}

We note that $W(I_2(5))$ is the maximal parabolic subgroup of $W=W(H_3)$ generated by $J=\{s_1,s_2\}$.  By the previous section, we therefore obtain the root poset of the parabolic subgroup $W(I_2(5))$, which ought to be the restriction of the full root poset of $H_3$ to that parabolic subgroup.  Now the parabolic quotient $W^J=W(H_3)/W(I_2(5))$ has a maximal element that is fully commutative (see the classification in~\Cref{fig:min_classification}). This fact allows us to canonically label the heap of \[\woJ=s_3s_2s_1s_2s_1s_3s_2s_1s_2s_3\] by the corresponding reflections.  In order to satisfy~\Cref{eq:heights}, we let $t \lessdot s_3ts_3$ for all reflections $t\in W_J$ such that $s_3ts_3 \neq t$.  This poset is illustrated in~\Cref{fig:h3rootposet}, and has been shown in~\cite{cuntz2015root} to be the unique poset satisfying a list of six natural conditions.

\begin{figure}
\includegraphics[height=2in]{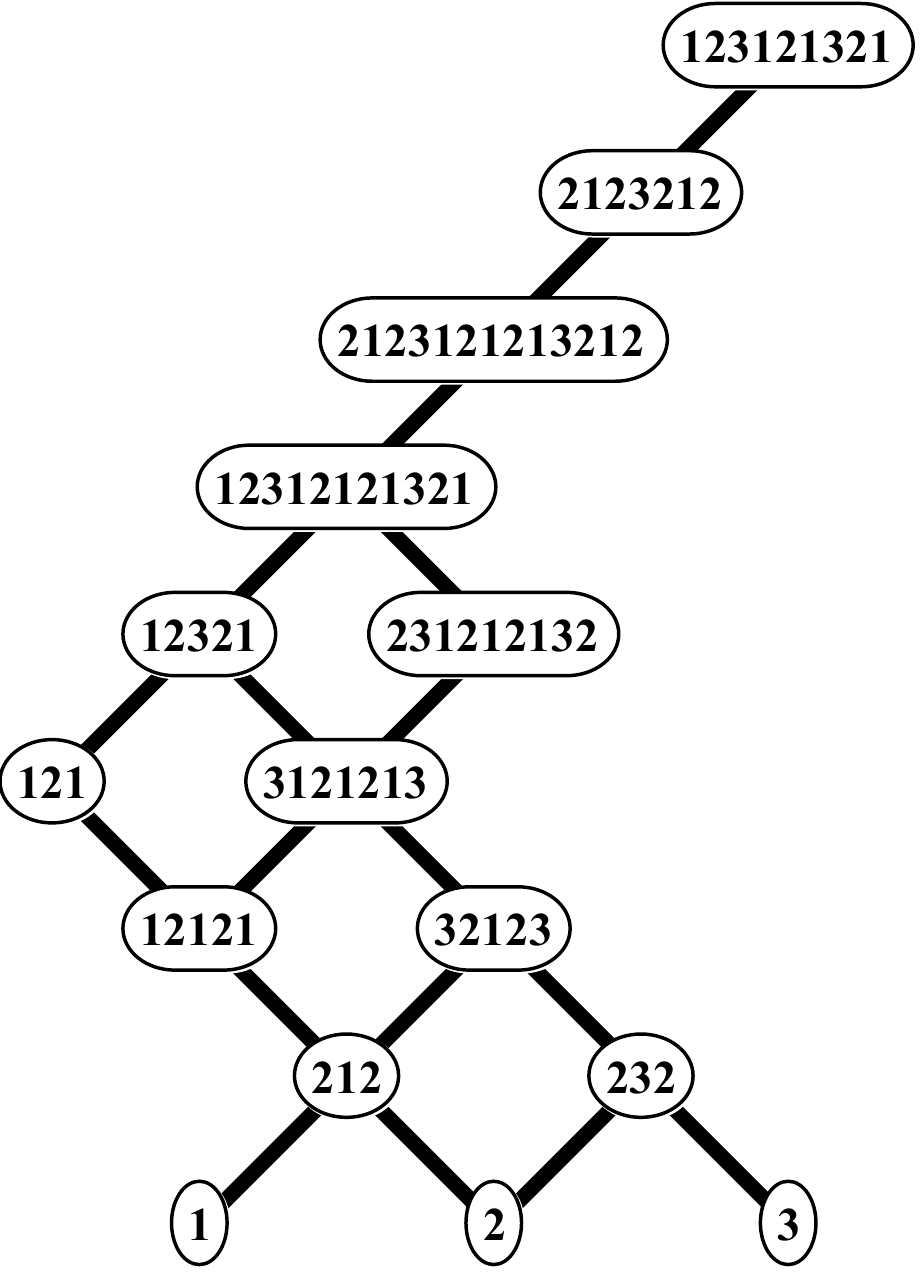}
\caption{The vertices of the root poset of type $H_3$ labeled using the method of~\Cref{sec:h3}.}
\label{fig:h3rootposet}
\end{figure}

\section{Poset Embeddings}
\label{sec:embedding}

Fix \[(\X,\Y,\Z) \in \left\{\begin{array}{l}\big(\Gr(k,n),B_{k,n},\OG(n,2n)\big),\\ \big(\OG(6,12),H_3,\mathsf{G}_\omega(\mathbb{O}^3,\mathbb{O}^6)\big),\\ \big(\QQ^{2n},I_2(2n),\QQ^{4n-2}\big) \end{array}\right\}\] a triple from~\Cref{fig:triples}. We will refer to such a triple by its label (B), (H), or (I), as in \Cref{fig:triples}.  Here again, $X$ and $Z$ stand for minuscule varieties, while $Y$ is a bookkeeping device specifying a Cartan type. Following~\cite{thomas2009combinatorial}, we formalize~\Cref{fig:coincidental_top_half_minuscules} by embedding the doppelg\"anger posets $\Lambda_{\X}$ and $\Phi^+_{\Y}$ into the ambient minuscule poset $\Lambda_{\Z}$.  That is, we explicitly characterize
    \[\vv/\uu := \Theta(\Lambda_{\X}) \subseteq \Lambda_{\Z} \hspace{3em} \text{ and } \hspace{3em} \w := \chi(\Phi^+_{\Y})  \subseteq \Lambda_{\Z}.\]

\subsection{$\X$ in $\Z$: Embedding Minuscule Varieties}
\label{sec:embedding_geo}
A minuscule flag variety is specified by a Cartan type and a minuscule weight, as in~\Cref{fig:min_classification}.  For $\X$ a minuscule flag variety, let $\mathrm{Cart}(\X)$ be the corresponding Cartan type, write $W(\X)$ for the Weyl group $W(\mathrm{Cart}(\X))$, and let $W^\X$ be the corresponding maximal parabolic quotient of $W(\X)$.

\begin{figure}[htbp]
\begin{center}
\[\begin{array}{ccc}
  \begin{tikzpicture}[scale=.3]
    \draw[thick] (0 cm,-1) -- (2 cm,0);
    \draw[thick] (0 cm,1) -- (2 cm,0);
    \draw[thick] (2 cm,0) -- (4 cm,0);
    \draw[dotted,thick] (4 cm,0) -- (8 cm,0);
    \draw[thick] (8 cm,0) -- (10 cm,0);
    \draw[thick,solid,fill=white] (0cm,-1) circle (.6cm) node  {$1$};
    \draw[thick,solid,fill=gray] (0cm,1) circle (.6cm) node[white]  {$2$};
    \draw[thick,solid,fill=gray] (2cm,0) circle (.6cm) node[white]  {$3$};
    \draw[thick,solid,fill=gray] (4cm,0) circle (.6cm);
    \draw[thick,solid,fill=gray] (8cm,0) circle (.6cm);
    \draw[thick,solid,fill=gray] (10cm,0) circle (.6cm) node[white]  {$n$};
  \end{tikzpicture} & 
      \begin{tikzpicture}[scale=.3]
    \draw[thick] (0 cm,0) -- (10 cm,0);
    \draw[thick] (6 cm,0) -- (6 cm,2);
    \draw[thick,solid,fill=gray] (0cm,0) circle (.6cm) node[white]  {$1$};
    \draw[thick,solid,fill=gray] (2cm,0) circle (.6cm) node[white]  {$3$};
    \draw[thick,solid,fill=gray] (4cm,0) circle (.6cm) node[white]  {$4$};
    \draw[thick,solid,fill=gray] (6cm,0) circle (.6cm) node[white]  {$5$};
    \draw[thick,solid,fill=gray] (8cm,0) circle (.6cm) node[white]  {$6$};
    \draw[thick,solid,fill=white] (10cm,0) circle (.6cm) node  {$7$};
    \draw[thick,solid,fill=gray] (6cm,2) circle (.6cm) node[white]  {$2$};
  \end{tikzpicture} & 
    \begin{tikzpicture}[scale=.3]
    \draw[thick] (0 cm,-1) -- (2 cm,0);
    \draw[thick] (0 cm,1) -- (2 cm,0);
    \draw[dotted,thick] (2 cm,0) -- (5 cm,0);
    \draw[thick] (5 cm,0) -- (7 cm,0);
    \draw[dotted,thick] (7 cm,0) -- (10 cm,0);
    \draw[thick,solid,fill=gray] (0cm,-1) circle (.6cm) node[white]  {$1$};
    \draw[thick,solid,fill=gray] (0cm,1) circle (.6cm) node[white]  {$2$};
    \draw[thick,solid,fill=gray] (2cm,0) circle (.6cm) node[white]  {$3$};
    \draw[thick,solid,fill=gray] (5cm,0) circle (.6cm) node[white]  {$m$};
    \draw[thick,solid,fill=white] (7cm,0) circle (.6cm);
    \draw[thick,solid,fill=white] (10cm,0) circle (.6cm);
  \end{tikzpicture}\\
  \Theta_B:\Delta(A_{n-1}) \hookrightarrow \Delta(D_{n}) & \Theta_H: \Delta(D_{6}) \hookrightarrow \Delta(E_7) & \Theta_I: \Delta(D_{m}) \hookrightarrow \Delta(D_{2m-2})
\end{array}\]
  \end{center}
\caption{Embedding $\Delta(\mathrm{Cart}(\X))$ into $\Delta(\mathrm{Cart}(\Z))$.}
\label{fig:embed}
\end{figure}

\Cref{fig:embed} illustrates the following embeddings of Dynkin diagrams, specified as an injection of the simple roots of $\mathrm{Cart}(\X)$ into the simple roots of $\mathrm{Cart}(\Z)$:
\begin{alignat*}{5} &\Theta_B &&: &\Delta(A_{n-1}) &\hookrightarrow \Delta(D_{n}) \\ \vspace{10pt} &&&& \alpha_i &\mapsto \alpha_{i+1}, \\ &\Theta_H &&: &\Delta(D_{6}) &\hookrightarrow \Delta(E_7) \\ &&&& \substack{\alpha_1 \mapsto \alpha_6 \\ \alpha_6 \mapsto \alpha_1},& \substack{\alpha_2 \mapsto \alpha_2 \\ \alpha_4 \mapsto \alpha_4}, \substack{\alpha_3 \mapsto \alpha_5 \\ \alpha_5 \mapsto \alpha_3}, \\ \vspace{10pt} &\Theta_I &&: &\Delta(D_{m}) &\hookrightarrow \Delta(D_{2m-2}) \\ &&&& \alpha_i&\mapsto \alpha_i. \end{alignat*}
We drop the subscript on $\Theta$ when the interpretation is clear from context.  These embeddings extend by linearity to injections of the full root system \[\Theta: \Phi(\mathrm{Cart}(\X)) \hookrightarrow \Phi(\mathrm{Cart}(\Z))\] and (under the correspondence between roots and reflections) to injections of the associated Weyl groups \[\Theta: W(\X) \hookrightarrow W(\Z).\] (Our abuse of notation, in calling all of these maps $\Theta$, should cause no confusion in context.)

Following~\cite{thomas2009combinatorial}, the next proposition states that these embeddings of Dynkin diagrams actually induce embeddings of minuscule flag varieties.  Using the Bruhat cell decomposition of~\Cref{eq:bruhatdecomposition} and results of~\cite{thomas2009combinatorial}, it suffices to embed the parabolic Weyl group quotients in a sufficiently nice way.

\begin{proposition}[After~\cite{thomas2009combinatorial}] There is an embedding $\Theta: X \hookrightarrow Z$ of minuscule varieties
such that \[\Theta(W^\X) = \{u x : x \in W^\X\} \subseteq W^\Z,\] for some $u \in W^\Z$.
\label{prop:thomasyong}
\end{proposition}

\begin{proof}
We first characterize $u \in W^\Z$ in each of the three cases.
\begin{enumerate}
\item[(B)] $u$ is the longest element of the parabolic quotient $W(D_k)^{\langle 1 \rangle}$, explicitly identified in~\Cref{sec:typeDels}.
\item[(H)] $u:=s_7s_6s_5s_4s_3s_1$; and
\item[(I)] $u:=s_{2m}s_{2m-1}\cdots s_{m+2}$. 
\end{enumerate}

One checks case-by-case that if $x \in W^\X$, then $u x \in W^\Z$ by examining the corresponding heaps.   These routine but tedious checks prove the containments in the proposition, which establish the analogues of~\cite[Corollary 6.7 and Lemma 6.8]{thomas2009combinatorial} in these settings.  The equality now follows from~\cite[Proposition 6.1]{thomas2009combinatorial}.
\end{proof}

We deduce that there is an embedding of the corresponding minuscule posets.

\begin{corollary}
For \[(\X,\Z) \in \left\{\begin{array}{l}\big(\Gr(k,n),\OG(n,2n)\big),\\ \big(\OG(6,12),\mathsf{G}_\omega(\mathbb{O}^3,\mathbb{O}^6)\big),\\ \big(\QQ^{2n},\QQ^{4n-2}\big) \end{array}\right\}\]  from any row of~\Cref{fig:triples}, there are poset embeddings $\Theta: \Lambda_{\X} \hookrightarrow \Lambda_{\Z}$.
\end{corollary}

\begin{proof}
Let $v:=u\wo^\X \in W^\Z$, where $\wo^\X$ is the longest element of $W^X$.  Since $\wo^X$ is fully commutative, by~\Cref{thm:labelsofels} and~\Cref{prop:thomasyong}, $\Lambda_{\X}$ embeds in $\Lambda_{\Z}$ as the poset $\vv/\uu$.  
\end{proof}


\subsection{$\Y$ in $\Z$: Embedding Root Posets}
\label{sec:root_posets_embed}
We now embed the root posets of the coincidental types $\Phi^+_{\Y}$ into the ambient minuscule posets $\Lambda_{\Z}$.  We find the existence of these embeddings mysterious; unexpectedly, the element $w \in W(\Z)$ whose heap $\w$ coincides with $\Phi^+_{\Y}$ has the same number of reduced words as $\wo(\Y) \in W(\Y)$.

\begin{proposition}
For \[(\Y,\Z) \in \left\{\begin{array}{l}\big(B_{k,n},\OG(n,2n)\big),\\ \big(H_3,\mathsf{G}_\omega(\mathbb{O}^3,\mathbb{O}^6)\big),\\ \big(I_2(2n),\QQ^{4n-2}\big) \end{array}\right\}\]
 from any row of~\Cref{fig:triples}, there is a poset embedding $\chi:\Phi^+_{\Y} \hookrightarrow \Lambda_{\Z}$, so that $\chi(\Phi^+_{\Y})$ has straight shape $\w$ for some $w \in W(\Z)$, and such that \[\Red(\wo(\Y)) \simeq \Red(w),\] where $\wo(\Y)$ is the element defined in~\Cref{eq:long_element} with inversion set $\Phi^+_{\Y}$.
\label{prop:rootposetembeds}
\end{proposition}

\begin{proof}

We first characterize the elements $w \in W(\Z)$ in each of the three cases.

\begin{enumerate}
\item[$\mathrm{(B)}$] $w :=\prod_{j=1}^k \left( s_{1,2}(j) \prod_{i=3}^{n-2j+2} s_i\right)$, where $s_{1,2}(j)=\begin{cases} s_1 & \text{if } j \text{ is odd} \\  s_2 & \text{if } j \text{ is even}\end{cases}.$
    \item[$\mathrm{(H)}$] $w := s_1s_3s_4s_5s_6s_7s_2s_5s_6s_4s_5s_2s_3s_4s_1$.
    \item[$\mathrm{(I)}$] $w := \left(\prod_{j=3}^{2n} s_j \right)^{-1} (s_1s_2)$.
\end{enumerate}

The statement that $\Red(\wo(\Y)) \simeq \Red(w)$ follows  for (B) by~\cite{kraskiewcz1989reduced,haiman1992dual,billey2014coxeter} using the poset isomorphism between $\Lambda_{\OG(n,2n)}$ and $\Lambda_{\LG(n-1,2n-2)}$).  For (I) and (H), this is an easy but unenlightening check~\cite{williams13cataland}.
\end{proof}

\begin{remark}
	 Note that, in contrast to the coincidental root posets, the graphs underlying the root posets of types $D_n$ ($n > 5$), $E_6$, $E_7$, $E_8$, and $F_4$ are nonplanar, and hence cannot embed in any minuscule poset.
\end{remark}

\section{Applications: Doppelg\"angers}\label{sec:applications}

As in \Cref{sec:embedding}, fix \[(\X,\Y,\Z) \in \left\{\begin{array}{l}\big(\Gr(k,n),B_{k,n},\OG(n,2n)\big),\\ \big(\OG(6,12),H_3,\mathsf{G}_\omega(\mathbb{O}^3,\mathbb{O}^6)\big),\\ \big(\QQ^{2n},I_2(2n),\QQ^{4n-2}\big) \end{array}\right\}\] a triple from~\Cref{fig:triples}. We continue to refer to such a triple by its label (B), (H), or (I), as in \Cref{fig:triples}.
  We recall that $\Theta$ is the specified embedding of the doppelg\"anger minuscule poset $\Lambda_{X}$ in $\Lambda_Z$, $\chi$ is the specified embedding of the coincidental root poset $\Phi^+_{\Y}$, and that we denote the corresponding shapes inside the ambient minuscule poset $\Lambda_{Z}$ as
\[  {\sf v}/{\sf u} := \Theta(\Lambda_{X})  \hspace{3em} \text{ and } \hspace{3em} {\sf w} := \chi(\Phi^+_{\Y}).\]
In all three cases, we will deduce~\Cref{thm:main_thm1} from~\Cref{thm:gen_main_thm} by showing that $\R_{{\sf w}, {\sf u}}^{{\sf v}}$ has a unique element and that $\R_{{\sf x}, {\sf u}}^{\sf v} = \emptyset$ for ${\sf x} \neq {\sf w}$. 

\begin{remark}
Geometrically, this amounts to showing that $[\mathcal{O}_{X_{\check{v}}}] \cdot [\mathcal{O}_{X_u}] = [\mathcal{O}_{X_{\check{w}}}]$ in the \K-theory ring $K(Z)$. 
Indeed, by \Cref{rem:notnecessarily}, it would be enough to establish that $\sigma_{\check{v}} \cdot \sigma_u = \sigma_{\check{w}}$ in $H^\star(Z)$, or equivalently that $[X_u^v] = [X^w] \in A^\star(Z)$. So in fact, for example, \Cref{thm:main_thm1}(B) (providing R.~Proctor's missing bijection) essentially follows from combining the geometric fact of \Cref{rem:notnecessarily}, the easy combinatorial observation of \Cref{thm:gen_main_thm}, and M.~Haiman's bijection of standard tableaux from \Cref{rem:previous_work}.
\end{remark}

\subsection{(B): Rectangles and Shifted Trapezoids}
\label{sec:b_rect_and_trap}
\Cref{sec:coincidental,sec:minuscule,sec:embedding} identify the following shapes for $Y=\trap{k}{n}$ and $X=\Gr(k,n)$: 

\begin{align*}
{\sf w} &= (n-1,n-3,\ldots,n-2k+1)_* \text{ is a shifted trapezoid, and}\\
{\sf v} / {\sf u} &= (n-k,n-k,\ldots,n-k) \text{ is a }k \times (n-k) \text{ rectangle.  Then}\\
{\sf u} & = (k-1, k-2, \ldots, 1)_* \text{ is a shifted staircase and}\\
{\sf v} & = (n-1,n-2,\dots, n-k)_*.
\end{align*}

\begin{figure}[htbp]
\begin{center}
\ydiagram{5,5,5} \hspace{.15\linewidth} \ydiagram{7,1+5,2+3}
\end{center}
\caption{${\sf v} / {\sf u} = \Theta(\Lambda_{\Gr(k,n)})$ and ${\sf w} = \chi(\trap{k}{n})$ for $k=3$ and $n=8$.}
\label{fig:anm_bnm}
\end{figure}

\Cref{fig:anm_bnm} illustrates examples of $\vv/\uu$ and $\w$.  Write $a:=2k-3$ and let $\U$ be the increasing anti-straight Ferrers tableau of shape ${\sf v} / {\sf w}$ obtained by labeling each southwest-to-northeast diagonal of the staircase ${\sf v} / {\sf w}$
with consecutive increasing integers, where the bottom row is labeled with the odd numbers from $1$ to $a$.  \Cref{fig:double}(a) and \Cref{fig:U} illustrate examples of $\T_\uu^{\rm min}$ and $\U$.

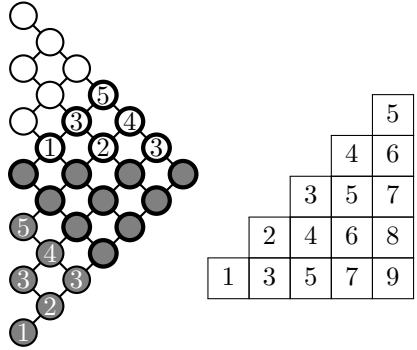
\begin{figure}[htbp]
\raisebox{-0.5\height}{\begin{tikzpicture}[scale=.35]
	\draw[thick] (0 cm,0) -- (6 cm,6);
	\draw[thick] (0 cm,2) -- (5 cm,7);
    \draw[thick] (0 cm,4) -- (4 cm,8);
	\draw[thick] (0 cm,6) -- (3 cm,9);
    \draw[thick] (0 cm,8) -- (2 cm,10);
    \draw[thick] (0 cm,10) -- (1 cm,11);
    \draw[thick] (4 cm,4) -- (0 cm,8);
	\draw[thick] (3 cm,3) -- (0 cm,6);
	\draw[thick] (2 cm,2) -- (0 cm,4);
	\draw[thick] (1 cm,1) -- (0 cm,2);
    \draw[thick] (5 cm,5) -- (0 cm,10);
    \draw[thick] (6 cm,6) -- (0 cm,12);
    \draw[thick,solid,fill=gray] (0cm,0) circle (.5cm)  node[white] {$1$}; 
    \draw[thick,solid,fill=gray] (0cm,2) circle (.5cm) node[white] {$3$};
    \draw[thick,solid,fill=gray] (0cm,4) circle (.5cm)  node[white] {$5$};
    \draw[ultra thick,solid,fill=gray] (0cm,6) circle (.5cm);
    \draw[thick,solid,fill=white] (0cm,8) circle (.5cm);
    \draw[thick,solid,fill=white] (0cm,10) circle (.5cm);
    \draw[thick,solid,fill=white] (0cm,12) circle (.5cm);
    \draw[thick,solid,fill=gray] (1cm,1) circle (.5cm) node[white] {$2$};
    \draw[thick,solid,fill=gray] (1cm,3) circle (.5cm) node[white] {$4$};
    \draw[ultra thick,solid,fill=gray] (1cm,5) circle (.5cm);
    \draw[ultra thick,solid,fill=white] (1cm,7) circle (.5cm) node {$1$};
    \draw[thick,solid,fill=white] (1cm,9) circle (.5cm) ;
    \draw[thick,solid,fill=white] (1cm,11) circle (.5cm) ;
    \draw[thick,solid,fill=gray] (2cm,2) circle (.5cm) node[white] {$3$};
    \draw[ultra thick,solid,fill=gray] (2cm,4) circle (.5cm);
    \draw[ultra thick,solid,fill=gray] (2cm,6) circle (.5cm);
    \draw[ultra thick,solid,fill=white] (2cm,8) circle (.5cm) node {$3$};
    \draw[thick,solid,fill=white] (2cm,10) circle (.5cm);
    \draw[ultra thick,solid,fill=gray] (3cm,3) circle (.5cm) ;
    \draw[ultra thick,solid,fill=gray] (3cm,5) circle (.5cm);
    \draw[ultra thick,solid,fill=white] (3cm,7) circle (.5cm)  node {$2$};
    \draw[ultra thick,solid,fill=white] (3cm,9) circle (.5cm)  node {$5$};
    \draw[ultra thick,solid,fill=gray] (4cm,4) circle (.5cm) ;
    \draw[ultra thick,solid,fill=gray] (4cm,6) circle (.5cm) ;
    \draw[ultra thick,solid,fill=white] (4cm,8) circle (.5cm) node {$4$};
    \draw[ultra thick,solid,fill=gray] (5cm,5) circle (.5cm) ;
    \draw[ultra thick,solid,fill=white] (5cm,7) circle (.5cm) node {$3$};
    \draw[ultra thick,solid,fill=gray] (6cm,6) circle (.5cm) ;
\end{tikzpicture}}
\raisebox{\height}{$\ytableaushort{\none \none \none \none 5, \none \none \none 46, \none \none 357,\none 2468, 13579}$}
\caption{On the left are the tableaux $\U$ (top, black numbers in white cirlces) and $\T_\uu^{\rm min}$ (bottom, white numbers in gray circles) for $k=3$ and $n=6$.  On the right is the increasing anti-straight Ferrers tableau $\U$ for $k=6$---inserting spaces for clarity, it has reading word $\rhow(\U) = 1\hspace{1ex}32\hspace{1ex}543\hspace{1ex}7654\hspace{1ex}98765$.}
\label{fig:U}
\end{figure}

We begin by characterizing a property of the reading word of any tableau that rectifies to $\T^{\rm min}_{\sf u}$.

\begin{lemma}
For $a=2k-3$, let $\pi \in \mathfrak{S}_{a}$ be the permutation with one-line notation \[[2,4,\ldots, a{+}1,1,3,\ldots,a].\]  Then any tableau $\widetilde \U$ that rectifies to $\T^{\rm min}_{\sf u}$ has a reading word $\rhow(\widetilde{\U})$ whose Demazure product is $\pi$.  In particular, since ${\sf len}(\pi)=\binom{k}{2}$, any such $\widetilde \U$ must have at least $\binom{k}{2}$ cells.
\label{lem:piword}
\end{lemma}
\begin{proof}
It is easy to see that the reading word $\rhow(\T^{\rm min}_{\sf u})$ is a reduced word of the permutation $\pi$.   Since $\rhow(\T^{\rm min}_{\sf u})$ is a reduced word, any weakly \K-Knuth equivalent word is at least as long (see the weak \K-Knuth relations in~\Cref{fig:knuth_rels}).   Furthermore, since every reduced word for $\pi$ begins with two commuting letters, the words that are weakly \K-Knuth equivalent to $\rhow(\T^{\rm min}_{\sf u})$ have Demazure product $\pi$.  We conclude the statement using \Cref{thm:jdt_and_knuth_equiv} on the interchangeability of \K-jeu-de-taquin equivalence of tableaux and \K-Knuth equivalence of their reading words.
\end{proof}

We now consider tableaux whose reading word can be $\pi$.  Recall that a permutation $\tau \in \mathfrak{S}_a$ is \defn{vexillary} if its one-line notation avoids the pattern $2143$; $\tau$ is fully commutative if and only if it avoids the pattern $321$; and $\tau$ is \defn{Grassmannian} if it has at most one descent.  A Grassmannian permutation is therefore both vexillary \emph{and} fully commutative.  In particular, the $\pi$ of~\Cref{lem:piword} is a Grassmannian permutation.

\begin{lemma}\label{lem:vex_anti}
For $a=2k-3$, let $\pi \in \mathfrak{S}_{a}$ be the permutation with one-line notation \[[2,4,\ldots, a{+}1,1,3,\ldots,a].\]  There is a unique increasing anti-straight Ferrers tableau $\T_\pi$ with $\rhow(\T_\pi) \in \Red(\pi)$.
\end{lemma}
\begin{proof}
We shall prove the statement more generally for any Grassmannian permutation $\tau$.  As suggested to us by V.~Reiner, it suffices to prove that there is a unique such (straight) Ferrers tableau, since if $\T_\tau'$ is the tableau obtained from $\T_\tau$ by reflecting across the antidiagonal and replacing $s_i \mapsto s_{a-i}$, then $\rhow(\T_\tau')$ is a reduced word for $\wo \tau \wo$ (both of the patterns 321 and 2143 are stable under conjugation by $\wo$).

Since $\tau$ is vexillary, it is well known that the reduced words of $\tau$ form a single Coxeter-Knuth equivalence class. Since $\tau$ is also fully-commutative, we have that, in the absence of braid moves, this Coxeter-Knuth class reduces to an ordinary Knuth equivalence class.  But any (semistandard) Knuth equivalence class contains a unique word that is the reading word of a Ferrers tableau~\cite[Section 2]{fulton1997young}, from which the lemma follows.
\end{proof}

Using the constraints provided by~\Cref{lem:piword} and~\Cref{lem:vex_anti}, we now show that $\U$ is the unique tableau whose shape is an order filter of $\vv$ and that rectifies to $\T^{\rm min}_{\sf u}$.

\begin{proposition}
\label{prop:B_mult-free}
$\R_{{\sf w},{\sf u}}^{\sf v} = \{ \U \} $ and $\R_{{\sf x},{\sf u}}^{\sf v} = \emptyset$ for ${\sf x} \neq {\sf w}$.
\end{proposition}

\begin{proof}
We first show that $\U \in \R_{{\sf w},{\sf u}}^{\sf v}$.  Since $\rhow(\T^{\rm min}_{\sf u}) = \rhow(\U)$ and $\T^{\rm min}_{\sf u}$ is a URT by \Cref{thm:urt}, $\rect_{\T}(\U) = \T^{\rm min}_{\sf u}$ for any $\T \in \IT({\sf w})$ by~\cite[Theorem 7.8]{buch2014k}. By definition, we then conclude $\U \in \R_{{\sf w},{\sf u}}^{\sf v}$.

Let $\widetilde \U \in \R_{{\sf x},{\sf u}}^{\sf v}$ for some ${\sf x}$.  We now argue that $\widetilde \U$ is necessarily of skew shape $\vv / \w$.  By~\Cref{prop:doubling,prop:kknuth}, since the shape of the doubling $(\T^{\rm min}_{\sf u})^D$ is a $(k-1) \times (k-1)$ square (see~\Cref{fig:double}(a)), the longest strictly increasing subsequence in $\rhow(\widetilde \U^D)$ is of length $k-1$.  We claim that this forces the $r$th column of $\widetilde \U$ (from the right) to have at most $k-r$ cells: if the $r$th column of $\widetilde \U$ has more than $k-r$ cells, then the $r$th row of $\widetilde \U^D$, along with the last $r-1$ entries in the bottom row of $\widetilde \U$, form a strictly increasing sequence of length at least $k$ in $\rhow(\widetilde \U^D)$.  Since there are at least ${k \choose 2}$ entries in $\widetilde \U$ by~\Cref{lem:piword}, the shape of $\widetilde \U$ is ${\sf v} / {\sf w}$.  This construction is illustrated in~\Cref{fig:double}(b). 
 
It remains to show that the fillings of $\widetilde \U$ and $\U$ are equal.   By~\Cref{lem:piword}, the Demazure product of $\rhow(\widetilde \U)$ is $\pi$. But since $\rhow(\widetilde \U)$ has length ${k \choose 2} = {\sf len}(\pi)$, $\rhow(\widetilde \U)$ is then a reduced word for $\pi$.  Since both $\rhow(\widetilde \U)$ and $\rhow(\U)$ are reduced words for the Grassmannian permutation $\pi$, and both $\widetilde \U$ and $\U$ have anti-straight shapes, \Cref{lem:vex_anti} implies  $\widetilde{\U}=\U$, as desired.
\end{proof}

By combining \Cref{thm:gen_main_thm} and \Cref{prop:B_mult-free}, we conclude~\Cref{thm:main_thm1}(B).

\subsection{(H)}

\begin{figure}
\raisebox{-0.5\height}{\begin{tikzpicture}[scale=.35]
	\draw[thick] (1 cm,1) -- (6 cm,6);
    \draw[thick] (3 cm,5) -- (5 cm,7);
	\draw[thick] (3 cm,7) -- (5 cm,9);
    \draw[thick] (2 cm,8) -- (6 cm,12);
    \draw[thick] (1 cm,9) -- (5 cm,13);
    \draw[thick] (3 cm,13) -- (4 cm,14);
    \draw[thick] (1 cm,17) -- (6 cm,12);
    \draw[thick] (3 cm,13) -- (5 cm,11);
    \draw[thick] (3cm,11) -- (5 cm,9);
    \draw[thick] (2cm,10) -- (6cm,6);
    \draw[thick] (1cm,9) -- (5cm,5);
    \draw[thick] (3cm,5) -- (4cm,4);
    \draw[thick,solid,fill=gray] (1cm,1) circle (.5cm) node[white] {$1$};
    \draw[thick,solid,fill=gray] (2cm,2) circle (.5cm)  node[white] {$2$};
    \draw[thick,solid,fill=gray] (3cm,3) circle (.5cm)  node[white] {$3$};
    \draw[thick,solid,fill=gray] (4cm,4) circle (.5cm)  node[white] {$4$};
    \draw[thick,solid,fill=gray] (5cm,5) circle (.5cm)  node[white] {$5$};
    \draw[thick,solid,fill=gray] (6cm,6) circle (.5cm)  node[white] {$6$};
    \draw[ultra thick,solid,fill=gray] (3cm,5) circle (.5cm) ;
    \draw[ultra thick,solid,fill=gray] (4cm,6) circle (.5cm) ;
    \draw[ultra thick,solid,fill=gray] (5cm,7) circle (.5cm) ;
    \draw[ultra thick,solid,fill=gray] (3cm,7) circle (.5cm) ;
    \draw[ultra thick,solid,fill=gray] (4cm,8) circle (.5cm) ;
    \draw[ultra thick,solid,fill=gray] (5cm,9) circle (.5cm) ;
    \draw[ultra thick,solid,fill=gray] (2cm,8) circle (.5cm) ;
    \draw[ultra thick,solid,fill=gray] (3cm,9) circle (.5cm) ;
    \draw[ultra thick,solid,fill=white] (4cm,10) circle (.5cm)  node {$1$};
    \draw[ultra thick,solid,fill=white] (5cm,11) circle (.5cm)  node {$4$};
    \draw[thick,solid,fill=white] (6cm,12) circle (.5cm) ;
    \draw[ultra thick,solid,fill=gray] (1cm,9) circle (.5cm) ;
    \draw[ultra thick,solid,fill=white] (2cm,10) circle (.5cm)  node {$2$};
    \draw[ultra thick,solid,fill=white] (3cm,11) circle (.5cm)  node {$3$};
    \draw[ultra thick,solid,fill=white] (4cm,12) circle (.5cm)  node {$5$};
    \draw[thick,solid,fill=white] (5cm,13) circle (.5cm) ;
    \draw[ultra thick,solid,fill=white] (3cm,13) circle (.5cm)  node {$6$};
    \draw[thick,solid,fill=white] (4cm,14) circle (.5cm) ;
    \draw[thick,solid,fill=white] (3cm,15) circle (.5cm) ;
    \draw[thick,solid,fill=white] (2cm,16) circle (.5cm) ;
    \draw[thick,solid,fill=white] (1cm,17) circle (.5cm) ;
\end{tikzpicture}}
\raisebox{-0.5\height}{\begin{tikzpicture}[scale=.35]
	\draw[thick] (-2 cm,-2) -- (4 cm,4);
    \draw[thick] (2 cm,4) -- (8 cm,10);
    \draw[thick] (2 cm,4) -- (3 cm,3);
    \draw[thick] (3 cm,5) -- (4 cm,4);
\draw[thick,solid,fill=gray] (-2cm,-2) circle (.5cm) node[white] {$1$};
    \draw[thick,solid,fill=gray] (-1cm,-1) circle (.5cm) node[white] {$2$};
   \draw[thick,solid,fill=gray] (0cm,0) circle (.5cm) node[white] {$3$};
    \draw[ultra thick,solid,fill=gray] (1cm,1) circle (.5cm);
    \draw[ultra thick,solid,fill=gray] (2cm,2) circle (.5cm) ;
    \draw[ultra thick,solid,fill=gray] (3cm,3) circle (.5cm) ;
    \draw[ultra thick,solid,fill=gray] (4cm,4) circle (.5cm) ;
    \draw[ultra thick,solid,fill=gray] (2cm,4) circle (.5cm) ;
    \draw[ultra thick,solid,fill=white] (3cm,5) circle (.5cm) node {$1$};
    \draw[ultra thick,solid,fill=white] (4cm,6) circle (.5cm) node {$2$};
    \draw[ultra thick,solid,fill=white] (5cm,7) circle (.5cm) node {$3$};
    \draw[thick,solid,fill=white] (6cm,8) circle (.5cm) ;
    \draw[thick,solid,fill=white] (7cm,9) circle (.5cm) ;
   \draw[thick,solid,fill=white] (8cm,10) circle (.5cm) ;
\end{tikzpicture}}
\caption{The shape $\vv$ is the set of all boxes that are either gray or thick-bordered; the shape $\w$ is the set of gray boxes; and the shape $\vv/\uu$ is the set of all thick-bordered boxes.  $\T_{\uu}^{\rm min}$ is the bottom tableau consisting of the gray boxes with white numbers.  The top tableau consisting of the white boxes with black numbers is the unique tableau $\U$ whose shape is an order filter in $\vv$ that rectifies to $\T_{\uu}^{\rm min}$.}
\label{fig:Uandtminue7}
\end{figure}
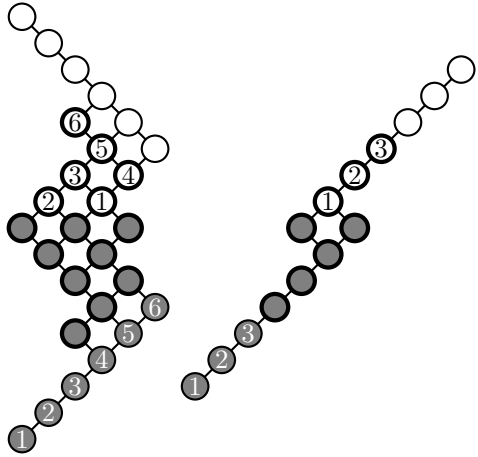

For the triple labeled (H), let the tableau $\U$ and its rectification be as illustrated on the left in \Cref{fig:Uandtminue7}. It is a straightforward but lengthy calculation to verify that $\R_{{\sf w}, {\sf u}}^{{\sf v}} = \{\U\}$ and that $\R_{{\sf x}, {\sf u}}^{\sf v} = \emptyset$ for ${\sf x} \neq {\sf w}$. We performed this calculation via computer, explicitly rectifying all applicable tableaux.  This case of \Cref{thm:main_thm1} now follows from~\Cref{thm:gen_main_thm}.

\subsection{(I)}
For the triple labeled (I), the shape $\vv / {\sf x}$ must be an order filter of size $\lfloor\frac{2n-2}{2}\rfloor$ in $\vv$. There is a unique such order filter in $\vv$---the shape $\vv/\w$---and because $\vv/\w$ is a chain, it has a unique filling that rectifies to $T_{\uu}^{\rm min}$.  By~\Cref{thm:gen_main_thm}, this proves the final case of~\Cref{thm:main_thm1}.  We refer the reader the the illustration on the right in \Cref{fig:Uandtminue7}.  

\subsection{Other Doppelg\"angers}

\cite[Figure 9]{thomas2009combinatorial} details the embeddings \[\Lambda_{\OG(5,10)} \hookrightarrow \Lambda_{\mathbb{OP}^2}\hookrightarrow \Lambda_{G_\omega(\mathbb{O}^3,\mathbb{O}^6)}.\] The three embeddings so specified give three pairs of additional (but completely trivial) doppelg\"angers.  

On the other hand, it is not the case that every embedding of a minuscule poset inside another gives a pair of doppelg\"angers.  For example, one can check that the embedding of $\Lambda_{\Gr(2,6)}$ inside $\Lambda_{\mathbb{OP}^2}$ indeed corresponds under~\Cref{thm:main1} to a multiplicity-free product in both ordinary cohomology and \K-theory; however, these products have three and five terms, respectively.


\section{Future Work}
\label{sec:future_work}

In~\cite{thomas2005multiplicity}, H.~Thomas and A.~Yong characterize all multiplicity-free products of Schubert classes in $\Gr(k,n)$ (extending work of J.~Stembridge~\cite{stembridge2001multiplicity}).  It is natural to wish to extend this to all minuscule flag varieties; except for the single remaining infinite family (up to isomorphism), this is a finite check.

\begin{problem}
    Classify all multiplicity-free products of cohomological Schubert classes in all minuscule flag varieties.
   \label{prob:multfrees}
\end{problem}

As pointed out~\Cref{rem:notnecessarily}, multiplicity-free products in cohomology are not necessarily multiplicity-free in \K-theory.  It would be interesting to apply A.~Knutson's~\Cref{thm:knutson_mult_free} to classify the latter products.  To our knowledge, this is open even in the Grassmannian case (although additional combinatorial tools are available in that case, e.g.~\cite{snider}).

\begin{problem}
     Classify all multiplicity-free products of \K-theoretic Schubert classes in all minuscule flag varieties.
     \label{prob:multfreek}
\end{problem}

Given any multiplicity-free product from~\Cref{prob:multfrees} or~\Cref{prob:multfreek}, \Cref{thm:gen_main_thm} then gives a combinatorial identity.  We have \emph{not} recorded in this paper all such identities---or even all such identities that lead to a pair of doppelg\"angers.

More generally, it is possible to derive poset identities (relating the number of standard or increasing fillings) by comparing Richardson varieties.  V.~Reiner, K.~Shaw, and S.~van Willigenburg have partial results in this direction for Grassmannians~\cite{reiner2007coincidences}.\footnote{We thank F.~Bergeron for pointing out that~\cite{reiner2007coincidences} was of the same spirit as the problems we have been considering.}

\begin{problem}
When do the (cohomological or \K-theoretic) classes of two Richardson varieties have the same expansion into Schubert classes?
\end{problem}

\subsection{Another coincidence}

There is one further poset identity involving the last coincidental type, relating $n$-staircases (the root poset of type $A_n$) and shifted $n$-staircases (the cominuscule poset of type $(C_n,1)$).  Although it does not fit directly into our framework, it seems closely related.


Recall that $\Lambda_{\LG(n,2n)}$ is the cominuscule poset of type $(C_n,1)$ (a shifted staircase of order $n$).  The \defn{diagonal} of $\Lambda_{\LG(n,2n)}$ is the set of its elements labeled by long roots of $\Phi^+_{C_n}$.  Let $\overline{\SYT}(\Lambda_{\LG(n,2n)})$ denote the product $[2^{n(n-1)/2}]\times \SYT(\Lambda_{\LG(n,2n)})$ (with elements represented as shifted standard tableaux with any set of off-diagonal entries barred), and let  $\overline{\PP}^{[2p]}(\Lambda_{\LG(n,2n)})$ be the subset of $\PP^{[2p]}(\Lambda_{\LG(n,2n)})$ with only even heights on the diagonal.  These definitions are illustrated in~\Cref{ex:ex2}.

\begin{example}\rm For $n=2$ and $p=1$, $\overline{\SYT}(\Lambda_{\LG(2,4)})$ and $\overline{\PP}^{[2]}(\Lambda_{\LG(2,4)})$ (first row), and (the duals of) $\SYT(\Phi^+_{A_2})$ and $\PP^{[1]}(\Phi^+_{A_2})$ (second row) are illustrated below.  The color white stands for height zero, gray for height one, and black for height two.  The modified fillings for $\Lambda_{\LG(n,2n)}$ have no barred element nor the color gray in their leftmost columns.
\vspace{0.5em}
\begin{center}\tiny
\begin{tikzpicture}[scale=.3]
	\draw[thick] (5 cm,2) -- (6 cm,3) -- (5cm,4);
    \draw[  thick,solid,fill=white] (5cm,2) circle (.5cm) node {1};
    \draw[ thick,solid,fill=white] (6cm,3) circle (.5cm) node {2};
    \draw[  thick,solid,fill=white] (5cm,4) circle (.5cm) node {3};
\end{tikzpicture}\hspace{10pt}
\begin{tikzpicture}[scale=.3]
	\draw[thick] (5 cm,2) -- (6 cm,3) -- (5cm,4);
    \draw[  thick,solid,fill=white] (5cm,2) circle (.5cm) node {1};
    \draw[ thick,solid,fill=white] (6cm,3) circle (.5cm) node {$\overline{2}$};
    \draw[  thick,solid,fill=white] (5cm,4) circle (.5cm) node {3};
\end{tikzpicture}
\hspace{5em}
\begin{tikzpicture}[scale=.3]
	\draw[thick] (5 cm,2) -- (6 cm,3) -- (5cm,4);
    \draw[  thick,solid,fill=white] (5cm,2) circle (.5cm);
    \draw[ thick,solid,fill=white] (6cm,3) circle (.5cm);
    \draw[  thick,solid,fill=white] (5cm,4) circle (.5cm);
\end{tikzpicture}\hspace{10pt}
\begin{tikzpicture}[scale=.3]
	\draw[thick] (5 cm,2) -- (6 cm,3) -- (5cm,4);
    \draw[  thick,solid,fill=white] (5cm,2) circle (.5cm);
    \draw[ thick,solid,fill=white] (6cm,3) circle (.5cm);
    \draw[  thick,solid,fill=black] (5cm,4) circle (.5cm);
\end{tikzpicture}\hspace{10pt}
\begin{tikzpicture}[scale=.3]
	\draw[thick] (5 cm,2) -- (6 cm,3) -- (5cm,4);
    \draw[  thick,solid,fill=white] (5cm,2) circle (.5cm);
    \draw[ thick,solid,fill=gray] (6cm,3) circle (.5cm);
    \draw[  thick,solid,fill=black] (5cm,4) circle (.5cm);
\end{tikzpicture}\hspace{10pt}
\begin{tikzpicture}[scale=.3]
	\draw[thick] (5 cm,2) -- (6 cm,3) -- (5cm,4);
    \draw[  thick,solid,fill=white] (5cm,2) circle (.5cm);
    \draw[ thick,solid,fill=black] (6cm,3) circle (.5cm);
    \draw[  thick,solid,fill=black] (5cm,4) circle (.5cm);
\end{tikzpicture}\hspace{10pt}
\begin{tikzpicture}[scale=.3]
	\draw[thick] (5 cm,2) -- (6 cm,3) -- (5cm,4);
    \draw[  thick,solid,fill=black] (5cm,2) circle (.5cm);
    \draw[ thick,solid,fill=black] (6cm,3) circle (.5cm);
    \draw[  thick,solid,fill=black] (5cm,4) circle (.5cm);
\end{tikzpicture}

\vspace{1em}
\begin{tikzpicture}[scale=.3]
    \draw[thick] (3 cm,4) -- (4 cm,3) -- (5 cm,4);
    \draw[thick,solid,fill=white] (3cm,4) circle (.5cm) node {3};
    \draw[thick,solid,fill=white] (5cm,4) circle (.5cm) node {2};
    \draw[thick,solid,fill=white] (4cm,3) circle (.5cm) node {1};
\end{tikzpicture}
\begin{tikzpicture}[scale=.3]
    \draw[thick] (3 cm,4) -- (4 cm,3) -- (5 cm,4);
    \draw[thick,solid,fill=white] (3cm,4) circle (.5cm) node {2};
    \draw[thick,solid,fill=white] (5cm,4) circle (.5cm) node {3};
    \draw[thick,solid,fill=white] (4cm,3) circle (.5cm) node {1};
\end{tikzpicture}
\hspace{5em}
\begin{tikzpicture}[scale=.3]
    \draw[thick] (3 cm,4) -- (4 cm,3) -- (5 cm,4);
    \draw[thick,solid,fill=white] (3cm,4) circle (.5cm);
    \draw[thick,solid,fill=white] (5cm,4) circle (.5cm);
    \draw[thick,solid,fill=white] (4cm,3) circle (.5cm);
\end{tikzpicture}
\begin{tikzpicture}[scale=.3]
    \draw[thick] (3 cm,4) -- (4 cm,3) -- (5 cm,4);
    \draw[thick,solid,fill=gray] (3cm,4) circle (.5cm);
    \draw[thick,solid,fill=white] (5cm,4) circle (.5cm);
    \draw[thick,solid,fill=white] (4cm,3) circle (.5cm);
\end{tikzpicture}
\begin{tikzpicture}[scale=.3]
    \draw[thick] (3 cm,4) -- (4 cm,3) -- (5 cm,4);
    \draw[thick,solid,fill=white] (3cm,4) circle (.5cm);
    \draw[thick,solid,fill=gray] (5cm,4) circle (.5cm);
    \draw[thick,solid,fill=white] (4cm,3) circle (.5cm);
\end{tikzpicture}
\begin{tikzpicture}[scale=.3]
    \draw[thick] (3 cm,4) -- (4 cm,3) -- (5 cm,4);
    \draw[thick,solid,fill=gray] (3cm,4) circle (.5cm);
    \draw[thick,solid,fill=gray] (5cm,4) circle (.5cm);
    \draw[thick,solid,fill=white] (4cm,3) circle (.5cm);
\end{tikzpicture}
\begin{tikzpicture}[scale=.3]
    \draw[thick] (3 cm,4) -- (4 cm,3) -- (5 cm,4);
    \draw[thick,solid,fill=gray] (3cm,4) circle (.5cm);
    \draw[thick,solid,fill=gray] (5cm,4) circle (.5cm);
    \draw[thick,solid,fill=gray] (4cm,3) circle (.5cm);
\end{tikzpicture}
\end{center}\normalsize
\label{ex:ex2}
\end{example}

The following theorem summarizes  work in~\cite{purbhoo2014marvellous} and~\cite{sheats1999symplectic}. The equinumerosity (AP) was first proved \emph{non-bijectively} by R.~Proctor~\cite{proctor1990new}.

\begin{theorem}
There is an explicit (symplectic jeu-de-taquin) bijection between
\begin{flalign*}
\mathrm{(AP)}	&&  \overline{\PP}^{[2p]}\left(\Lambda_{\LG(n,2n)}\right) \simeq \PP^{[\p]}\left(\Phi^+_{A_n}\right), && \text{(J.~Sheats~\cite{sheats1999symplectic})}
\end{flalign*}
and there is also an explicit (jeu-de-taquin) bijection between
\begin{flalign*}
\mathrm{(AS)}	&& \overline{\SYT}\left(\Lambda_{\LG(n,2n)}\right) \simeq \SYT\left(\Phi^+_{A_n}\right). && \text{(K.~Purbhoo~\cite{purbhoo2014marvellous})}
\end{flalign*}
	\label{thm:main_thm2}
\end{theorem}

We suspect that these poset identities are related to our theory above.  In~\cite{purbhoo2014marvellous}, K.~Purbhoo describes an embedding \[\Theta: \mathsf{LG}(n,2n) \hookrightarrow \mathsf{Gr}(n,2n+1).\]  As a corollary, one obtains a poset embedding $\Theta: \Lambda_{\LG(n,2n)} \hookrightarrow \Lambda_{\Gr(n,2n+1)}$. Furthermore, there is an embedding $\chi$ of $\Phi^+_{A_n}$ inside $\Lambda_{\mathsf{Gr}(n,2n)}$ given by the element $w := \prod_{j=1}^{n}\prod_{i=n-j+1}^{2n-2j+1} s_i \in W(A_{2n})$.   

  We can embed an element of $\SYT(\Phi^+_{A_n})$ on the standard alphabet $[\kk]$ as a standard tableau of shape $\Lambda_{\Gr(n,2n)}$ by reflecting across the diagonal and barring this reflection---extending the standard alphabet to $[\kk]\sqcup [\underline{\kk}]$.  K.~Purbhoo~\cite{purbhoo2014marvellous} then gives a bijection \[\fold: \SYT(\Phi^+_{A_n}) \to \overline{\SYT}(\Lambda_{\LG(n,2n)})\] using an operation similar to the infusion involution---but rather than completely slide one alphabet past another, he instead \defn{folds} the two alphabets together, so that the total ordering \[[\kk]\sqcup [\underline{\kk}] = 1<2<\cdots<\kk<\overline{\kk}<\cdots<\overline{2}<\overline{1}\] becomes the total ordering
\[[\overline{\kk}] \shuffle [\kk]:=\overline{1}<1<\overline{2}<2<\cdots<\overline{\kk}<\kk.\]

Writing $\overline{\IT}^{[\kk]}(\Lambda_{\LG(n,2n)})$ for the set of increasing tableaux of shape $\Lambda_{\LG(n,2n)}$ on the alphabet $[\overline{\kk}]\shuffle[\kk]$ with any set of off-diagonal entries barred, we have an easy bijection \[\overline{\PP}^{[2p]}(\Lambda_{\LG(n,2n)})\simeq \overline{\IT}^{[2p+n-1]}(\Lambda_{\LG(n,2n)}).\]
   The bijection in~\Cref{prop:bij_PP_IT} similarly gives \[\PP^{[\p]}\left(\Phi^+_{A_n}\right)\simeq \IT^{[p+n]}(\Phi^+_{A_n}),\] and we can perform the embedding of $\IT^{[\kk]}(\Phi^+_{A_n})$ into $\IT^{[\kk]\sqcup[\underline{\kk}]}(\Gr(n,2n))$ by reflecting and barring, as in the standard case.  It is perhaps surprising that the obvious increasing modification of K.~Purbhoo's folding is \emph{not} a bijection between $\overline{\IT}^{[2p+n-1]}(\Lambda_{\LG(n,2n)})$ and $\IT^{[p+n]}(\Phi^+_{A_n})$---folding may result in barred letters on the diagonal.

\begin{problem}\rm
    Prove R.~Proctor's identity (\Cref{thm:main_thm2}, (AP)) using a \K-theoretic jeu-de-taquin extension of K.~Purbhoo's bijection  (\Cref{thm:main_thm2}, (AS)).
\end{problem}

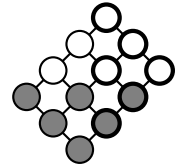
\begin{figure}[htbp]	
\[\begin{array}{ccccc} \Lambda_\X & \hookrightarrow & \Lambda_\Z & \xmapsto{\mathsf{fold}} & \Phi^+_\Y \\ \hline  \LG(n,2n)& & \mathsf{Gr}(n,2n) & & A_n  \end{array} \hspace{5em}
\raisebox{-0.5\height}{
\begin{tikzpicture}[scale=.35]
	\draw[thick] (3 cm,0) -- (6 cm,3);
    \draw[thick] (2 cm,1) -- (5 cm,4);
    \draw[thick] (1 cm,2) -- (4 cm,5);
    \draw[thick] (6 cm,3) -- (4 cm,5);
    \draw[thick] (5 cm,2) -- (3 cm,4);
    \draw[thick] (4 cm,1) -- (2 cm,3);
    \draw[thick] (3 cm,0) -- (1 cm,2);
    \draw[thick,solid,fill=gray] (3cm,0) circle (.5cm) ;
    \draw[thick,solid,fill=gray] (2cm,1) circle (.5cm) ;
    \draw[ultra thick,solid,fill=gray] (4cm,1) circle (.5cm) ;
    \draw[thick,solid,fill=gray] (1cm,2) circle (.5cm) ;
    \draw[thick,solid,fill=gray] (3cm,2) circle (.5cm) ;
    \draw[ultra  thick,solid,fill=gray] (5cm,2) circle (.5cm) ;
    \draw[thick,solid,fill=white] (2cm,3) circle (.5cm) ;
    \draw[ultra thick,solid,fill=white] (4cm,3) circle (.5cm) ;
    \draw[ultra thick,solid,fill=white] (6cm,3) circle (.5cm) ;
    \draw[thick,solid,fill=white] (3cm,4) circle (.5cm) ;
    \draw[ultra  thick,solid,fill=white] (5cm,4) circle (.5cm) ;
    \draw[ultra  thick,solid,fill=white] (4cm,5) circle (.5cm) ;
\end{tikzpicture}}\]
\caption{The analogues of~\Cref{fig:triples,fig:coincidental_top_half_minuscules} for an additional identity.   On the right is the poset $\Lambda_{\mathsf{Gr}(3,6)}$; the vertices with thick borders correspond to the embedding $\Theta(\Lambda_{\LG(3,6)})$, while the gray vertices represent $\chi(\Phi^+_{A_3})$.}
\end{figure}

\subsection{Reduced Words}

This relationship between linear extensions and reduced words has two different combinatorial proofs in types $A_n$ and $B_n$, which are related in~\cite{hamaker2013relating,billey2014coxeter}.  We refer the reader to~\cite{lascoux1995polynomes} for additional historical context.

\begin{itemize}
\item One proof is via modified RSK insertion algorithms due to P.~Edelman and C.~Greene in type $A_n$, and W.~Kraskiewicz in type $B_n$~\cite{edelman1987balanced,kraskiewcz1989reduced,lam1995b}---these insertions read and insert a reduced word for $\wo(A_n)$ or $\wo(B_{n,k})$ letter by letter to produce a standard tableau of shape $\w$ (a staircase or a shifted trapezoid).  The bijection is concluded using~\Cref{thm:linandred}, which canonically bijects $\SYT(\w)$ with $\Red(w)$.  The map backwards proceeds via \emph{promotion} on the standard tableau encoding a reduced word of $w$.

\item Another proof is via \emph{Little bumps} and \emph{signed Little bumps}~\cite{little2003combinatorial,billey2014coxeter}.  Thinking of a reduced word as a wiring diagram, these methods take a reduced word for $\wo(A_n)$ or $\wo(B_{n,k})$ and systematically eliminate all braid moves---by introducing additional strands---to obtain a reduced word for $w$.  Little bumps may be viewed as a combinatorialization of transition for Schubert polynomials (due in type $A_n$ to A.~Lascoux and M.-P.~Sch\"utzenberger)~\cite{billey1995schubert,billey1998transition}.
\end{itemize}

The Edelman-Greene promotion bijection between standard tableaux of root poset shape and reduced words for the longest element works in all coincidental types, which suggests the following open problem.

\begin{problem}\rm
    Uniformly develop a theory of insertion algorithms and Little bumps to explain the relation  $\Red(\wo(Y)) \simeq \Red(w)$ in the coincidental types.
\end{problem}

\begin{remark}\rm
The first step towards a theory of Little bumps in types $I_2(m)$ and $H_3$ would be the representation of reduced words using wiring diagrams.  Such representations exist, since both $W(I_2(m))$ and $W(H_3)\simeq {\sf Alt}_5 \times \mathbb{Z}/2\mathbb{Z}$ have (small) permutation representations coming from their actions on the parabolic subgroups identified in~\Cref{sec:noncrystroot}; the usual permutation representations of $W(A_n)$ and $W(B_n)$ may be obtained in a similar manner.  In analogy to the situation in type $A$ and $B$, using the embedings of~\Cref{sec:root_posets_embed} we expect that the Little bumps of type $H_3$ to take place in $W(E_7)$, while those of $I_2(m)$ should take place in a Weyl group of type $D$.

\end{remark}


\section*{Acknowledgements}
Some of the foundational ideas of this paper were developed when ZH, OP and NW attended the ``Dynamical Algebraic Combinatorics'' workshop
at the American Institute of Mathematics (AIM) in March 2015. 
The authors would like to thank AIM for hosting this workshop, as well as Jim Propp, Tom Roby, and Jessica Striker for their organizational efforts.

The authors are grateful for helpful and inspiring conversations with Bob Proctor, Vic Reiner, and Hugh Thomas and for careful comments from the anonymous referees that helped to improve exposition.

OP is grateful to Kevin Dilks and Jessica Striker for their collaboration on~\cite{DPS}, without which the current paper would not exist, and to Francois Ziegler for help understanding the Freudenthal variety.

\bibliographystyle{amsalpha}
\bibliography{IncreasingGenerality}

\providecommand{\bysame}{\leavevmode\hbox to3em{\hrulefill}\thinspace}
\providecommand{\MR}{\relax\ifhmode\unskip\space\fi MR }
\providecommand{\MRhref}[2]{%
  \href{http://www.ams.org/mathscinet-getitem?mr=#1}{#2}
}
\providecommand{\href}[2]{#2}
\begin{thebibliography}{BHRY14}

\bibitem[Arm09]{armstrong2009generalized}
D.~Armstrong, \emph{Generalized noncrossing partitions and combinatorics of
  {C}oxeter groups}, Mem. Amer. Math. Soc. \textbf{202} (2009), no.~949, x+159
  pp.

\bibitem[Bae01]{baez}
J.~Baez, \emph{The octonions}, Bull. Amer. Math. Soc. \textbf{39} (2001),
  no.~2, 145--205.

\bibitem[BH95]{billey1995schubert}
S.~Billey and M.~Haiman, \emph{Schubert polynomials for the classical groups},
  J. Amer. Math. Soc. \textbf{8} (1995), no.~2, 443--482.

\bibitem[BHRY14]{billey2014coxeter}
S.~Billey, Z.~Hamaker, A.~Roberts, and B.~Young, \emph{Coxeter-{K}nuth graphs
  and a signed {L}ittle map for type {B} reduced words}, Electron. J. Combin.
  \textbf{21} (2014), no.~4, Paper 4.6, 1--39.

\bibitem[Bil98]{billey1998transition}
S.~Billey, \emph{Transition equations for isotropic flag manifolds}, Discrete
  Math. \textbf{193} (1998), no.~1, 69--84.

\bibitem[Bir37]{birkhoff1937}
G.~Birkhoff, \emph{Rings of sets}, Duke Math. J. \textbf{3} (1937), no.~3,
  443--454.

\bibitem[BL00]{billey2000book}
S.~Billey and V.~Lakshmibai, \emph{Singular loci of {S}chubert varieties},
  Progr. Math., vol. 182, Birkh\"{a}user, 2000.

\bibitem[Bou75]{bourbaki1975}
N.~Bourbaki, \emph{Groupes et alg\`ebres de {L}ie, chapitres 7 et 8}, Hermann,
  Paris, 1975.

\bibitem[BR15]{barnardreading}
E.~Barnard and N.~Reading, \emph{Coxeter-bi{C}atalan combinatorics}, Discrete
  Math.\ Theor.\ Comput.\ Sci.\ Proc. \textbf{FPSAC'15} (2015), 157--168, FPSAC
  2015, Daejeon, South Korea.

\bibitem[Bri02]{brion2002positivity}
M.~Brion, \emph{Positivity in the {G}rothendieck group of complex flag
  varieties}, J. Algebra \textbf{258} (2002), 137--159.

\bibitem[BS16]{buch2014k}
A.~Buch and M.~Samuel, \emph{{$K$}-theory of minuscule varieties}, J. Reine
  Angew. Math. (Crelle's J.) \textbf{719} (2016), 133--171.

\bibitem[CLS14]{ceballos2014subword}
C.~Ceballos, J.-P. Labb{\'e}, and C.~Stump, \emph{Subword complexes, cluster
  complexes, and generalized multi-associahedra}, J. Algebraic Combin.
  \textbf{39} (2014), no.~1, 17--51.

\bibitem[CS15]{cuntz2015root}
M.~Cuntz and C.~Stump, \emph{On root posets for noncrystallographic root
  systems}, Math. Comp. \textbf{84} (2015), no.~291, 485--503.

\bibitem[CTY14]{clifford2014k}
E.~Clifford, H.~Thomas, and A.~Yong, \emph{{$K$}-theoretic {S}chubert calculus
  for ${OG}(n, 2n+ 1)$ and jeu de taquin for shifted increasing tableaux}, J.
  Reine Angew. Math. (Crelle's J.) \textbf{2014} (2014), no.~690, 51--63.

\bibitem[DPS17]{DPS}
K.~Dilks, O.~Pechenik, and J.~Striker, \emph{Resonance in orbits of plane
  partitions and increasing tableaux}, J. Combin. Theory Ser. A \textbf{148}
  (2017), 244--274.

\bibitem[EG87]{edelman1987balanced}
P.~Edelman and C.~Greene, \emph{Balanced tableaux}, Adv. Math. \textbf{63}
  (1987), no.~1, 42--99.

\bibitem[Eli15]{elizalde2015bijections}
S.~Elizalde, \emph{Bijections for pairs of non-crossing lattice paths and walks
  in the plane}, Eur. J. Combin. \textbf{49} (2015), 25--41.

\bibitem[FR05]{fomin2005generalized}
S.~Fomin and N.~Reading, \emph{Generalized cluster complexes and {C}oxeter
  combinatorics}, Int. Math. Res. Not. \textbf{2005} (2005), no.~44,
  2709--2757.

\bibitem[Ful97]{fulton1997young}
W.~Fulton, \emph{Young tableaux: with applications to representation theory and
  geometry}, vol.~35, Cambridge University Press, 1997.

\bibitem[Gre13]{green2013combinatorics}
R.~Green, \emph{Combinatorics of minuscule representations}, vol. 199,
  Cambridge University Press, 2013.

\bibitem[Hai92]{haiman1992dual}
M.~Haiman, \emph{Dual equivalence with applications, including a conjecture of
  {P}roctor}, Discrete Math. \textbf{99} (1992), no.~1, 79--113.

\bibitem[Hil82]{hiller1982geometry}
H.~Hiller, \emph{Geometry of {C}oxeter groups}, vol.~54, Pitman Publishing,
  1982.

\bibitem[Hum92]{humphreys1992reflection}
J.~Humphreys, \emph{Reflection groups and {C}oxeter groups}, vol.~29, Cambridge
  University Press, 1992.

\bibitem[HY14]{hamaker2013relating}
Z.~Hamaker and B.~Young, \emph{Relating {E}delman--{G}reene insertion to the
  {L}ittle map}, J. Algebraic Combin. \textbf{40} (2014), no.~3, 693--710.

\bibitem[Kan01]{kane}
R.~Kane, \emph{Reflection groups and invariant theory}, CMS Books in
  Mathematics/Ouvrages de Math\'ematiques de la SMC, vol.~5, Springer-Verlag,
  New York, 2001.

\bibitem[Kin75]{king1975branching}
R.C. King, \emph{Branching rules for classical {L}ie groups using tensor and
  spinor methods}, J. Phys. A \textbf{8} (1975), 429--449.

\bibitem[Knu09]{knutson2009frobenius}
A.~Knutson, \emph{{F}robenius splitting and {M}\"{o}bius inversion}, arXiv
  preprint arXiv:0902.1930 (2009).

\bibitem[Knu10]{knutson2010KT}
\bysame, \emph{Puzzles, positroid varieties, and equivariant {K}-theory of
  {G}rassmannians}, Arxiv preprint arXiv:1008.4302 (2010).

\bibitem[Kos59]{kostant1959principal}
B.~Kostant, \emph{The principal three-dimensional subgroup and the {B}etti
  numbers of a complex simple lie group}, Amer. J. Math. \textbf{81} (1959),
  no.~4, 973--1032.

\bibitem[Kra89]{kraskiewcz1989reduced}
W.~Kra{\'s}kiewicz, \emph{Reduced decompositions in hyperoctahedral groups}, C.
  R. Acad. Sci. Paris S\'{e}r. I Math. \textbf{309} (1989), no.~16, 903--907.

\bibitem[Lam95]{lam1995b}
T.-K. Lam, \emph{${B}$ and ${D}$ analogues of stable {S}chubert polynomials and
  related insertion algorithms}, Ph.D. thesis, Massachusetts Institute of
  Technology, 1995.

\bibitem[Las95]{lascoux1995polynomes}
A.~Lascoux, \emph{Polyn\^{o}mes de {S}chubert: une approche historique},
  Discrete Math. \textbf{139} (1995), no.~1, 303--317.

\bibitem[Les47]{lesieur1947problemes}
L.~Lesieur, \emph{Les problemes d'intersection sur une vari{\'e}t{\'e} de
  {G}rassmann}, C. R. Acad. Sci. Paris \textbf{225} (1947), 916--917.

\bibitem[Lit50]{littlewood1950theory}
D.E. Littlewood, \emph{The theory of group characters and matrix
  representations of groups}, vol. 357, American Mathematical Soc., 1950.

\bibitem[Lit03]{little2003combinatorial}
D.~Little, \emph{Combinatorial aspects of the {L}ascoux--{S}ch{\"u}tzenberger
  tree}, Adv. Math. \textbf{174} (2003), no.~2, 236--253.

\bibitem[Mil15]{miller2015foulkes}
A.~Miller, \emph{Foulkes characters for complex reflection groups}, Proc. Amer.
  Math. Soc. \textbf{143} (2015), no.~8, 3281--3293.

\bibitem[Mou33]{moufang}
R.~Moufang, \emph{Alternativk\"{o}rper und der {S}atz vom vollst\"{a}ndigen
  {V}ierseit}, Abh. Math. Sem. Hamburg \textbf{9} (1933), 207--222.

\bibitem[Pro83]{proctor1983shifted}
R.~Proctor, \emph{Shifted plane partitions of trapezoidal shape}, Proc. Amer.
  Math. Soc. \textbf{89} (1983), no.~3, 553--559.

\bibitem[Pro84]{proctor1984bruhat}
\bysame, \emph{Bruhat lattices, plane partition generating functions, and
  minuscule representations}, Eur. J. Combin. \textbf{5} (1984), no.~4,
  331--350.

\bibitem[Pro90]{proctor1990new}
\bysame, \emph{New symmetric plane partition identities from invariant theory
  work of {D}e {C}oncini and {P}rocesi}, Eur. J. Combin. \textbf{11} (1990),
  no.~3, 289--300.

\bibitem[Pur14]{purbhoo2014marvellous}
K.~Purbhoo, \emph{A marvellous embedding of the {L}agrangian {G}rassmannian},
  arXiv preprint arXiv:1403.0984 (2014).

\bibitem[Rea08]{reading2008chains}
N.~Reading, \emph{Chains in the noncrossing partition lattice}, SIAM J.
  Discrete Math. \textbf{22} (2008), no.~3, 875--886.

\bibitem[Rei97]{reiner1997non}
V.~Reiner, \emph{Non-crossing partitions for classical reflection groups},
  Discrete Math. \textbf{177} (1997), no.~1, 195--222.

\bibitem[RSvW07]{reiner2007coincidences}
V.~Reiner, K.~Shaw, and S.~van Willigenburg, \emph{Coincidences among skew
  {S}chur functions}, Adv. Math. \textbf{216} (2007), no.~1, 118--152.

\bibitem[She99]{sheats1999symplectic}
J.~Sheats, \emph{A symplectic jeu de taquin bijection between the tableaux of
  {K}ing and of {D}e {C}oncini}, Trans. Amer. Math. Soc. \textbf{351} (1999),
  no.~9, 3569--3607.

\bibitem[Sni09]{snider}
M.~Snider, \emph{A combinatorial approach to multiplicity-free {R}ichardson
  subvarieties of the {G}rassmannian}, arXiv preprint arXiv:0902.1931 (2009).

\bibitem[Sta72]{stanley1972ordered}
R.~Stanley, \emph{Ordered structures and partitions}, American Mathematical
  Soc., 1972.

\bibitem[Sta80]{stanley1980weyl}
Richard~P Stanley, \emph{Weyl groups, the hard {L}efschetz theorem, and the
  {S}perner property}, SIAM Journal on Algebraic Discrete Methods \textbf{1}
  (1980), no.~2, 168--184.

\bibitem[Sta84]{stanley1984number}
R.~Stanley, \emph{On the number of reduced decompositions of elements of
  {C}oxeter groups}, Eur. J. Combin. \textbf{5} (1984), no.~4, 359--372.

\bibitem[Sta86]{stanley1986symmetries}
\bysame, \emph{Symmetries of plane partitions}, J. Combin. Theory Ser. A
  \textbf{43} (1986), no.~1, 103--113.

\bibitem[Ste86]{stembridge1986trapezoidal}
J.~Stembridge, \emph{Trapezoidal chains and antichains}, Eur. J. Combin.
  \textbf{7} (1986), no.~4, 377--387.

\bibitem[Ste96]{stembridge1996fully}
\bysame, \emph{On the fully commutative elements of {C}oxeter groups}, J.
  Algebraic Combin. \textbf{5} (1996), no.~4, 353--385.

\bibitem[Ste98]{stembridge1998partial}
\bysame, \emph{The partial order of dominant weights}, Adv. Math. \textbf{136}
  (1998), no.~2, 340--364.

\bibitem[Ste01a]{stembridge2001minuscule}
\bysame, \emph{Minuscule elements of {W}eyl groups}, Journal of Algebra
  \textbf{235} (2001), no.~2, 722--743.

\bibitem[Ste01b]{stembridge2001multiplicity}
\bysame, \emph{Multiplicity-free products of {S}chur functions}, Ann. Comb.
  \textbf{5} (2001), 113--121.

\bibitem[Tit55]{tits}
J.~Tits, \emph{Sur certaines classes d'espaces homog\`{e}nes de groupes de
  {L}ie}, Palais des acad\'{e}mies, 1955.

\bibitem[TY05]{thomas2005multiplicity}
H.~Thomas and A.~Yong, \emph{Multiplicity-free {S}chubert calculus}, Canad.
  Math. Bull. \textbf{51} (2005), 171--186.

\bibitem[TY09a]{thomas2009combinatorial}
\bysame, \emph{A combinatorial rule for (co)minuscule {S}chubert calculus},
  Adv. Math. \textbf{222} (2009), no.~2, 596--620.

\bibitem[TY09b]{thomas2009jeu}
\bysame, \emph{A jeu de taquin theory for increasing tableaux, with
  applications to ${K}$-theoretic {S}chubert calculus}, Algebra Number Theory
  \textbf{3} (2009), no.~2, 121--148.

\bibitem[Vie86]{viennot1986heaps}
G.~Viennot, \emph{Heaps of pieces, {I}: {B}asic definitions and combinatorial
  lemmas}, Combinatoire {\'e}num{\'e}rative, Springer, 1986, pp.~321--350.

\bibitem[Wil13]{williams13cataland}
N.~Williams, \emph{{Cataland}}, Ph.D. thesis, University of Minnesota, 2013.

\end{thebibliography}

\end{document}